\documentclass[a4paper,11pt]{amsart}

\usepackage{indent}
\newcommand{\ew}{\epsilon}
\DeclareMathOperator{\lcm}{lcm}
\usepackage{tikz}
\usepackage{tkz-graph}
\usetikzlibrary{positioning}
\usetikzlibrary{arrows.meta}

\usepackage{mathrsfs,nccmath,amssymb,amsthm,mathtools}
\usepackage{mathabx}
\usepackage{enumitem}
\usepackage[all]{xy}
\xyoption{line}
\usepackage{color}
\usepackage[normalem]{ulem}
\renewcommand{\le}{\varleq}
\renewcommand{\ge}{\vargeq}

\newcommand\mL{L\kern-0.08cm\char39}
\newcommand{\myforall}{\text{ for all }}

\newcommand{\myand}{\text{ and }}

\newcommand{\seb}{\{\,}
\newcommand{\sen}{\,\}}

\newcommand{\getsby}[1]{\xleftarrow{#1}}

\newcommand{\Acal}{\mathcal{A}}

\newcommand{\Gcal}{\mathcal{G}}

\newcommand{\Ucal}{\mathcal{U}}
\newcommand{\Vcal}{\mathcal{V}}

\newcommand{\kuu}{\emptyset}
\newcommand{\nekuu}{\neq \kuu}
\newcommand{\iskuu}{= \kuu}

\newcommand{\fai}{\varphi}

\newcommand{\bl}{\boldsymbol{l}}

\newcommand{\bx}{\boldsymbol{x}}

\newcommand{\barf}{\bar{f}}

\newcommand{\barx}{\bar{x}}

\newcommand{\barB}{\bar{B}}

\newcommand{\barE}{\bar{E}}
\newcommand{\barF}{\bar{F}}

\newcommand{\barT}{\bar{T}}
\newcommand{\barU}{\bar{U}}
\newcommand{\barV}{\bar{V}}
\newcommand{\barW}{\bar{W}}
\newcommand{\barX}{\bar{X}}

\newcommand{\barfai}{\bar{\fai}}

\newcommand{\barrho}{\bar{\rho}}

\newcommand{\fg}[1]{{\widetilde{#1}}}
\newcommand{\fgG}{\fg{G}}
\newcommand{\fgV}{\fg{V}}
\newcommand{\fgE}{\fg{E}}
\newcommand{\fgW}{\fg{W}}
\newcommand{\fge}{\fg{e}}

\newcommand{\fgGcal}{\fg{\Gcal}}
\newcommand{\wg}[1]{\overline{#1}}
\newcommand{\wgG}{\wg{G}}
\newcommand{\wgV}{\wg{V}}
\newcommand{\wgE}{\wg{E}}
\newcommand{\wgW}{\wg{W}}
\newcommand{\wge}{\wg{e}}

\newcommand{\wgp}{\wg{p}}
\newcommand{\wgv}{\wg{v}}

\newcommand{\wgGcal}{\wg{\Gcal}}
\newcommand{\bai}[1]{\check{#1}}
\newcommand{\baiG}{\bai{G}}
\newcommand{\baiV}{\bai{V}}
\newcommand{\baiE}{\bai{E}}

\newcommand{\baie}{\bai{e}}

\newcommand{\baiu}{\bai{u}}
\newcommand{\baiv}{\bai{v}}

\newcommand{\baifai}{\bai{\fai}}
\newcommand{\baiGcal}{\bai{\Gcal}}
\newcommand{\invlim}{\varprojlim}

\newcommand{\tesgh}{edge-surjective graph homomorphism}

\newcommand{\pdirectional}{\raise0.05em\hbox{$+$}directional}
\newcommand{\pdirectionality}{\raise0.05em\hbox{$+$}directionality}
\newcommand{\pdirectionalitys}{\raise0.05em\hbox{$+$}directionality }
\newcommand{\pdirectionals}{\raise0.05em\hbox{$+$}directional }
\newcommand{\mdirectional}{\raise0.05em\hbox{$-$}directional}
\newcommand{\mdirectionality}{\raise0.05em\hbox{$-$}directionality}
\newcommand{\mdirectionalitys}{\raise0.05em\hbox{$-$}directionality }
\newcommand{\mdirectionals}{\raise0.05em\hbox{$-$}directional }
\newcommand{\bidirectional}{bidirectional}
\newcommand{\bidirectionals}{bidirectional }
\newcommand{\bidirectionality}{bidirectionality}
\newcommand{\bidirectionalitys}{bidirectionality }

\newcommand{\Z}{\mathbb{Z}}

\newcommand{\R}{\mathbb{R}}
\newcommand{\Nonne}{\mathbb{N}}
\newcommand{\Posint}{\mathbb{N}^+}

\newcommand{\bi}{\in \Z}

\newcommand{\bpi}{\ge 1}

\newcommand{\benonne}{\in \Nonne}
\newcommand{\bni}{\ge 0}

\newcommand{\diam}{{\rm diam}}

\newcommand{\dist}{{\rm dist}}

\newcommand{\limf}{\lim_f}
\newcommand{\liml}{\lim_{\hspace{0.2ex}l}}

\newcommand{\ep}{\varepsilon}

\newcommand{\sL}{\mathscr{L}}

\newcommand{\hf}{\hat{f}}

\newcommand{\hx}{\hat{x}}

\newcommand{\hX}{\hat{X}}

\newcommand{\ddf}{\ddot{f}}

\newcommand{\ddx}{\ddot{x}}

\newcommand{\ddX}{\ddot{X}}

\newcommand{\centb}{\begin{center}}
\newcommand{\centn}{\end{center}}
\newcommand{\enumb}{\begin{enumerate}}
\newcommand{\enumn}{\end{enumerate}}
\newcommand{\itemb}{\begin{itemize}}
\newcommand{\itemn}{\end{itemize}}
\numberwithin{equation}{section}
\usepackage{hyperref}
\usepackage[left=30mm,right=30mm,top=25mm,bottom=25mm]{geometry}
\usepackage{cleveref}
\setlist[enumerate,1]{label=(\alph*),ref=(\alph*)}
\setlist[enumerate,2]{label=(\arabic*),ref=(\alph{enumi}-\arabic{enumii})}
\setlist[enumerate,3]{label=(\Alph*),ref=(\roman{enumi}-\alph{enumii}-\Alph*)}
\setlist[enumerate,4]{label=(\arabic*),ref=(\roman{enumi}-\alph{enumii}-\Alph{enumiii}-\arabic*)}

\newlist{deepenum}{enumerate}{1}
\setlist[deepenum,1]{label=(\alph*$'$),ref=(\alph*$'$)}
\newlist{Lminusoneenum}{enumerate}{1}
\setlist[Lminusoneenum,1]{label=($2$:\alph*),ref=($2$:\alph*)}
\newlist{Ldeepenum}{enumerate}{1}
\setlist[Ldeepenum,1]{label=($3$:\alph*),ref=($3$:\alph*)}
\newlist{Lddeepenum}{enumerate}{1}
\setlist[Lddeepenum,1]{label=($4$:\alph*),ref=($4$:\alph*)}
\newtheorem{thm}{Theorem}[section]
\newtheorem{lem}[thm]{Lemma}
\newtheorem{prop}[thm]{Proposition}
\newtheorem{cor}[thm]{Corollary}

\theoremstyle{definition}
\newtheorem{defn}[thm]{Definition}

\theoremstyle{remark}

\newtheorem{nota}[thm]{Notation}
\newtheorem{rem}[thm]{Remark}

\crefname{sec}{\S}{\S\S}
\crefname{mainthm}{Theorem}{Theorems}
\crefname{thm}{Theorem}{Theorems}
\crefname{lem}{Lemma}{Lemmas}
\crefname{prop}{Proposition}{Propositions}
\crefname{cor}{Corollary}{Corollaries}
\crefname{defn}{Definition}{Definitions}
\crefname{conj}{Conjecture}{Conjectures}
\crefname{example}{Example}{Examples}
\crefname{nota}{Notation}{Notations}
\crefname{rem}{Remark}{Remarks}
\crefname{note}{Note}{Notes}
\crefname{case}{Case}{Cases}
\crefname{figure}{Figure}{Figures}
\crefname{section}{\S}{\S\S}
\crefname{enumi}{}{}
\crefname{enumii}{}{}
\crefname{equation}{}{}

\newcommand{\abs}[1]{\lvert#1\rvert}

\newcommand{\Vp}{V \setminus V_0}

\newcommand{\covrepa}[2]{#1_0 \getsby{#2_1} #1_1 \getsby{#2_2} #1_2 \getsby{#2_3} \dotsb}

\newcommand{\viatheargument}{via the argument in \cref{rem:principle}}

\begin{document}
  
\title[Bratteli--Vershik models and graph covering]{Bratteli--Vershik models and graph\\ covering models}

\author{Takashi Shimomura}

\address{Nagoya University of Economics, Uchikubo 61-1, Inuyama 484-8504, Japan}
\curraddr{}
\email{tkshimo@nagoya-ku.ac.jp}
\thanks{}



\date{\today}

\dedicatory{}

\commby{}

\begin{abstract}
Based on our previous graph covering method, we introduce weighted graph covering models
 and flexible graph covering models
 that are almost equivalent to the well-known Bratteli--Vershik models.
These models play important roles in showing that every invertible dynamical system on
 compact metrizable zero-dimensional space
 admits a non-trivial Bratteli--Vershik model and a basic set.
We can also obtain an analogue of Krieger's lemma
 for compact metrizable zero-dimensional systems.
The flexible graph covering models enable us to consider ``stationary''
 graph covering models, by which some portion of the substitution subshifts can be expressed. As an application, we show a way of constructing some class of transitive substitution subshifts.
\end{abstract}

\maketitle

\section{Introduction}\label{sec:introduction}
In connection with
 the $K$-theory for $C^*$-algebras, Herman, Putnam, and Skau \cite{HERMAN_1992OrdBratteliDiagDimGroupTopDyn} showed that an 
 essentially minimal invertible Cantor system admits
 the simple ordered Bratteli—Vershik model. 
Subsequently, there have been numerous studies on the Bratteli--Vershik models. 
For instance, 
 Medynets \cite{Medynets_2006CantorAperSysBratDiag} showed that
 aperiodic Cantor systems admit Bratteli--Vershik models,
 where the basic sets
 (see \cref{defn:basic-set})
 coincide with
 the sets of maximal (minimal) paths of
 the ordered Bratteli diagrams.
This discovery diversified the study of Bratteli--Vershik
 models whose basic sets are not single points.
Considering systems with periodic points,
 Downarowicz and Karpel
 \cite{DownarowiczKarpel_2016DynamicsInDimensionZeroASurvey}
 stated that an ordered Bratteli diagram is \textit{decisive} if the
 Vershik map can be prolonged in a unique way to a homeomorphism, and
 referred to zero-dimensional dynamical systems
 as \textit{Bratteli--Vershikizable} if they are 
 conjugate to the extended Vershik map of a decisive ordered Bratteli diagram.
Moreover, 
  Downarowicz and Karpel \cite{DownarowiczKarpel_2017DecisiveBratteliVershikmodels}
improved upon these results by showing that
 an invertible zero-dimensional system
 $(X,f)$ is Bratteli--Vershikizable if and only if the set of aperiodic points
 is dense or its closure misses one periodic orbit.

In this paper, we show that all invertible zero-dimensional systems
 admit non-trivial Bratteli--Vershik models.
We can derive \cref{cor:KriegerLemma},
 a version of Krieger's Marker Lemma, for all invertible zero-dimensional
 systems.
We use the graph covering model developed
 in \cite{Shimomura_2014SpecialHomeoApproxCantorSys}, and introduce
 two different types of graph covering models.
The first type is
the \textit{weighted graph covering model},
 in which each edge is weighted
 with a positive integer (length).
When we reduce the length (formally, but not substantially),
 we obtain a \textit{flexible graph covering model},
 which is the second type of model.
The relation between the Bratteli--Vershik models and new graph covering
 models is studied systematically.
Furthermore, these new models allow us to consider stationary graph covering models
 with links to the stationary Bratteli--Vershik models,
 and to make contact using substitution subshifts
 that may not be primitive.
We note that Durand, Host, and Skau
 \cite{DURAND_1999SubstDynSysBratteliDiagDimGroup}
 have shown that the primitive substitution subshifts are related to
 stationary, properly ordered Bratteli--Vershik systems.

A \textit{zero-dimensional system}
 denotes a pair $(X,f)$ of
 a compact metrizable totally disconnected space $X$ and a continuous 
 surjective map $f : X \to X$.
If $X$ is homeomorphic to the Cantor set,
 then the zero-dimensional system is said to be a \textit{Cantor system}.
We mainly consider the invertible case, i.e. $f$ is a homeomorphism.

We rename the graph covering developed in 
 \cite{Shimomura_2014SpecialHomeoApproxCantorSys}
 as the \textit{basic graph covering model}.
This is because we have introduced the two new types of graph covering models.
For each zero-dimensional system $(X,f)$,
 there exists a basic graph covering model that
 can be translated as 
 a trivial Bratteli--Vershik model, in which
 every point $x \in X$ might be maximal and also minimal.
By introducing weighted and flexible graph covering models,
 we can show that invertible zero-dimensional systems always
 have non-trivial Bratteli--Vershik models; hence, the set of maximal (minimal)
 paths forms a basic set.
Graph covering models have been investigated by a number of researchers.
One study, conducted by Bernardes and Darji 
 \cite{BernardesDarji_2012GraphTheoreticStructureOfMapsOfTheCantorSpace},
investigated co-meagre conjugacy classes of Cantor systems.
Their paper provides further references to some important works in this field.
Following our work \cite{Shimomura_2014SpecialHomeoApproxCantorSys},
 Fern{\'a}ndez, Good, and Puljiz
 \cite{FernandezGoodPuljiz_2017AlmostMinimalSystemsAndPeriodicityInHyper}
 constructed an almost totally minimal homeomorphism of the Cantor set,
 and Boro{\'{n}}ski, Kupka, and Oprocha \cite{BoronskiKupkaOprocha_2017AMixingCompletelyScrambledSystemExists}
 showed that there exists a completely scrambled topologically mixing system.
The latter result was then used to show that there exists
 a minimal weakly mixing invertible Cantor system $(X,f)$
 that can be embedded in $\R$ with vanishing derivative everywhere \cite{BoronskiKupkaOprocha_2018EdreisConjectureRevisited}. These studies indicate that the graph covering models themselves are interesting structures worthy of further research.

The remainder of this paper is organized as follows. In \cref{sec:basic-covering}, we define
 the basic graph covering models and state some basic facts
 that will be used later.
In \cref{sec:weighted-covering}, we introduce
 weighted and flexible graph covering models
 (see \cref{defn:weighted-graph,defn:weighted-and-flexible-graph-cover}).
Using the weighted graph covering model,
 we then define a set $\wgV_{\infty}$ (see \cref{nota:Vinfty})
and derive a necessary and sufficient condition for $\wgV_{\infty}$
 to be a basic set
 (see
 \cref{defn:basic-set,thm:closing-implies-basic-set-wg}).
Without explicitly assigning a positive integer to each edge,
 we obtain a rather abstract
 notion of flexible graph covering models
 (see \cref{defn:flexible-graph,defn:weighted-and-flexible-graph-cover}).
These notions of weighted and flexible graph covering models are
 precisely included in
 the notion of the Bratteli--Vershik models
 such that the ordered Bratteli diagrams can have multiple maximal (minimal) 
 paths.
In this relation,
 our graph covering models always produce continuous Vershik maps.
The main part of our main result is given in \cref{sec:weighted-covering}.
In particular, we derive a zero-dimensional version of Krieger's Marker Lemma
 \cite[(2.2) Lemma (Krieger)]{Boyle_1984LowerEntroFactOfSoficSys}
 (see \cref{cor:KriegerLemma}).

The relations with the Bratteli--Vershik models
 are described in \cref{sec:link}.
We believe that some additional regularity conditions
 should be included in the 
 Bratteli--Vershik models.
We introduce the \textit{closing property}
 (see
 \cref{defn:closing-BV}),
under which the set of minimal paths
 $E_{0,\infty,\min}$ is a basic set
 (see \cref{thm:closing-implies-basic-set-BV}).
With an arbitrary sequence of positive integers
 $\bl : l_1 < l_2 < \dotsb$,
 another regularity condition
 is obtained such that each tower of height less than or equal to $l_n$
 must have a periodic orbit, and
 the least period must match the height of the tower.
We refer to this condition as
 $\bl$-\textit{periodicity-regulated}
 (see \cref{defn:periodicity-regulated-BV}).
In general, this condition is strictly stronger than the closing property.
We show that an arbitrary invertible zero-dimensional system
 admits the Bratteli--Vershik model that satisfies
 the latter regularity condition (see \cref{thm:Bratteli-Vershik-for-all}).
We derive the next theorem from its counterparts for weighted graph covering models
 (see
 \cref{defn:closing-weighted-covering,defn:periodicity-regulated-wg,thm:main}).
\begin{thm}\label{thm:Bratteli-Vershik-for-all}
Let $(X,f)$ be a homeomorphic topological dynamical system, where $X$ is a
compact metrizable zero-dimensional set.
Let $\bl : l_1 < l_2 < \dotsb$ be a sequence of positive integers.
Then, $(X,f)$ admits an $\bl$-periodicity-regulated Bratteli--Vershik model.
\end{thm}

Let us explain this theorem.
In a Bratteli--Vershik model,
it is well known that each $V_n$ decomposes $X$ by finite towers
 that are linked to each vertex in $V_n$.
In this regard, the $\bl$-periodicity-regulated condition implies that,
 if the tower corresponding to $v \in V_n$ has height $l(v) \le l_n$,
 then there must exist a periodic orbit with a least period of $l(v)$.
If $(X,f)$ is positively transitive,
 then for an arbitrary sequence of small $\delta_n > 0$ ($n \bpi$),
 we may assume that
 the decomposition by $V_n$ includes
 a tower that is $\delta_n$-dense (see \cref{rem:delta-dense}).
In addition,
 suppose that there exists a closed and open set $U \subseteq X$ and a positive
 integer $n$ such that $\bigcup_{i = 0}^nf^i(U) = X$.
Then, \cref{rem:basic-set-is-contained} indicates that
 a basic set can be taken as a subset of $U$.

In \cite{DOWNAROWICZ_2008FiniteRankBratteliVershikDiagAreExpansive},
 Downarowicz and Maass defined the topological rank
 $K$ for every invertible Cantor minimal system.
Following the work of Bezuglyi, Kwiatkowski, and Medynets
 \cite{BEZUGLYI_2009AperioSubstSysBraDiag},
 our main result offers the possibility of defining the topological rank
 for every
 compact zero-dimensional homeomorphic system.

Bezuglyi, Kwiatkowski, and Yassawi \cite{BezuglyiKwiatkowskiYassawi_2014PerfectOrderingsOnFiniteRankBratteliDiagrams}
 considered cases in which, for finite-rank 
 Bratteli diagrams, the order
 of the diagram admits a homeomorphic Vershik map.
They refer to such orderings as \textit{perfect}.
One of their aims was to find a necessary and sufficient condition for an
 order to be perfect.
Further, Bezuglyi and Yassawi
 \cite{BezuglyiYassawi2017OrdersThatYieldHomeoOnBratteliDiagrams}
 considered this for infinite-rank Bratteli diagrams.
In these studies, it was essential to use sufficient telescopings.
The discussion of the continuity of Vershik maps has 
 been clarified by this research.
If we restrict our aim to stationary systems, then the argument becomes
 extremely simple.
Based on flexible graph coverings,
 we considered the link with some class of substitution subshifts.
In \cref{subsec:examples},
 we present a few examples of how stationary graph covering models
 can express some portion of substitution dynamical systems.
In \cref{subsec:transitive-substitution},
 we present a method of constructing some class of
 transitive substitution subshifts.
Through these examples and arguments, we intend to
 present concrete arguments that use weighted (flexible) graph covering models.
\section{Preliminaries}\label{sec:basic-covering}
Let $\Z$ be the set of all integers, $\Nonne$ be the set of all non-negative
 integers, and $\Posint$ be the set of all positive integers.
For integers $a < b$, the intervals are denoted by
 $[a,b] := \seb a, a+1, \dotsc,b \sen$, and so on.
Let $(X,f)$ be an invertible zero-dimensional system, i.e.
 $f$ is a homeomorphism.
Let $h$ be a positive integer
 and $U \subseteq X$ be a closed and open set.
If all $f^i(U)$ $(0 \le i < h)$ are mutually disjoint, then $\xi := \seb f^i(U) \mid 0 \le i < h \sen$ is called a \textit{tower}
 with \textit{base} $U$ and \textit{height} $h$.
In this case, we write $\barU := \bigcup \xi$, and say that $\barU = \bigcup_{0 \le i < h}f^i(U)$ is a \textit{tower}
 of height $h$ and  base $U$.
The notion of the tower plays a central role in our argument.
We have described the basic graph covering models for
 all zero-dimensional continuous surjections under the naming of
 the sequences of covers
 (see \cite[\S 3]{Shimomura_2014SpecialHomeoApproxCantorSys}).
In \cite{Shimomura_2014ergodic},
 we used the term `graph covering', or just `covering',
 whereas in this paper, with the notion of the tower, we develop the idea
 of a finite directed graph with multiple edges,
 and assign a positive integer, referred to as the `length', to each edge.
These two types of graph coverings are termed differently
 to avoid ambiguity
 (see \cref{nota:basic-graph-covering,defn:weighted-graph-covering}.)

In this section,
 we recall the process of constructing basic graph covering models
 for general zero-dimensional systems used in
 \cite{
 Shimomura_2014SpecialHomeoApproxCantorSys,
 Shimomura_2014ergodic}.
In \cref{sec:extended-graph-covering},
 we develop the new notion of \textit{weighted graph covering models}
 in which each edge is assigned a weight corresponding to its length.
Further, in \cref{sec:extended-graph-covering},
 without explicitly assigning the edge lengths, we define
 \textit{flexible graph covering models}.
\begin{nota}\label{nota:basic-graph-covering}
To avoid ambiguity among the three types of graph covering models,
 the original graph coverings
 are referred to here as \textit{basic graph covering models}.
Hereafter, 
 we do not use the ambiguous simple term `graph covering'.
Instead, 
 we deliberately use the three terms
 `basic (graph) covering model',
 `weighted (graph) covering model', and
 `flexible (graph) covering model'; the word `graph' may be omitted.
\end{nota}

A pair $\baiG = (\baiV,\baiE)$
 of a finite set $\baiV$
 and a relation $\baiE \subseteq \baiV \times \baiV$ on $\baiV$
 can be considered as a directed graph with vertices $\baiV$
 and an edge from $\baiu$ to $\baiv$ when $(\baiu,\baiv) \in \baiE$.
In this sense, we refer to the finite directed graph
 $\baiG = (\baiV,\baiE)$ as a \textit{basic graph}.
We write $\baiV = V(\baiG)$ and $\baiE = E(\baiG)$.
Hereafter, inverted hats are used to denote a basic graph $\baiG$, its vertex
 set $\baiV$, and edge set $\baiE$.  
Inverted hats are also used to denote a vertex $\baiv \in \baiV$ and an edge $\baie \in \baiE$. 

\begin{nota}
In this paper, we assume that a basic graph $\baiG = (\baiV,\baiE)$
 is a surjective relation,
 i.e. for every vertex $\baiv \in \baiV$,
 there exist edges $(\baiu_1,\baiv),(\baiv,\baiu_2) \in \baiE$.
\end{nota}

For basic graphs $\baiG_i = (\baiV_i,\baiE_i)$ with $i = 1,2$,
 a map $\baifai : \baiV_1 \to \baiV_2$ is said to be
 a \textit{basic graph homomorphism}
 if, for every edge $(\baiu,\baiv) \in \baiE_1$,
 it follows that $(\baifai(\baiu),\baifai(\baiv)) \in \baiE_2$.
In this case, we write $\baifai : \baiG_1 \to \baiG_2$.
For a basic graph homomorphism $\baifai : \baiG_1 \to \baiG_2$,
 we say that $\baifai$ is \textit{edge-surjective}
 if $\baifai(\baiE_1) = \baiE_2$.
Suppose that a basic graph homomorphism
 $\baifai : \baiG_1 \to \baiG_2$ satisfies the following condition:
\[(\baiu,\baiv),(\baiu,\baiv') \in \baiE_1 \text{ implies that }
 \baifai(\baiv) = \baifai(\baiv').\]
In this case, $\baifai$ is said to be \textit{\pdirectional}.
Suppose that a basic graph homomorphism
 $\baifai : \baiG_1 \to \baiG_2$ satisfies the following condition:
\[(\baiu,\baiv),(\baiu',\baiv) \in \baiE_1 \text{ implies that }
 \baifai(\baiu) = \baifai(\baiu').\]
In this case, $\baifai$ is said to be \textit{\mdirectional}.
A basic graph homomorphism is
 \textit{\bidirectional} if it satisfies both of the above conditions.

\begin{defn}\label{defn:cover}
For basic graphs $\baiG_1$ and $\baiG_2$,
 a basic graph homomorphism $\baifai : \baiG_1 \to \baiG_2$
 is called a \textit{basic (graph) cover} if it is a \pdirectionals \tesgh.
\end{defn}

\begin{defn}\label{defn:basic-covering-model}
We call a sequence of basic covers
 $\baiGcal : \covrepa{\baiG}{\baifai}$
  a \textit{basic graph covering model},
 or just a \textit{basic covering model},
 if $\baiG_0 := \left(\seb \baiv_0 \sen, \seb (\baiv_0,\baiv_0) \sen \right)$
 is a singleton graph.
\end{defn}
In this paper, considering the numbering of Bratteli diagrams,
 we use this numbering for basic covering models.
In the original paper,
 we used the numbering $\baiG_n \getsby{\baifai_n} \baiG_{n+1}$.
Let us write basic graphs as $\baiG_i = (\baiV_i,\baiE_i)$ for $i \benonne$.
We define the \textit{inverse limit} of $\baiGcal$ as follows:
\[V_{\baiGcal} := \seb (\baiv_0,\baiv_1,\baiv_2,\dotsc)
 \in \prod_{i = 0}^{\infty}\baiV_i~|~\baiv_i = \baifai_{i+1}(\baiv_{i+1})
 \text{ for all } i \benonne \sen \text{ and}\]
\[E_{\baiGcal} := 
\seb (x,y) \in V_{\baiGcal} \times V_{\baiGcal}~|~
(\baiu_i,\baiv_i) \in \baiE_i \text{ for all } i \benonne\sen,\]
where $x = (\baiu_0 = \baiv_0,\baiu_1,\baiu_2,\dotsc),
 y = (\baiv_0,\baiv_1,\baiv_2,\dotsc) \in V_{\baiGcal}$.
The sets $\baiV_i$ $(i \bni)$ are equipped with the discrete topology, and the
 set $\prod_{i = 0}^{\infty}\baiV_i$ is equipped with the product topology.
\begin{nota}\label{nota:open-sets-of-vertices}
Let $X = V_{\baiGcal}$,
 and let us define a map $f : X \to X$ as $f(x) = y$
 if and only if $(x,y) \in E_{\baiGcal}$.
For each $n \bni$, the projection
 from $X$ to $\baiV_n$ is denoted by $\baifai_{\infty,n}$. 
For $\baiv \in \baiV_n$,
 we define a closed and open set $U(\baiv) := \baifai_{\infty,n}^{-1}(\baiv)$.
For a subset $A \subset \baiV_n$,
 we define a closed and open set $U(A) := \bigcup_{\baiv \in A}U(\baiv)$.
\end{nota}

\begin{nota}
Let $\baiGcal$ be a basic covering model $\covrepa{\baiG}{\baifai}$.
Let $X = V_{\baiGcal}$, and let us define $f : X \to X$ as stated above.
Then, by \cite[Theorem 3.9]{Shimomura_2014SpecialHomeoApproxCantorSys},
 $(X,f)$ is a zero-dimensional system.
This zero-dimensional system $(X,f)$ is written as $\invlim \baiGcal$.
By \cite[Lemma 3.5]{Shimomura_2014SpecialHomeoApproxCantorSys},
 if all $\baifai_n$ $(n > 0)$ are \bidirectional, then
 $f$ is a homeomorphism.
\end{nota}

\begin{defn}
Let $\baiGcal : \covrepa{\baiG}{\baifai}$ be a basic covering model.
A zero-dimensional system $(Y,g)$ \textit{admits a basic covering model}
 $\baiGcal$ if $(Y,g)$ is topologically conjugate to $\invlim{\baiGcal}$.
A zero-dimensional system $(Y,g)$
 \textit{admits a \bidirectionals basic covering model $\baiGcal$} if,
 in addition, all $\baifai_n$ $(n > 0)$ are \bidirectional.
\end{defn}

The following holds:
\begin{thm}[Theorem 3.9 and Lemma 3.5 of \cite{Shimomura_2014SpecialHomeoApproxCantorSys}]\label{thm:0dim=covering}
Every zero-dimensional system $(X,f)$ admits a basic covering model.
Every zero-dimensional system $(X,f)$ is invertible if and only if
 it admits a \bidirectionals basic covering model.
\end{thm}

\begin{rem}
Let $\baiGcal : \covrepa{\baiG}{\baifai}$ be a basic
 covering model and $\invlim \baiGcal = (X,f)$.
For each $n \bni$,
 the set $\Ucal(\baiG_n) := \seb U(\baiv) \mid \baiv \in V(\baiG_n) \sen$
 is a closed and open partition
 such that $U(\baiv) \cap f(U(\baiu)) \nekuu$
 if and only if $(\baiu,\baiv) \in E(\baiG_n)$.
Furthermore, $\bigcup_{n \bni} \Ucal_n$ generates the topology of $X$.
Conversely, suppose that $\Ucal_n$ $(n \bni)$
 is a sequence of finite closed and open partitions
 of a compact metrizable zero-dimensional space $X$,
 $\bigcup_{n \bni}\Ucal_n$ 
 generates the topology of $X$, and $f : X \to X$ is a
 continuous surjective map such that,
 for any $U \in \Ucal_{n+1}$, there exists 
 $U' \in \Ucal_n$ such that $f(U) \subset U'$.
Then, we can define a basic covering model in a natural way
 (see the discussion after
 \cite[Theorem 3.9]{Shimomura_2014SpecialHomeoApproxCantorSys}).
\end{rem}

\begin{nota}
Let $\baiG = (\baiV,\baiE)$ be a basic graph.
A sequence of vertices $w = (\baiv_0,\baiv_1,\dotsc,\baiv_l)$ of $G$
 is said to be a \textit{walk} of
 \textit{length} $l$ if $(\baiv_i, \baiv_{i+1}) \in \baiE$
 for all $0 \le i < l$.
We denote $l(w) := l$.
\end{nota}

\section{Graph coverings weighted by length}\label{sec:weighted-covering}
\label{sec:extended-graph-covering}

In this section, we introduce both weighted and flexible graph covering
 models.
We derive some regularity conditions on these models,
and state a condition that we think is closely related to the basic sets.
We also present our main theorem in terms of graph covering models.

\subsection{Introduction to the new graph covering models}\label{subsec:newmodels}
In this subsection, we introduce weighted (flexible) graph covering models
 to solve the main part of \cref{thm:Bratteli-Vershik-for-all}.
In this subsection,
 we introduce the weighted (flexible) graphs and their covers.
To distinguish them from basic graphs,
 weighted graphs are denoted as $\wgG = (\wgV,\wgE)$,
 with overlines for sets of vertices, edges, and the graphs themselves.
To distinguish flexible graphs from weighted graphs,
 the latter are denoted as $\fgG = (\fgV,\fgE)$,
 with large tildes for sets of vertices, edges, and the graphs themselves.
In \cref{nota:basic-graph},
 we construct a basic graph $\bai{G} = (\bai{V},\bai{E})$ from
 a weighted graph $\wgG = (\wgV,\wgE)$.

\begin{defn}\label{defn:weighted-graph}
Let $\wgG = (\wgV,\wgE)$ be a pair of finite sets such that
\itemb
\item there exists a source map $s : \wgE \to \wgV$
 and a range map $r : \wgE \to \wgV$ and
\item each vertex $v \in \wgV$ has edges
 $e_1,e_2 \in \wgE$
 such that
 $s(e_1) = r(e_2) = v$.
\itemn
In contrast to the case of basic graphs, the elements of $\wgV$ and $\wgE$
 may not have overlines if there is no confusion.
Let $l : \wgE \to \Posint$ be a map.
A pair $(\wgG,l)$ is called a \textit{weighted graph},
 and has the vertex set $\wgV$, edge set $\wgE$,
 and weight map $l : \wgE \to \Posint$.
Multiple directed edges are permitted between each pair of vertices.
In general, we omit $l$ and say that $\wgG$ is a weighted graph.
For each $e \in \wgE$, $l(e)$ is called the \textit{length of} $e$.
\end{defn}

\begin{defn}\label{defn:flexible-graph}
Let $\fgG = (\fgV,\fgE)$ be a pair of finite sets such that 
\itemb
\item there exists a source map $s : \fgE \to \fgV$
 and a range map $r : \fgE \to \fgV$ and
\item each vertex $v \in \fgV$ has edges
 $e_1,e_2 \in \fgE$
 such that
 $s(e_1) = r(e_2) = v$.
\itemn
As for weighted graphs, the elements of $\fgV$ and $\fgE$
 may not be denoted by large tildes if there is no confusion.
A finite directed graph $\fgG$ is called a \textit{flexible graph},
 and has the vertex set $\fgV$ and the edge set $\fgE$.
Multiple directed edges are permitted between each pair of vertices.
\end{defn}


A sequence $w = e_1 e_2 \dotsb e_k$ $(e_i \in \wgE, i = 1,2,\dotsc,k)$
 is a \textit{walk} if $r(e_i) = s(e_{i+1})$ for all $1 \le i < k$.
We denote $\wgE(w) := \seb e_i \mid 1 \le i \le k \sen$
 and $\wgV(w) := \seb s(e_i) \mid 1 \le i \le k \sen \cup \seb r(e_k) \sen$.
Further, we define the \textit{range map} $r(w) := r(e_k)$
 and the \textit{source map} $s(w) := s(e_1)$.
A walk $w$ is a \textit{cycle} if $s(w) = r(w)$.
A cycle $w = e_1 e_2 \dotsb e_k$ is a \textit{circuit} if
 all $s(e_i)$ $(1 \le i \le k)$ are mutually distinct.
For a flexible graph, a \textit{walk} $w$, $\fgE(w)$, $\fgV(w)$,
 \textit{range map}, \textit{source map}, \textit{cycle},
 and \textit{circuit} are also defined.
For a walk $w = e_1\ e_2\ \dotsb\ e_k$ of a weighted graph,
 the length is defined as $l(w) := \sum_{1 \le i \le k}l(e_i)$.
For a walk $w = e_1\ e_2\ \dotsb\ e_k$ in both $\wgG$ and $\fgG$,
 we write $w(\min) = e_1$ and $w(\max) = e_k$.
For a weighted graph $\wgG = (\wgV,\wgE)$,
 the set of finite walks is denoted by $\wgW(\wgG)$.
For a flexible graph $\fgG = (\fgV,\fgE)$,
 the set of finite walks is denoted by $\fgW(\fgG)$.

Let $\wgG_1 = (\wgV_1,\wgE_1)$
 and $\wgG_2 = (\wgV_2,\wgE_2)$ be weighted graphs,
and let $\wgW_i = \wgW(\wgG_i)$ $(i = 1,2)$.
A \textit{weighted graph homomorphism}
 $\fai : \wgG_1 \to \wgG_2$ is a pair of maps
 $\fai_V : \wgV_1 \to \wgV_2$
 and $\fai_E : \wgE_1 \to \wgW_2$ such that
\itemb
\item $\fai_V(s(e)) = s(\fai_E(e))$ for all $e \in \wgE_1$,
\item $\fai_V(r(e)) = r(\fai_E(e))$ for all $e \in \wgE_1$, and
\item $l(\fai_E(e)) = l(e)$ for all $e \in \wgE_1$.
\itemn
Note that we have \textbf{not} assumed the condition
 $\fai_V(\wgV_1) = \wgV_2$.
We extend
 $\fai_E(e_1 e_2 \dotsb e_k) = \fai_E(e_1) \fai_E(e_2) \dotsb \fai_E(e_k)$
 for each finite walk $e_1 e_2 \dotsc e_k  \in \wgW_1$.

For flexible graphs $\fgG_1 = (\fgV_1,\fgE_1)$
 and $\fgG_2 = (\fgV_2,\fgE_2)$,
a \textit{flexible graph homomorphism} $\fai$ satisfies
\itemb
\item $\fai_V(s(e)) = s(\fai_E(e))$ for all $e \in \fgE_1$ and
\item $\fai_V(r(e)) = r(\fai_E(e))$ for all $e \in \fgE_1$.
\itemn

\begin{nota}
For both weighted graph homomorphisms and flexible graph homomorphisms,
 we abbreviate $\fai(w) := \fai_E(w)$ for a walk $w$
 and $\fai(v) := \fai_V(v)$ for a vertex $v$.
\end{nota}

A weighted graph homomorphism $\fai$ is \textit{edge-surjective}
 if $\bigcup_{e \in \wgE_1}\wgE(\fai(e)) = \wgE_2$.
A flexible graph homomorphism $\fai$ is \textit{edge-surjective}
 if $\bigcup_{e \in \fgE_1}\fgE(\fai(e)) = \fgE_2$.

\begin{defn}\label{defn:weighted-and-flexible-graph-cover}
A weighted graph homomorphism $\fai$ is \textit{\pdirectional} if, 
 for every $e,e' \in \wgE_1$ with $s(e) = s(e')$,
 the walks $w = \fai(e)$ and $w' = \fai(e')$ satisfy
 $w(\min) = w'(\min)$.
A weighted graph homomorphism $\fai$ is \textit{\mdirectional} if, 
 for every $e,e' \in \wgE_1$ with $r(e) = r(e')$,
 the walks $w = \fai(e)$ and $w' = \fai(e')$ satisfy
 $w(\max) = w'(\max)$.
A weighted graph homomorphism is \textit{\bidirectional}
 if it satisfies both of the above conditions.
A weighted graph homomorphism $\fai : \wgG_1 \to \wgG_2$
 is called a \textit{weighted cover}
 if it is \pdirectionals and edge-surjective.
For flexible graph homomorphisms, \textit{\pdirectionality},
 \textit{\mdirectionality}, \textit{\bidirectionality},
 and \textit{flexible covers} are also defined.  
\end{defn}

\begin{defn}\label{defn:weighted-constant-cover}
Let $\fai : \wgG_1 \to \wgG_2$ be a weighted graph cover.
An edge $e \in \wgE_2$ is \textit{constantly covered by} an edge $e' \in \wgE_2$
 if $\fai(e') = e$.
Note that, in this case, we have $l(e') = l(e)$.
\end{defn}

\begin{defn}\label{defn:flexible-constant-cover}
Let $\fai : \fgG_1 \to \fgG_2$ be a flexible graph cover.
An edge $e \in \fgE_2$ is \textit{constantly covered by}
 an edge $e' \in \fgE_2$ if $\fai(e') = e$.
\end{defn}

\begin{defn}\label{defn:weighted-graph-covering}
A sequence of weighted covers
 $\wgGcal : \covrepa{\wgG}{\fai}$
 is said to be a \textit{weighted graph covering model}
 or \textit{weighted covering model}.
We assume that $\wgG_0$ is a singleton graph
 $(\seb v_0 \sen, \seb e_0 \sen)$ with $l(e_0) = 1$ and $s(e_0) = r(e_0) =v_0$.
For $m > n \ge 0$,
 the composition map
 $\fai_{m,n} := \fai_{n+1} \circ \fai_{n+2} \circ \dotsb \circ \fai_m$
 is naturally well defined.
Note that $\wgGcal$ may not be \bidirectional.
\end{defn}

\begin{nota}\label{nota:basic-graph}
We now construct a basic graph from 
 a weighted graph $\wgG = (\wgV,\wgE)$.
For each $e \in \wgE$,
 we form the set $\baiV(e)$ of vertices
 $\baiV(e) :=
 \seb \baiv_{e,0} = s(e),
 \baiv_{e,1}, \baiv_{e,2}, \dotsc, \baiv_{e,l(e)-1},
 \baiv_{e,l(e)} = r(e) \sen$.
Let $\baiV := \bigcup_{e \in \wgE}\baiV(e)$.
For each $e \in \wgE$,
 we form the set of edges
 $\baiE(e) := \seb (\baiv_{e,i}, \baiv_{e,i+1}) \mid 0 \le i < l(e)\sen$.
Let $\baiE := \bigcup_{e \in \wgE} \baiE(e)$.
Thus, we obtain a basic graph $\baiG = (\baiV,\baiE)$.
From this construction,
 the set $\wgV$ is considered to be a subset of $\baiV$.
Note that, if a pair of vertices $(v,v') \in \wgV \times \wgV$
 has more than one edge of length $1$
  directed from $v$ to $v'$,
 then they are merged into a single edge
 $(v,v') \in \wgV \times \wgV \subseteq \baiV \times \baiV$.
\end{nota}

\begin{lem}\label{lem:weighted-cover-to-basic-cover}
For a weighted cover $\fai : \wgG_1 \to \wgG_2$,
 we obtain a basic cover
 $\baifai : \baiG_1 \to \baiG_2$ in a natural way.
If $\fai$ is \bidirectional, then $\baifai$ is \bidirectional.
\end{lem}
\begin{proof}
We have considered $\wgV_1 \subseteq \baiV_1$ and $\wgV_2 \subseteq \baiV_2$.
The map $\baifai|_{\wgV_1} = \fai_V : \wgV_1 \to \wgV_2$ is well defined, and
 for each $e \in \wgE_1$, the map $\baifai|_{\baiV(e)} : \baiV(e) \to \baiV_2$
 is defined uniquely and $s(e), r(e) \in \wgV_1$ are mapped compatibly
 with $\baifai|_{\wgV_1}$.
Because each $e \in \wgE_1$ is mapped to a walk in $\wgG_2$,
 we have a graph homomorphism $\baifai : \baiG_1 \to \baiG_2$.
Thus, we only need to check the \pdirectionalitys condition.
For each vertex in $\baiV(e) \setminus \seb s(e),r(e) \sen$,
 this condition is trivial.
We only need to check every $s(e)$ $(e \in \wgE_1)$.
By the \pdirectionalitys condition
 for $\fai : \wgG_1 \to \wgG_2$,
 this is obvious.
Finally, the last statement is also obvious.
\end{proof}

\begin{nota}\label{nota:from-weighted-to-basic-covering}
In \cref{nota:basic-graph}, we transformed each $\wgG_n = (\wgV_n,\wgE_n)$
 into $\baiG_n = (\baiV_n,\baiE_n)$.
Further, in \cref{lem:weighted-cover-to-basic-cover}, each $\fai_n$
 is transformed into a basic cover $\baifai_n$.
Thus, from a weighted graph covering model $\wgGcal : \covrepa{\wgG}{\fai}$,
 we obtain
 a basic graph covering model $\baiGcal : \covrepa{\baiG}{\baifai}$.
\end{nota}

\begin{defn}\label{defn:flexible-graph-covering}
A sequence of flexible covers
 $\fgGcal : \covrepa{\fgG}{\fai}$
 is said to be a \textit{flexible graph covering model}
 or \textit{flexible covering model}.
We assume that $\fgG_0$ is a singleton graph
 $(\seb v_0 \sen, \seb e_0 \sen)$.
For $m > n \ge 0$,
 the composition map
 $\fai_{m,n} := \fai_{n+1} \circ \fai_{n+2} \circ \dotsb \circ \fai_m$
 is naturally well defined.
Note that $\fgGcal$ may not be \bidirectional.
We also note that, in the sequence, multiple or even infinite occurrences of
 the same flexible graphs are permitted.
\end{defn}

\begin{rem}\label{rem:principle}
Let $\fgGcal : \covrepa{\fgG}{\fai}$ be a flexible graph covering model with
 the singleton graph $\fgG_0 = (\seb v_0 \sen, \seb e_0 \sen)$.
We consider $e_0$ to have length $l(e_0) = 1$.
Then, the lengths of all the edges of the graphs will be fixed
 by assuming that each $\fai_n$ $(n \bpi)$ preserves the lengths of the walks.
Thus, a flexible covering model defines a weighted covering model uniquely.
It is also clear that a weighted covering model
 defines a flexible covering model uniquely.
We use the symbols $\wgGcal$ and $\fgGcal$ to show the relation between these graph covering models.
Each $\wgG_n$ $(n \bni)$ corresponds to $\fgG_n$ $(n \bni)$.
\end{rem}

\begin{defn}
Let $\fgGcal : \covrepa{\fgG}{\fai}$ be a flexible graph covering model 
 and $\wgGcal : \covrepa{\wgG}{\fai}$
 be the corresponding weighted graph covering model.
Then, for each $n \bni$ and $\fge \in \fgE_n$,
 there exists a corresponding edge $\wge \in \wgE_n$ of length
 $l(\wge)$.
We write $l(n,\fge) := l(\wge)$.
In contrast to basic covering models,
 using the overlined symbols for elements of $\wgE$ is \textbf{not} a notation.
Similarly, using symbols with large tildes is \textbf{not} a notation in the following, i.e. we use $e \in \wgE$
 and $e \in \fgE$ if there is no confusion.  
\end{defn}

Fix a weighted covering model $\wgGcal : \covrepa{\wgG}{\fai}$
 or a corresponding flexible covering model $\fgGcal : \covrepa{\fgG}{\fai}$.
In \cref{nota:from-weighted-to-basic-covering},
 we obtained a basic covering model
 $\baiGcal : \covrepa{\baiG}{\baifai}$.
Thus, we have a zero-dimensional system $\invlim \baiGcal$.
The zero-dimensional system $\invlim \baiGcal$ is also
 denoted as $\invlim \wgGcal$ or $\invlim \fgGcal$.
Thus, these three inverse limits are the same; in other words,
 $\invlim \fgGcal$ is considered to be $\invlim \wgGcal$,
 and $\invlim \wgGcal$ is considered to be $\invlim \baiGcal$.

\begin{defn}\label{defn:weighted-graph-covering-model}
A zero-dimensional system $(Y,g)$ has
 a \textit{weighted graph covering model} $\wgGcal : \covrepa{\wgG}{\fai}$
 or a \textit{flexible graph covering model} $\fgGcal : \covrepa{\fgG}{\fai}$
 if $(Y,g)$ is topologically conjugate
 to $\invlim \baiGcal$.
\end{defn}

\begin{nota}
Hereafter, a definition that is stated for a weighted graph covering model
 is also applicable to a flexible graph covering model, and
 a property that is satisfied by a weighted graph covering model
 is also satisfied by a flexible graph covering model,
 if it is possible via the argument in \cref{rem:principle}.
\end{nota}

Let $\wgGcal : \covrepa{\wgG}{\fai}$ be a weighted covering model and
 $n(0) = 0 < n(1) < n(2) < \dotsb$ be a sequence of positive integers.
We obtain a weighted covering model $\wg{\Gcal'} : \covrepa{\wg{G'}}{\fai'}$ by
 letting $\wg{G'}_0 = \wgG_0$,
 ${\wg{G'}}_i = \wgG_{n(i)}$ for all $i \ge 1$, and
 $\fai'_{i}
 = \fai_{n(i-1)+1} \circ\ \cdots\ \circ \fai_{n(i) -1} \circ \fai_{n(i)}$
 for all $i \ge 1$.
This procedure is called \textit{telescoping}, based on
 the telescoping of the Bratteli diagrams.
Suppose that an invertible zero-dimensional system $(Y,g)$
 has a weighted graph covering model $\wgGcal$.
It is easy to show that, by telescoping $\wgGcal$,
 we obtain a \bidirectionals weighted graph covering model.

\subsection{Basic sets and regularity conditions on models}
Medynets \cite{Medynets_2006CantorAperSysBratDiag} 
 defined the notion of the basic set for the Bratteli--Vershik models
 of aperiodic zero-dimensional systems. 
We define the same for our case (see \cref{defn:basic-set}).
From a weighted covering model $\wgGcal : \covrepa{\wgG}{\fai}$,
 we can obtain a basic covering model $\baiGcal : \covrepa{\baiG}{\baifai}$.
We write $\invlim \baiGcal = (X,f)$.
In \cref{nota:open-sets-of-vertices},
 we defined a natural projection $\baifai_{\infty,n} : X \to \baiV_n$
 such that $\baifai_{n+1} \circ \baifai_{\infty,n+1} = \baifai_{\infty,n}$ for
 all $n \bni$.
For $m > n \ge 0$, if we define
 $\baifai_{m,n} :=
 \baifai_{n+1} \circ \baifai_{n+2} \circ \dotsb \circ \baifai_m$,
 then we have $\baifai_{m,n} \circ \baifai_{\infty,m} = \baifai_{\infty,n}$.
We have also defined $U(\baiv) = \baifai_{n,\infty}^{-1}(\baiv)$
 for each $\baiv \in \baiV_n$.
Because we consider that $\wgV_n \subseteq \baiV_n$, 
 for each $v \in \wgV_n$, we have $U(v)$ by the above equation.
Using the same notation, we define $U(\wgV_n)$ for all $n \bni$.
Thus, we can state the following.

\begin{nota}\label{nota:Vinfty}
Suppose that $\wgGcal$ is a weighted covering model
 with $\invlim \wgGcal = (X,f)$.
We denote $\wgV_{\infty} := \bigcap_{n \ge 0}U(\wgV_n)$.
The above argument can also be applied to $\fgGcal$ via the argument
 in \cref{rem:principle}, i.e. we define $\fgV_{\infty} := \wgV_{\infty}$.
\end{nota}

\begin{defn}\label{defn:basic-set}
Let $(X,f)$ be an invertible zero-dimensional system.
A closed set $A \in X$ is called a \textit{basic set}
 if, for every $x \in X$, the orbit of $x$ enters $A$ at most once,
 and for every $x \in X$ and open set $U \supset A$,
 the positive orbit of $x$ enters $U$ at least once.
\end{defn}

In general, an orbit of $(X,f)$ may enter the set $\wgV_{\infty}$ many times.
Therefore, we need to define a condition that prevents this.
We will show that under certain conditions, the set $\wgV_{\infty}$
 is a basic set (see
 \cref{defn:closing-weighted-covering,thm:closing-implies-basic-set-wg}).

\begin{nota}\label{nota:we-can-define-a-natural-base-map}
For $e \in \wgE_n$ $(n \bpi)$, we have constructed a set of vertices
 $\baiV(e) := \seb \baiv_{e,0} = s(e),
 \baiv_{e,1}, \baiv_{e,2}, \dotsc, \baiv_{e,l(e)-1},
 \baiv_{e,l(e)} = r(e) \sen$.
In the case of $l(e) \ge 2$, we define the base floor
 $B(e) := f^{-1}(U(\baiv_{e,1}))$ and the tower
 $\barB(e) := \bigcup_{0 \le i < l(e)}f^i(B(e))$ with height $l(e)$.
It follows that $f^i(B(e)) = U(\baiv_{e,i})$ for each $0 < i < l(e)$.
If $l(e) = 1$, many towers may start from
 $U(s(e))$ and many different towers with height $1$
 may also be included in
 $U(s(e))$, so we cannot use $B(e) = U(s(e))$.
To discard these additional parts from $U(s(e))$, we use the towers
 obtained by $\wgE_{n+1}$.
We set $v := s(e)$ and
 $A(e) := \seb e' \in \wgE_{n+1} \mid \wgE(\fai_{n+1}(e')) \ni e \sen$.
Further, we set
 $A_1(e) := \seb e' \in A(e) \mid l(e') = 1 \sen$
 and $A_2(e) := \seb e' \in A(e) \mid l(e') \ge 2 \sen$.
For each $e' \in A_2(e)$,
 we write $\fai_{n+1}(e') = e'_1 e'_2 \dotsb e'_{k(e')}$ and
 set $J(e,e') := \seb i \in [1,k(e')] \mid e'_i = e \sen$.
For each $1 < j \in J(e,e')$, we define
 $l(e,e',j) := \sum_{1 \le i < j}l(e'_i)$.
If $1 \in J(e,e')$, we define $l(e,e',1) := 0$.
Then, $B(e) :=
 \bigcup_{e' \in A_1(e)}U(s(e')) \cup
 \bigcup_{e' \in A_2(e),\ j \in J(e,e')}f^{l(e,e',j)}(B(s(e')))$,
 and we define a tower $\barB(e) := B(e)$ with height $1$.
Now, by the \pdirectionalitys condition,
 we obtain a decomposition $X = \bigcup_{e \in \wgE_n}\barB(e)$
 for each $n \bni$.
The last decomposition is also denoted for $\fgGcal$
 \viatheargument.
\end{nota}

The following definitions lead us to the first regularity condition
 on weighted (flexible) graph covering models.

\begin{defn}\label{defn:infinite-constant-covering}
Let $\wgGcal : \covrepa{\wgG}{\fai}$ be a weighted covering model and
 $n \bni$.
A sequence $e_m \in \wgE_m$ $(m \ge n)$
 is an \textit{infinite constant covering from} $n$
 if $\fai_{m,m'}(e_m) = e_{m'}$ for all $m > m' \ge n$.
Note that, in this case, $l(e_m) = l(e_n)$ for all $m \ge n$.
If such a sequence exists over $e = e_n \in \wgE_n$, then
 we say that $e$ is \textit{infinitely constantly covered}.
We define the same for a flexible covering model $\fgGcal$.
We also note that if $e_m$ is a circuit for some $m > n$, then
 for every $i \in [n,m]$, $e_i$ is a circuit.
\end{defn}

Suppose that some $e \in \wgE_n$ $(n \bni)$ is infinitely constantly covered by
 a sequence $e_m$ $(m \ge n)$.
If we write $\seb x \sen = \bigcap_{i \ge n}U(s(e_i))$,
 then it follows that
\[\bigcap_{i \ge n}\barU(e_i)
 = \seb x, f(x), f^2(x), \dotsc, f^{l(e)-1}(x)\sen.\]
In addition, if the ${e_m}$'s are circuits, then $\bigcap_{i \ge n}\barU(e_i)$
 is a periodic orbit of least period $l(e)$.

\begin{defn}\label{defn:closing-weighted-covering}
For a weighted covering model $\wgGcal : \covrepa{\wgG}{\fai}$ having
 the \textit{closing property}, if
 there exists an infinite constant covering $e_m \in \wgE_m$ $(m \ge n)$
 from $n$, then each $e_m$ is a circuit.  
We define the same for a flexible covering model $\fgGcal$.
\end{defn}

The next theorem demonstrates the use of the closing property.

\begin{thm}\label{thm:closing-implies-basic-set-wg}
Let $\wgGcal : \covrepa{\wgG}{\fai}$ be
 a \bidirectionals weighted covering model.
It follows that $\wgV_{\infty}$ is a basic set if and only if
 $\wgGcal$ has the closing property.
The same is true for $\fgGcal$ and $\fgV_{\infty}$ 
 \viatheargument.
\end{thm}

\begin{proof}
Let $\invlim \wgGcal = (X,f)$.
By \cref{nota:we-can-define-a-natural-base-map}, we have a decomposition
 $X = \bigcup_{e \in \wgE_n}\barU(e)$ for each $n \bni$.
Suppose that, for every infinite constant covering $e_m$ $(m \ge n)$
 (for some $e_n \in \wgE_n$ and $n \bni$),
 every $e_m$ $(m \ge n)$ is a circuit.
We must show that $\wgV_{\infty}$ is a basic set.
Let $x \in X$
be a fixed point.
Evidently, there then exists a sequence $v_n \in \wgV_n$ $(n \bni)$
 such that $x \in U(v_n)$ for all $n \bni$.
Consequently, we have that $x \in \wgV_{\infty}$.
Next, suppose that $x$ is a periodic point with least period $l \ge 2$.
For every sufficiently large $n$, 
 there exists a unique edge $e_n \in \wgE_n$ with $l(e_n) \le l$ such that 
 $\barU(e_n) \ni x$.
Further, for every sufficiently large $n$, we have $\fai_{n+1}(e_{n+1}) = e_n$.
Thus, for sufficiently large $n$,
 the sequence $e_m$ ($m \ge n$) is
 an infinite constant cover.
By the assumption,
 $e_m$ ($m \ge n$) are circuits.
Because the diameters of $U(s(e_m))$ ($m \ge n$) tend to $0$,
 it follows that $l(e_m) = l$ ($m \ge n$).
Thus, for all $m \ge n$, the periodic orbit of $x$ enters $U(s(e_m))$ 
 exactly once.
Finally, suppose that $x$ is not periodic.
For each $k \bni$,
 take a unique $e_k \in \wgE_k$ such that $x \in \barU(e_k)$.
It is evident that the sequence $l(e_k)$ ($k \bni$) is not decreasing.
Suppose that $\lim_{k \to \infty} l(e_k) < \infty$.
Then, there exists an $n \bpi$ such that the sequence
 $e_m$ ($m \ge n$) is an infinite constant covering.
By the assumption,
 $e_m$ ($m \ge n$) are circuits, and
 $x$ becomes a periodic point, which is a contradiction.
Thus, we have that $\lim_{k \to \infty} l(e_k) = \infty$.
We have defined $\baiV(e_k) :=
 \seb \baiv_{e_k,0} = s(e_k),
 \baiv_{e_k,1}, \baiv_{e_k,2}, \dotsc, \baiv_{e_k,l(e_k)-1},
 \baiv_{e_k,l(e_k)} = r(e_k) \sen$.
There exists some $j(k)$ such that $0 \le j(k) < l(e_k)$ with
 $x \in U(\baiv_{e_k,j(k)})$.
Let $a(k) := j(k)$ and $b(k) := l(e_k) - j(k)$ for each $k$.
Then, both $a(k)$ and $b(k)$ with $k \bpi$ are non-decreasing,
 and $\lim_{k \to \infty}a(k)$ and/or $\lim_{k \to \infty}b(k)$
 are $\infty$.
Thus, the orbit of $x$ enters $\wgV_{\infty}$ at most once.
Next, suppose that $U \supseteq \wgV_{\infty}$ is an open set.
Because each $U(\wgV_i)$ $(i \ge 0)$ is also compact,
 for sufficiently large $m > 0$, it follows that $U(\wgV_m) \subseteq U$.
Let $l_m := \max \seb l(e) \mid e \in \wgE_m \sen$.
Then, for each $x \in X$,
 it follows that $f^i(x) \in U(\wgV_m) \subseteq U$ for some $0 \le i < l_m$.
This proves that $\wgV_{\infty}$ is a basic set.

To show the converse, suppose that $\wgV_{\infty}$ is a basic set.
The proof is established by contradiction.
Let $e_m \in \wgE_m$ ($m \ge n$)
 be an infinite constant covering with $l = l(e_m)$
 for all $m \ge n$.
Suppose that there exists some $N > n$ such that $e_N$ is not a circuit.
Then, for every $m > N$, it follows that $e_{m}$ is not a circuit.
We denote $\seb y \sen = \bigcap_{m \ge N}U(s(e_m))$.
It follows that $y, f^l(y) \in \wgV_{\infty}$ and $y \ne f^l(y)$,
 which is a contradiction.
Thus, we have proved the converse.
The last statement follows \viatheargument.
\end{proof}

\begin{rem}\label{rem:periodicity-regulation-is-stronger-than-closing}
In our main result, we show that there exists a weighted graph
 covering model $\wgGcal$ in which, as $n$ increases,
 the lower towers in $\wgG_n$ must have periodic orbits.
Nevertheless, the notion of the closing property discussed above
 is not sufficiently strong for this purpose.
To see this, consider a zero-dimensional system with a fixed point $p$
 such that no periodic orbit of least period $2$ exists;
 in addition, there exists a sequence of periodic orbits
 $\seb P_n \sen_{n > 0}$ with $\lim_{n \to \infty}P_n = \seb p \sen$, i.e.
 the orbits converge to $p$.
We assume that each $P_m$ has least period $4$.
In expressing this zero-dimensional system by a weighted graph covering model,
 as $n$ increases, the weighted graph $\wgG_n$ has to separate $P_m$
 from $p$ for smaller $m$ compared to $n$.
We assume that $P_m$ ($m \le n$) are all separated by
 towers of level $n$ and all $P_{n'}$ ($n' > n$) are contained
 by the tower of height $1$ that also contains $p$.
In level $n$, when $P_m$ $(m \le n)$ are all separated from $p$
 and also from each other,
 we have to provide towers that cover $P_m$ ($m \le n$).
We assume that all $P_m$ ($m < n$) are covered by a single tower with
 height $4$
 and $P_n$ is covered by two towers of height $2$.
In the next level, i.e. $n+1$, the two parts will be unified to be covered
 by a single tower of height $4$.
Thus, as $n$ increases, there must exist an $e \in \wgE_n$
 with $l(e) = 2$ while $\barU(e)$ does not contain any periodic orbit
 of least period $2$.  
\end{rem}

Through this observation, we present another notion for weighted graph 
 covering models.
\begin{defn}\label{defn:periodicity-regulated-wg}
Let $\bl$ be a sequence $l_1 < l_2 < \dotsb$ of positive integers.
A weighted covering model $\wgGcal : \covrepa{\wgG}{\fai}$ is
 $\bl$-\textit{periodicity-regulated}
 if
\itemb 
\item for each $n \bpi$ and each $e \in \wgE_n$ with $l(e) \le l_n$,
 $e$ is infinitely constantly covered by circuits, i.e.
 there exists a sequence $e_m \in \wgE_m$ $(m \ge n)$ such that
 $e_n = e$ and $\fai_{m,m'}(e_m) = e_{m'}$ for all $m > m' \ge n$.
\itemn
This definition is possible for a flexible covering model \viatheargument.
\end{defn}

By telescoping, it is evident that the $\bl$-periodicity-regulations
 are preserved.
It is also evident that we can obtain new $\bl$ by `telescoping'.
This notion of $\bl$-periodicity-regulation is
 stronger than the closing property, as shown by \cref{rem:periodicity-regulation-is-stronger-than-closing} and 
 \cref{lem:periodicity-regulated-implies-closing-wg}.
The existence of the $\bl$-periodicity-regulated weighted graph covering
 model will be shown in \cref{thm:main}.

\begin{lem}\label{lem:periodicity-regulated-implies-closing-wg}
Let $\bl : l_1 < l_2 < \dotsb$ be a sequence of positive integers and
 $n \bni$.
Suppose that $\wgGcal : \covrepa{\wgG}{\fai}$ is
 a weighted graph covering model
 that satisfies the $\bl$-periodicity-regulated condition.
Then, $\wgGcal$ has the closing property.
The same is true for $\fgGcal$ via \cref{rem:principle}.
\end{lem}
\begin{proof}
Let $e_m \in \wgE_m$ ($m \ge n$) be an infinite constant covering.
Take some $N > n$ such that $l_N \ge l(e_N)$.
Then, for every $m \ge N$, it follows that $l_m \ge l_N \ge l(e_N) = l(e_m)$.
Thus, $e_m$ with $m \ge N$ are all circuits.
It follows that all $e_m$ with $m \ge n$ are circuits.
The last statement follows \viatheargument.
\end{proof}

Immediately, we have the following lemma.

\begin{lem}\label{lem:periodicity-regulated-implies-basic-set-wg}
Let $\bl : l_1 < l_2 < \dotsb$ be a sequence of positive integers and
 $n \bni$.
Suppose that $\wgGcal : \covrepa{\wgG}{\fai}$ is
 a weighted graph covering model
 that satisfies the $\bl$-periodicity-regulated condition.
Then, it follows that $\wgV_{\infty}$ is a basic set.
\end{lem}
\begin{proof}
The proof follows from \cref{thm:closing-implies-basic-set-wg,lem:periodicity-regulated-implies-closing-wg}.
\end{proof}

\subsection{Main theorem for weighted graph covering models}\label{subsec:mainpart}

We now present a version of the main theorem.

\begin{thm}\label{thm:main}
Let $(X,f)$ be an invertible zero-dimensional system.
For every sequence $\bl: l_1 < l_2 < \dotsb$ of positive integers,
 $(X,f)$ admits a \bidirectionals weighted graph covering model
 $\wgGcal : \covrepa{\wgG}{\fai}$ that is $\bl$-periodicity-regulated.
It follows that $\wgV_{\infty}$ is a basic set.
The same is true for $\fgGcal$ and $\fgV_{\infty}$ \viatheargument.
\end{thm}

The main part of the proof is somewhat lengthy
 and is presented in \cref{prop:main}.
First, we present \cref{cor:KriegerLemma} and its proof under the assumption
 that \cref{thm:main} holds.
Second, we present the proof of \cref{thm:main} under the assumption that
 \cref{prop:main} holds.
The proof of \cref{prop:main} is presented at the end of this subsection.
We start by presenting some devices for the argument.
\textbf{The following argument is applicable only to weighted graphs.}
It is not applicable to flexible graphs.
Nevertheless,
 the conclusion that is related to flexible covering models in \cref{thm:main}
 is also valid via \cref{rem:principle}.
A corollary of \cref{thm:main} is an analogue of Krieger's lemma
 (cf. \cite[Lemma 2.2]{Boyle_1984LowerEntroFactOfSoficSys})
 for the zero-dimensional case.
\begin{cor}[Krieger's Marker Lemma]\label{cor:KriegerLemma}
Let $(X,f)$ be an invertible zero-dimensional system.
Let $L \ge 1$ and $\ep > 0$.
Then, there exists a closed and open set $F$ such that
$\barF := \bigcup_{0 \le i \le L}f^i(F)$ is a tower, and
 if $x \notin \bigcup_{-L \le i \le L} f^i(F)$,
 then there exists a periodic point $p$ of least period $\le L$
 such that $d(x,p) \le \ep$.
\end{cor}

\begin{proof}
Let $\bl : l_1 < l_2 < \dotsb$ be an arbitrary sequence of
 positive integers.
By \cref{thm:main}, we have a weighted graph covering model
 $\wgGcal : \covrepa{\wgG}{\fai}$ of $(X,f)$
 that is $\bl$-periodicity-regulated.
We can assume that $\invlim \wgGcal = (X,f)$ by the conjugating map.
Then, there exists an $n \bpi$ such that
 $L \le l_n$.
Let $l = l_n$, $\wgG = \wgG_n$, $\wgV = \wgV_n$, and $\wgE = \wgE_n$.
By \cref{nota:we-can-define-a-natural-base-map},
 we have a decomposition $X = \bigcup_{e \in \wgE}\barB(e)$.
Let $L' := L + 1$.
Further, let $J := \seb e \in \wgE \mid l(e) \ge L' \sen$ and
 $P := \seb e \in \wgE \mid l(e) \le L \sen$.
Then, $J \cup P = \wgE$.
For each $e \in J$, take the maximal integer $k(e) \ge 0$ with
 $L' k(e) \le l(e) - L'$.
For each $e \in J$, we define
 $E(e) := \bigcup \seb f^{L'i}(B(e)) \mid 0 \le i \le k(e)\sen$.
We set $E := \bigcup_{e \in J} E(e)$.
Then, for each $e \in J$,
 $f^i(E(e))$ $(0 \le i \le L)$ are mutually disjoint and
 $f^i(E(e)) \subseteq \barB(e)$ for all $0 \le i \le L$.
If $L' k(e) = l(e) - L'$ ($e \in J$),
 then we denote $\barE(e) := \bigcup_{0 \le i \le L}f^i(E(e))$.
In this case, we denote $F(e) := E(e)$.
If $L' k(e) < l(e) - L'$ ($e \in J$),
 then we denote 
\[F(e) := E(e) \cup \seb x \in f^{L' k(e) + L'}(B(e)) \mid
 f^i(x) \notin E \text{ for all } 0 \le i \le L \sen.\]
Further, we denote $F := \bigcup_{e \in J} F(e)$.
Then, $F$ is closed and open and $f^i(F)$ ($0 \le i \le L$) 
 are mutually disjoint.
If $x \notin \bigcup_{-L \le i \le L}f^i(F)$, then
 $x \in \barE(e)$ for some $e \in P$.
Thus, there exists a periodic point $p \in \barE(e)$ with least period 
 $\le L \le l$ and $d(x,p) \le \ep$.
This concludes the proof.
\end{proof}

To state \cref{prop:main}, we introduce some notions related to towers.
Let $(X,f)$ be an invertible zero-dimensional system
and let $\barU = \bigcup_{0 \le i < h}f^i(U)$ be a tower.
Suppose that $V \cap \barU = \kuu$
 and that $\barV := \bigcup_{0 \le i' < h'}f^{i'}(V)$ is  another tower.
We say that $\barV$ \textit{accompanies the tower} $\barU$
 if $f^{i'}(V) \cap f^i(U) \nekuu$ $(0 \le i < h, 0 \le i' < h')$
 implies that $f^{i'}(V) \subseteq f^i(U)$.
We say that a tower $\barV := \bigcup_{0 \le i' < h'}f^{i'}(V)$
 accompanies the tower $\barU := \bigcup_{0 \le i < h}f^i(U)$
 \textit{properly}
 if  $\barV$ accompanies  $\barU$ and
 $f^{h'-1}(V) \cap f^i(U) \nekuu$ with $0 \le i < h$
 implies that $i = h -1$; i.e. if the tower $\barV$ meets the tower $\barU$,
 then the former always goes through the latter to the end of the latter.
Note that, when $h = 1$, the accompaniments are always proper.

For a weighted graph $\wgG = (\wgV,\wgE)$,
 we say that a pair of maps $B = (B_V,B_E)$ is a \textit{base map}
 if
\enumb
\item
 there exists an injective map
 $B_V : \wgV \to \seb U \mid U \text{ is closed and open and } U \nekuu\sen$,
\item
 there exists an injective map
 $B_E : \wgE \to \seb U \mid U \text{ is closed and open and } U \nekuu\sen$,
\item
 $B_V(v)$ $(v \in \wgV)$ are mutually disjoint, 
\item\label{item:base-mutually-disjoint}
 $B_E(e)$ $(e \in \wgE)$ are mutually disjoint,
\item for each $e \in \wgE$, $B_E(e) \subseteq B_V(s(e))$ 
 and $f^{l(e)}(B_E(e)) \subseteq B_V(r(e))$,
\item for each $v \in \wgV$,
 $B_V(v) = \bigcup_{e \in \wgE, s(e) = v}B_E(e)$,
\item for each $e \in \wgE$,
 $\barB_E(e) := \bigcup_{0 \le i < l(e)}f^i(B_E(e))$ is a tower,
\item\label{item:tower-mutually-disjoint}
 $\barB_E(e)$ $(e \in \wgE)$ are mutually disjoint, and
\item $X = \bigcup_{e \in \wgE}\barB_E(e)$.
\enumn
It is evident that \cref{item:tower-mutually-disjoint}
 implies \cref{item:base-mutually-disjoint}.
Note that, for each $v \in \wgV$, we can obtain 
 $B_V(v) = \bigcup_{e \in \wgE, r(e) = v}f^{l(e)}(B_E(e))$ with
 any base map.
We write $B : \wgG \to (X,f)$.
For a singleton weighted graph $\wgG_0 = (\seb v_0 \sen, \seb e_0 \sen)$
 in which $l(e_0) = 1$,
 the base map $B : \wgG \to (X,f)$ is well defined such that
 $B_V(v_0) = X$ and $B_E(e_0) = X$.
Suppose that there exists a base map $B : \wgG \to (X,f)$.
Then, there exists a unique natural map $\baifai : X \to \baiV$.
We denote $U(v) := \baifai^{-1}(v)$ for all $v \in \wgV$.
We say that a point $x \in X$ \textit{follows} a walk
 $w = e_1 e_2 \dotsb e_k \in \wgW$
 if $x \in B_E(e_1)$, and for all $1 \le i < k$,
 $f^{\sum_{j = 1}^i l(e_j)}(x) \in B_E(e_{i+1})$.
By the definition of the base map,
 it then follows that $f^{\sum_{j = 1}^k l(e_j)}(x) \in B_V(r(e_k))$.
For a weighted covering map
 $\fai : \wgG_1 \to \wgG_2$ and the base maps
 $B_i :\wgG_i \to (X,f)$ $(i= 1,2)$,
 we say that $B_1$ is \textit{compatible} with $B_2$
 if, for every $e \in E_1$, all points $x \in {B_1}_E(e)$ follow
 the same walk $\fai_E(e)$.
Even if $B_1$ is compatible with $B_2$, there may be multiple definitions
 of $B_1$.
In contrast, if $B_1$ is determined first, then
 there exists a unique $B_2$ that is compatible with $B_1$.
\begin{nota}
For a weighted graph $\wgG = (\wgV,\wgE)$
 and a base map $(B_V,B_E)$,
 we denote $U(v) := B_V(v)$ for each $v \in \wgV$ and
 $U(e) := B_E(e)$ for each $e \in \wgE$.
\end{nota}

\begin{lem}\label{lem:follow}
Let $(X,f)$ be an invertible zero-dimensional system.
Let $B : \wgG \to (X,f)$ be a base map with $\wgG = (\wgV,\wgE)$.
We can choose an arbitrarily small $\ep > 0$ 
 such that,
 for every vertex $v, v' \in \wgV$, every subset $K \subset U(v)$,
 and every $n \bpi$ with $f^n(K) \subset U(v')$,
 the following holds:
 if $\diam(f^i(K)) \le \ep$ for all $0 \le i \le n$,
 then
 all $x \in K$ follow the same walk, regardless of the choice of $x \in K$.
\end{lem}
\begin{proof}
We omit the proof as it is trivial.
\end{proof}

\begin{prop}\label{prop:main}
Let $(X,f)$ be an invertible zero-dimensional system
and let $B : \wgG \to (X,f)$ be a base map with $\wgG = (\wgV,\wgE)$.
Let $\ep > 0$ be arbitrarily small so as to satisfy the condition
 of \cref{lem:follow}.
Let $L > 0$ be a positive integer.
There exist an integer $r \ge 1$,
 a finite set
 $\seb U(j) \mid  1 \le j \le r\sen$ of closed and open subsets,
 integers $0 < h_j$ $(1 \le j \le r)$, and
 $v_j, v'_j \in \wgV$ $(1 \le j \le r)$ that satisfy the following:
\enumb
\item\label{prop:tower} for each $1 \le j \le r$,
 $\barU(j) := \bigcup_{0 \le i < h_j}f^i(U(j))$
 is a tower,
\item\label{prop:diameter}
 $\diam(f^i(U(j))) \le \ep$ for all $0 \le i \le h_j, 
 1 \le j \le r$,
\item\label{prop:disjoint}
 $\barU(j)$ $(1 \le j \le r)$ are mutually disjoint,
\item\label{prop:total} $X = \bigcup_{1 \le j \le r}\barU(j)$,
\item\label{prop:vertex} $U(j) \subseteq U(v_j)$
 and $f^{h_j}(U(j)) \subseteq U(v'_j)$ for each
 $1 \le j \le r$, and
\item\label{prop:periodic}
 for each $1 \le j \le r$ with $h_j \le L$,
 the tower $\barU(j)$
 contains a periodic point with least period $h_j$.
\enumn
\end{prop}

The above proposition is the main part of the (long) proof.
Note that, if $\ep$ is sufficiently small, then
 for each $j$ ($1 \le j \le r)$ there exist $e_j, e'_j \in \wgE$ such that
 $U(j) \subseteq U(e_j)$ and $f^{h_j}(U_j) \subseteq U(e'_j)$.
We continue our discussion while postponing the proof of \cref{prop:main}
 until the end of this subsection.

\begin{rem}\label{rem:delta-dense}
Let $(X,f)$ be an invertible zero-dimensional system.
Let $\Lambda$ be a finite set,
 $\seb B_{\lambda} \mid \lambda \in \Lambda \sen$ be mutually disjoint
 closed and open sets of $X$, and $h_{\lambda}$ ($\lambda \in \Lambda$) be
 positive integers.
Suppose that there exists a disjoint decomposition 
 $X = \bigcup_{\lambda \in \Lambda}
 \bigcup_{i = 0}^{h_{\lambda}-1}f^i(B_{\lambda})$.
Let $x \in \bigcup_{\lambda \in \Lambda}B_{\lambda}$ and $L$ be a positive
 integer such that $x, f(x), f^2(x), \dotsc, f^{L-1}(x)$ are mutually
 distinct and $f^L(x) \in \bigcup_{\lambda \in \Lambda}B_{\lambda}$.
Then, we take a small closed and open neighborhood $U_x \ni x$
 such that for each $0 \le i \le L$, $f^i(U_x)$ are in some $f^j(B_{\lambda})$
 ($\lambda \in \Lambda, 0 \le j < h_{\lambda}$).
By removing all the $f^i(U_x)$ ($0 \le i < L)$ from the original decomposition
 and adding $\bigcup_{i = 0}^{L-1}f^i(U_x)$,
 we obtain a new decomposition.
Suppose that $x \in X$ is positively transitive.
Then,
 it is easy to see that $X$ does not have isolated points
 and $f^n(x)$ is positively transitive for every $n \bi$.
Take an arbitrarily small $\delta > 0$.
We assume that $x \in \bigcup_{\lambda \in \Lambda}B_{\lambda}$ and
 $\seb f^i(x) \mid 0 \le i < L \sen$ is $\delta$-dense in $X$.
Then, by the above construction, we have a $\delta$-dense tower
 $\bigcup_{i=0}^{L-1}f^i(U_x)$.
Thus, in \cref{thm:main}, for an arbitrary sequence $\delta_n > 0$
 ($n \bpi$),
 we may assume that the decomposition by towers caused by $\wgG_n$ or
 $\fgG_n$ contains a tower
 that is $\delta_n$-dense in $X$.
\end{rem}

\begin{nota}
For closed sets $K,K' \subset X$, we denote
 $\dist(K,K') := \min \seb d(x,x') \mid x \in K, x' \in K'\sen$.
\end{nota}

\begin{lem}\label{lem:refine}
Let $X$ be a compact zero-dimensional metrizable space with metric $d$.
Let $\Ucal$ be a finite partition of $X$ by closed and open sets.
Then, there exists some $\ep > 0$ such that,
 for every set $\Acal_{\ep}$ of closed sets of $X$ such that
 $\diam(A) \le \ep$ for all $A \in \Acal_{\ep}$,
 if $\dist(A,A') \le \ep$ $(A,A' \in \Acal_{\ep})$,
 then there exists a $U \in \Ucal$ such that
 $A \cup A' \subset U$.
\end{lem}
\begin{proof}
We omit the proof as it is trivial. 
\end{proof}

Here, we present the proof of \cref{thm:main}
 under the assumption that \cref{prop:main} holds.

\vspace{2mm}

\noindent \textit{Proof of \cref{thm:main}.}
By \cref{rem:principle}, it is sufficient to consider the case of weighted 
 graph covering models.
Let $\wgG_0$ be the singleton graph with $\wgV := \seb v_0 \sen$,
 $\wgE := \seb e_0 \sen$, and $l(e) := 1$.
Then, it is obvious that we can obtain the unique base map
 $B_0 : \wgG_0 \to (X,f)$.
Let $\bl : l_1 < l_2 < \dotsb$ be a sequence of positive integers.
Suppose that we have constructed weighted covering maps
 $\wgG_0 \getsby{\fai_1} \wgG_1 \getsby{\fai_2} \wgG_2 \getsby{\fai_3}
 \dotsb \getsby{\fai_n} \wgG_n$ and base maps 
 $B_m : \wgG_m \to (X,f)$ for all $0 \le m \le n$ that are compatible
 such that, for each $0 \le m \le n$ and $e \in \wgE_m$ with $l(e) \le l_m$,
 there exists a periodic point $p \in U(e)$ with least period $l(e)$.

We write $\wgG = \wgG_n$, $\wgV = \wgV_n$, $\wgE = \wgE_n$, and
 $B = B_n$.
Then, we have mutually disjoint towers $\barU(e)$
 $(e \in \wgE)$ such that 
 $X = \bigcup_{e \in \wgE}\barU(e)$.
We have selected an arbitrary $l_{n+1} > l_n$.
Let $L = l_{n+1}$ and apply \cref{prop:main} with some $\ep = \ep_{n+1} > 0$
 that satisfies the condition of \cref{lem:follow}
 for $\wgG$ and $B$.
Then, we have a finite set $\seb U(j) \mid 1 \le j \le r\sen$ of closed and
 open sets, integers $0 < h_j$ $(1 \le j \le r)$, and
 $v_j,v'_j \in \wgV$ $(1 \le j \le r)$ that satisfies  
 \crefrange{prop:tower}{prop:periodic} of \cref{prop:main}.
On the set
 $\Acal := \seb U(j) \mid 1 \le j \le r\sen$,
 we construct the equivalence relation
 $A \sim A'$ generated by
 the relation $A \sim A'$ if and only if
 $\dist(A, A) \le \ep$.
Let $\seb K_i \mid 1 \le i \le r'\sen$ be the set of equivalence classes.
For each $1 \le i \le r'$, we define $U(K_i) := \bigcup_{A \in K_i}A$.
Then, we have mutually disjoint closed and open sets
 $U(K_i)$ $(1 \le i \le r')$.
By \cref{lem:refine}, we have
 $\lim_{\ep \to 0}\max_{1 \le i \le r'} \seb \diam(U(K_i))\sen = 0$.
In particular, we can choose $\ep$ such that for each $1 \le i \le r'$,
 there exists some $e \in \wgE_n$ with $U(K_i) \subset U(e)$.
We construct $\wgG_{n+1}$ and the base map $B_{n+1} : \wgG_{n+1} \to (X,f)$
 as follows:
\itemb
\item $\wgV_{n+1} := \seb U(K_i) \mid 1 \le i \le r'\sen$,
\item ${B_{n+1}}_V(U(K_i)) := U(K_i) \subseteq X$ for all $1 \le i \le r'$,
\item $\wgE_{n+1} := \seb U(j) \mid 1 \le j \le r \sen$,
\item ${B_{n+1}}_E(U(j)) := U(j)$ for all $1 \le i \le r$,
\item if $e = U(j) \in \wgE_{n+1}$, we define $l(e) := h_j$, and
\item for each $e = U(j) \in \wgE_{n+1}$, $s(e)$ satisfies
 $U(j) \subseteq U(s(e))$
 and $r(e)$ satisfies $f^{l(e)}(U(j)) \subseteq U(r(e))$.
\itemn
The covering map $\fai_{n+1} : \wgG_{n+1} \to \wgG_n$ is defined as follows:
\itemb
\item $\fai_{n+1,V} : \wgV_{n+1} \to \wgV_n$ is defined such that,
 if $v = U(K_i)$, then $\fai_{n+1,V}(v)$ is a unique $v \in \wgV_n$
 such that $U(K_i) \subset U(v)$;
\item because of the choice of $\ep > 0$ in \cref{lem:follow,prop:main},
 the definition of $\fai_{n+1,E}$ is obvious.
 The compatibility of the base maps $B_n, B_{n+1}$ is evident.
\itemn
The \pdirectionalitys condition is derived from the condition that
 for each $1 \le i \le r'$,
 there exists some $e \in \wgE_n$ with $U(K_i) \subset U(e)$.
In a similar way, by the definition of $\ep$,
 we have the \mdirectionalitys condition.
We take $\ep_i$ $(i \bpi)$ such that it is strictly decreasing and satisfies
 $\ep_i \to 0$ as $i \to \infty$.
We obtain a weighted graph covering model $\wgGcal : \covrepa{\wgG}{\fai}$.
From this,
 we have a basic graph covering model $\baiGcal : \covrepa{\baiG}{\baifai}$
 and its inverse limit $\invlim \baiGcal = (Y,g)$.
We wish to show that $(Y,g)$ is topologically conjugate to $(X,f)$.
For each $n \bni$, we have shown that there exists a partition
 $X = \bigcup_{e \in \wgE_n} \barU(e)$ formed by closed and open sets.
Thus, for each $x \in X$ and each $n \bni$,
 there exists a unique $e \in \wgE_n$ and a unique $0 \le i < l(e)$
 such that $x \in f^i(U(e))$.
This determines a unique vertex $\baiv_n(x) = \baiv_{e,i} \in \baiV_n$ for each
 $x \in X$ and  each $n \bni$.
Evidently,
 the sequence $\psi(x) := (\baiv_0(x), \baiv_1(x), \baiv_2(x), \dotsc)$
 satisfies $\baifai_{n+1}(\baiv_{n+1}(x)) = \baiv_n(x)$ for all $n \bni$.
Thus, $\psi(x) \in Y$.
We have defined a map $\psi : X \to Y$;
obviously, $\psi$ is a continuous surjective mapping.
The injectivity is also obvious.
Finally, it is obvious that $(\baiv_n(x),\baiv_n(f(x))) \in \baiE_n$
 for all $n \bni$, which implies the commutativity.
Thus, a weighted graph covering model with $\bl$-periodicity-regulation
 has been constructed.
By \cref{lem:periodicity-regulated-implies-basic-set-wg},
 $\wgV_{\infty}$ is a basic set.
\qed

\vspace{5mm} 

At the end of this subsection, we present the proof of \cref{prop:main}.

\vspace{2mm} 

\noindent \textit{Proof of \cref{prop:main}.}
For the base map $B : (\wgV,\wgE) \to (X,f)$,
 we denote closed and open sets as $U(v) := B_V(v)$ ($v \in \wgV$).
Let $U := \bigcup_{v \in \wgV}U(v)$.
Then, this is also closed and open.
For each positive integer $p$, we define 
\[K(p) := \seb x \in X \mid
 x \text{ is a periodic point with the least period } p \sen.\]
Note that each periodic orbit crosses $U$ at least once.
We define the closed sets $K_1,K_2,\dotsc,K_L$,
 non-negative integers $i(1),i(2), \dotsc i(L)$, 
 and closed and open sets $W_{p,i} \subseteq U$
 ($1 \le p \le L, 1 \le i \le i(p)$)
 that satisfy the following conditions.
We state the convention that $K_0 \iskuu$, $i(0) = 1$, and $W_{0,1} \iskuu$.
\itemb
\item $K_p \subseteq K(p)$ for all $1 \le p \le L$,
\item each $\barW_{p,i} := \bigcup_{0 \le j < p}f^j(W_{p,i})$
 is a tower,
\item $\diam (f^j(W_{p,i})) \le \ep$ for all $1 \le p \le L$,
 $1 \le i \le i(p)$, $1 \le j \le p$,
\item $K_p \subseteq \bigcup_{1 \le i \le i(p)}\barW_{p,i}$
  for all $1 \le p \le L$,
\item  $\bigcup_{1 \le p \le p'+1}K(p)
 \subseteq \left(\bigcup_{0 \le p \le p', 1 \le i \le i(p)}
 \barW_{p,i} \right) \bigcup K_{p'+1}$ for all $0 \le p' < L$,
\item the towers $\barW_{p,i}$ ($1 \le p \le L, 1 \le i \le i(p)$)
 are mutually disjoint.
\itemn
These sequences are inductively obtained as follows.
Let $K_1 := \seb x \in X \mid f(x) = x \sen$.
Then, we have $K_1 \subseteq U$.
We decompose $K_1$ into mutually disjoint closed sets
 $K_{1,i}$ ($1 \le i \le i(1))$
 such that $\diam (K_{1,i}) \le \ep/2$ for all ($1 \le i \le i(1))$.
We can take mutually disjoint closed and open sets
 $K_{1,i} \subseteq W_{1,i} \subseteq U$ ($1 \le i \le i(1))$
 such that, for all $1 \le i \le i(1)$,
 it follows that $\diam (W_{1,i}) \le \ep$ and $\diam (f(W_{1,i})) \le \ep$.
Thus, the sequences are obtained for $L = 1$.

Suppose that the sequences are obtained for $L = l$.
Therefore, we have obtained the sequences $i(p)$ ($1 \le p \le l$) and 
 $W_{p,i}$ ($1 \le p \le l$, $1 \le i \le i(p)$) that satisfy
 the above conditions.
In particular, if we define
 $W := \bigcup_{1 \le p \le l, 1 \le i \le i(p)} \barW_{p,i}$,
 then $W$ contains all the periodic orbits with least period $\le l$.
We reorder $W_{p,i}$ ($1 \le p \le l$, $1 \le i \le i(p)$) as
 $V_j\ (1 \le j \le t)$.
We define $M := \seb x \in X \mid f^{l+1}(x) =x \sen \setminus W$.
It is obvious that $M$ is a closed set consisting of periodic points
 with least period $l+1$.
We define $M' := M \cap U$ as a closed set.
For each $x \in M'$, let
 $0 < i_1 < i_2 < \dotsb < i_{n(x)} \le l$
 be the maximal sequence such that
 $f^{i_k}(x)$ enters some
 $V_j$ $(1 \le j \le t)$.
For each $k$ with $1 \le k \le n(x)$,
 let $j_k$ $(1 \le j_k \le t)$
 be such that $f^{i_k}(x) \in V(j_k)$.
We construct a string $s(x) := ((i_k,j_k))_{1 \le k \le n(k)}$.
If $f^i(x) \notin V(j)$ for all $1 \le  i \le t$,
 then $s(x) := \kuu$.
We define a set $S := \seb s(x) \mid x \in M' \sen$.
Let $M(s) := \seb x \in M' \mid s(x) = s \sen$
 for each $s \in S$.
Then, we have a disjoint union
 $M' = \bigcup \seb M(s) \mid s \in S \sen$
 of closed and open subsets of $M'$.
We decompose each $M(s)$ as $M(s) = \bigcup_{1 \le i \le i(s)} M(s,i)$
 by mutually disjoint closed and open subsets of $M'$
 such that $\diam (f^j(M(s,i))) \le \ep/2$
 ($s \in S, 1 \le i \le i(s), 0 \le j \le l$).
We reorder $M(s,i)$ ($s \in S, 1 \le i \le i(s)$)
 as $M_j$ ($1 \le j \le J$) that are mutually disjoint closed and open
 subsets of $M'$,
 and define $M'_1 := M_1$ and, for each $1 < j \le J$,
 $M'_j := M_j \setminus \bigcup_{1 \le j' < j, 0 \le k \le l} f^k(M_{j'})$.
Then, all the $f^k(M'_j)$ ($1 \le j \le J, 0 \le k \le l$) are mutually
 disjoint.
We remove any empty sets $M'_j$ if necessary.
We then have closed and open sets $W_{l+1,j}$ ($1 \le j \le J$)
 such that $M'_j \subseteq W_{l+1,j}$ ($1 \le j \le J$) and 
  the towers $\barW_{l+1,j}$ ($1 \le j \le J$) are mutually disjoint.
It is evident that we can obtain $W_{l+1,j}$ ($1 \le j \le J$) small enough
 such that $\diam (f^k(W_{l+1,j})) \le \ep$
 for all $0 \le j \le J$ and $0 \le k \le l+1$. 
It is also evident that there is some $W_{l+1,j}$ ($0 \le j \le J$)
 small enough such that
 the towers $\barW_{l+1,j}$ ($1 \le j \le J$) accompany
 the towers $\barW(p,i)$ $(1 \le p \le l, 1 \le i \le i(p)$).
Let $W' = \bigcup_{1 \le i \le i(l+1)}\barW_{l+1,i}$.
We must renew the sets
 $W_{p,i}$ ($1 \le p \le l, 1 \le i \le i(p)$) by taking
 $W_{p,i} \setminus W'$ instead.
Finally, we define $i(l+1) := J$
 and $K_{l+1} := \seb x \in W' \mid f^{l+1}(x) = x \sen$.
This concludes the inductive step.

We reorder the sets $W_{p,i}$ ($1 \le p \le L, 1 \le i \le i(p)$)
 to obtain $V(j,0)$ ($1 \le j \le t_0$).
If $V(j,0) = W_{p,i}$, we define $h(j,0) := p$ for all $1 \le j \le t_0$.
Then, it follows that
 $\barV(j,0) := \bigcup_{0 \le i < h(j,0)}f^i(V(j,0))$
 are mutually disjoint towers and 
 for each $1 \le j \le t_0$,
 $\diam (f^i(V(j,0)) \le \ep$ for $0 \le i \le h(j,0)$.
We take vertices $v(j,0) \in \wgV$ ($1 \le j \le t_0$)
 such that
 $V(j,0) \subseteq U(v(j,0))$ ($1 \le j \le t_0$).
Let $1 \le j \le t_0$.
Because all $x \in V(j,0)$ follow the same walk,
 in particular, we have $f^{h(j,0)}(V(j,0)) \subseteq U(v(j,0))$.
Thus, we can take $v'(j,0) = v(j,0)$ ($1 \le j \le t_0$).
We define a closed and open set
 $X_0 := \bigcup_{1 \le j \le t_0}\barV(j,0)$.
Let $Y := \bigcup_{v \in \wgV}U(v) \setminus X_0$.
Then, $Y$ is a closed and open set.
If $Y = \kuu$, then $X = X_0$, and the proof is complete.
Therefore, we assume that $Y \nekuu$.

Let $x \in Y$.
Suppose that $x$ is a periodic point with least period $h(x)$.
Then, $h(x) > L$.
Take a closed and open set $Y \supseteq V(x) \ni x$ such that
 $V(x) \subseteq U(v)$ for some $v \in \wgV$,
 $\diam (f^i(V(x))) \le \ep$ for all $0 \le i \le h(x)$,
 $f^{h(x)}(V(x)) \subseteq U(v)$,
 and $f^i(V(x))$  $(0 \le i < h(x))$ are mutually disjoint.
Next, suppose that $x$ is not periodic.
Take an arbitrarily large integer $h(x) > L$  
 such that $f^{h(x)}(x) \in U(v')$
 for some $v' \in \wgV$.
We take a closed and open set $Y \supseteq V(x) \ni x$ such that
 $V(x) \subseteq U(v)$ for some $v \in \wgV$,
 $\diam (f^i(V(x))) \le \ep$ for all $0 \le i \le h(x)$,
 $f^{h(x)}(V(x)) \subseteq U(v')$,
 and all $f^i(V(x))$ $(0 \le i < h(x))$ are mutually disjoint.
Take a finite set $\seb x_j \mid 1 \le j \le n\sen$ such that
 $\bigcup_{1 \le j \le n}V(x_j) = Y$.
Then, it is easy to see that
 $\bigcup_{1 \le j \le n}\barV(j) \cup X_0 = X$.
We write $h(j) := h(x_j) > 1$ for each $1 \le j \le n$.
We define $V(1) := V(x_1)$.
When $V(j)$ is defined for $1 \le j < n$, we define
 $\barV(j) := \bigcup_{0 \le i < h(j)}f^i(V(j))$
 and
\[V(j+1)
 := V(x_{j+1})
 \setminus \bigcup_{1 \le j' \le j}\barV(j').\]
Finally, when $V(n)$ is defined,
 we define $\barV(n) := \bigcup_{0 \le i < h(n)}f^i(V(n))$.
Thus, we have $\bigcup_{1 \le j \le n} \barV(j) \supseteq Y$.
Note that $\barV(j)$ may not be mutually disjoint.
Moreover,
 each $\barV(j)$ $(1 \le j \le n)$
 may not be disjoint with respect to some $V(j',0)$ $(1 \le j' \le t_0)$.

We begin another inductive procedure for $\alpha = 0, 1,2,\dotsc,n$.
Suppose that, for $0 \le \alpha < n$, there exist $t_{\alpha} \ge 0$,
  bases $V(j,\alpha)$ $(1 \le j \le t_{\alpha})$,
  heights $h(j,\alpha)$ $(1 \le j \le t_{\alpha})$, and
  vertices $v(j,\alpha), v'(j,\alpha) \in \wgV$ $(1 \le j \le t_{\alpha})$
 such that
\begin{deepenum}
\item\label{al:tower} 
 $\barV(j,\alpha) := \bigcup_{0 \le i < h(j,\alpha)}f^i(V(j,\alpha))$
 is a tower for each $1 \le j \le t_{\alpha}$,
\item\label{al:diameter} $\diam (f^i(V(j,\alpha))) \le \ep$
 for all $0 \le i \le h(j,\alpha), 1 \le j \le t_{\alpha}$,
\item\label{al:disjoint} 
 the towers $\barV(j,\alpha)$ $(1 \le j \le t_{\alpha})$
 are mutually disjoint,
\item\label{al:region}
 if we define $X_{\alpha} : = \bigcup_{1 \le j \le t_{\alpha}}\barV(j,\alpha)$,
 then $X_{\alpha} = \bigcup_{0 \le j \le \alpha}\barV(j)$, with the convention
 that $\barV(0) := X_0$,
\item\label{al:vertex}  $V(j,\alpha) \subseteq U(v(j,\alpha))$
 and $f^{h(j,\alpha)}(V(j,\alpha)) \subseteq U(v'(j,\alpha))$ for each
 $1 \le j \le t_{\alpha}$, and
\item\label{al:period}
 if $1 \le \alpha \le n$,
 then $h(j,\alpha) > L$ for all $1 \le j \le t_{\alpha}$.
\end{deepenum}
Note that we have already shown the case in which $\alpha = 0$.
In addition, if the case in which $\alpha = n$ is true,
 then $X_n = \bigcup_{0 \le j \le n}\barV(j) = X$, and the proof is complete.


To proceed with the inductive step, let $\beta = \alpha +1$.
By the inductive hypothesis for \ref{al:region},
 $X_{\alpha} \cap V(\beta) = \kuu$.
For each $x \in V(\beta)$, let
 $0 < i_1 < i_2 < \dotsb < i_{n(x)} < h(\beta)$
 be the maximal sequence such that
 $f^{i_k}(x)$ enters some
 $V(j,\alpha)$ $(1 \le j \le t_{\alpha})$.
For each $k$ with $1 \le k \le n(x)$,
 let $j_k$ $(1 \le j_k \le t_{\alpha})$
 be such that $f^{i_k}(x) \in V(j_k,\alpha)$.
We construct a string $s(x) := ((i_k,j_k))_{1 \le k \le n(k)}$.
If $f^i(x) \notin V(j,\alpha)$ for all $0 < i < h(\beta)$ and for all
 $1 \le j \le t_{\alpha}$, then $s(x) := \kuu$.
We define a set $S := \seb s(x) \mid x \in V(\beta) \sen$.
Let $V(s) := \seb x \in V(\beta) \mid s(x) = s \sen$
 for each $s \in S(\beta)$.
Then, we have a disjoint union
 $V(\beta) = \bigcup \seb V(s) \mid s \in S \sen$
 of closed and open sets.
Moreover,
 $\barV(s) := \bigcup_{0 \le i < h(\beta)}f^i(V(s))$ $(s \in S)$
 are mutually disjoint towers.
It is evident that the towers $\barV(s)$ $(s \in S)$ accompany
 the towers $\barV(j,\alpha)$ $(1 \le j \le t_{\alpha})$.
For each $s \in S$,
 the accompaniment may not be proper, i.e. the following may occur:
\begin{equation}\label{beta:eqn:not-full:l=1}
f^{h(\beta)-1}(V(s)) \subseteq f^i(V(j,\alpha)) \text{ with }
 1 \le j \le t_{\alpha}, 1 \le i < h(j,\alpha) -1.
\end{equation}
Therefore, roughly speaking,
 we extend the height of tower $\barV(s)$ for 
 each $s \in S$.
If \cref{beta:eqn:not-full:l=1} does not occur for $s  \in S$,
 then we need not change
 the height $h(\beta)$ of tower $\barV(s)$, i.e.
 the base is $V(s)$ and the height is $h(s) := h(\beta)$.
If  \cref{beta:eqn:not-full:l=1} occurs for $s \in S$,
 then we do not change the base $V(s)$,
 but we change the height to $h(s) := h(\beta) + h(j,\alpha) - 1 - i$.
Then, $f^{h(s)-1}(V(s)) \subseteq f^{h(j,\alpha)-1}(V(j,\alpha))$ is satisfied.
It follows that $\diam (f^i(V(s))) \le \ep$ for all $0 \le i < h(s)$,
 because the extension is
 part of the tower with base $V(j,\alpha)$.
The disjointness of $f^i(V(s))$ $(0 \le i < h(s))$ is also obvious.
Now, we remove the towers $\barV(s)$ from each $V(j,\alpha)$.
Precisely, for each $1 \le j \le t_{\alpha}$,
 we define
 $V'(j,\alpha) := V(j,\alpha) \setminus \bigcup_{s \in S}\barV(s)$.
Thus, we obtain $\barV'(j,\alpha)
 := \bigcup_{0 \le i < h(j,\alpha)}f^i(V'(j,\alpha))
  = \barV(j,\alpha) \setminus \bigcup_{s \in S}\barV(s)$.
Now, we have mutually disjoint towers
 $\seb \barV'(j,\alpha) \mid 1 \le j \le t_{\alpha} \sen
 \cup \seb \barV(s) \mid s \in S \sen$.
Note that, for all $s \in S$, it follows that $h(s) \ge h(\beta) > L$.
We reorder these towers and rewrite the bases
 $V(j,\beta)$ $(1 \le j \le t_{\beta})$,
  heights $h(j,\beta)$ $(1 \le j \le t_{\beta})$,
 and  towers $\barV(j,\beta)$ $(1 \le j \le t_{\beta})$.
The existence of vertices
 $v(j,\beta), v'(j,\beta) \in \wgV$ $(1 \le j \le t_{\beta})$
 that satisfy $V(j,\beta) \subseteq U(v(j,\beta))$
 and $f^{h(j,\beta)}(V(j,\beta)) \subseteq U(v'(j,\beta))$ for each
 $1 \le j \le t_{\beta}$ is also obvious.
If we define $X_{\beta} : = \bigcup_{1 \le j \le t_{\beta}}\barV(j,\beta)$,
 then $X_{\beta} = X_{\alpha} \cup \barV(\beta)$.
This concludes the inductive step for $\alpha$, thus completing the proof.
\qed

\begin{rem}\label{rem:basic-set-is-contained}
Suppose that there exists a closed and open set $U \subseteq X$ and a positive
 integer $n$ such that $\bigcup_{i = 0}^nf^i(U) = X$.
Then, $U$ intersects with each periodic orbit.
According to the construction of $U(j)$ ($1 \le j \le r$)
 in terms of the proof of \cref{prop:main}, which is applied to the trivial base map, 
 it is easy to see that we can take all of the $U(j)$ as subsets of $U$.
Thus, we have that $\wgV_{\infty} \subseteq U(\wgV_1) \subseteq U$.
\end{rem}

\section{Link to the Bratteli--Vershik models}\label{sec:link}
In this section, we clarify the link to the Bratteli--Vershik models.
The Bratteli--Vershik models for zero-dimensional systems were 
 extended to aperiodic cases in \cite{Medynets_2006CantorAperSysBratDiag}.
In this section, we extend them to the case in which periodic points
 exist.

\subsection{The Bratteli--Vershik models}\label{subsec:Bratteli-Vershik-models}
In this subsection, we introduce the well-known Bratteli diagrams.
$V_n$ $(n \ge 0)$ denotes the set of vertices of 
 a Bratteli diagram and $E_n$ $(1 \le n)$ denotes the set of edges.
These symbols are not used for the descriptions of basic graphs discussed
 in \cref{sec:basic-covering}.
Basic graphs are denoted by inverted hats, e.g. $\baiG$, $\baiV$,
 and $\baiv \in \baiV$.  

\begin{defn}\label{defn:Bratteli-diagram}
A \textit{Bratteli diagram} is an infinite directed graph $(V,E)$,
 where $V$ is the vertex set and $E$ is the edge set.
These sets are partitioned into non-empty disjoint finite sets
$V = V_0 \cup V_1 \cup V_2 \cup \dotsb$ and $E = E_1 \cup E_2 \cup \dotsb$,
 where $V_0 = \seb v_0 \sen$ is a one-point set.
Each $E_n$ is a set of edges from $V_{n-1}$ to $V_n$.
Therefore, there exist two maps $r,s : E \to V$ such that $r:E_n \to V_n$
 and $s : E_n \to V_{n-1}$ for all $n \bpi$,
i.e. the \textit{range map} and the \textit{source map}, respectively.
Moreover, $s^{-1}(v) \nekuu$ for all $v \in V$ and
$r^{-1}(v) \nekuu$ for all $v \in V \setminus V_0$.
We say that $u \in V_{n-1}$ is connected to $v \in V_{n}$ if there
 exists an edge $e \in E_n$ such that $s(e) = u$ and $r(e) = v$.
\end{defn}

\begin{defn}
Let $(V,E)$ be a Bratteli diagram.
For each $0 \le n < m$, a sequence of edges
 $p = (e_{n+1},e_{n+2},\dotsc,e_m) \in \prod_{n < i \le m}E_i$ 
 with $r(e_i) = s(e_{i+1})$ for all $n < i < m$ is called a \textit{path}.
A path $p =  (e_{n+1},e_{n+2},\dotsc,e_m)$ goes from one vertex $v \in V_n$
 to another vertex $v' \in V_m$ if $v = s(e_{n+1})$ and $v' = r(e_m)$.
We write $p(i) := e_i$ for $n < i \le m$.
For each $n < m$, we define
 $E_{n,m} := \seb p \in \prod_{n < i \le m}E_i \mid p \text{ is a path.} \sen.$
For $p = (e_{n+1},e_{n+2},\dotsc,e_m) \in E_{n,m}$,
 the source map $s : E_{n,m} \to E_n$ and the range map $r : E_{n,m} \to E_m$
 are defined by $s(p) = s(e_{n+1})$ and $r(p) = r(e_{m})$.

For each $n \ge 0$, an infinite path $p = (e_{n+1},e_{n+2},\dotsc)$
 is also defined.
For each $n \ge 0$, $E_{n,\infty}$ denotes
 the set of all infinite paths from $V_n$.
For $p = (e_{n+1},e_{n+2},\dotsc) \in E_{n,\infty}$,
 the source map $s : E_{n,\infty} \to V_n$ is defined as $s(p) = s(e_{n+1})$.
For $p = (e_{n+1},e_{n+2},\dotsc) \in E_{n,\infty}$,
 we denote $p(i) := e_i$ for $n < i$.
For $0 \le n \le n' < m' \le m$ and $p \in E_{n,m}$,
 we denote $p[n',m'] := (p(n'+1),p(n'+2),\dotsc,p(m')) \in E_{n',m'}$.
For $0 \le n \le n'$ and $p \in E_{n,\infty}$,
 $p[n',\infty)$ is also defined.

In particular, we have defined the set
 $E_{0,\infty}$.
We consider $E_{0,\infty}$ with the product topology.
Under this topology, it is a compact zero-dimensional space.
\end{defn}

\begin{defn}\label{defn:ordered-Bratteli-diagram}
Let $(V,E)$ be a Bratteli diagram such that
$V = V_0 \cup V_1 \cup V_2 \cup \dotsb$ and $E = E_1 \cup E_2 \cup \dotsb$
 are the partitions,
 where $V_0 = \seb v_0 \sen$ is a one-point set.
Let $r,s : E \to V$ be the range map and the source map, respectively.
We say that $(V,E,\ge)$ is an \textit{ordered Bratteli diagram} if
 a partial order $\ge$ is defined on $E$ such that 
 $e, e' \in E$ are comparable if and only if $r(e) = r(e')$.
 Thus, we have a linear order on each set $r^{-1}(v)$ for each $v \in \Vp$.
The edges $r^{-1}(v)$ are numbered from $1$ to $\abs{r^{-1}(v)}$, and
 the maximal (resp. minimal) edge is denoted
 by $e(v,\max)$ (resp. $e(v,\min)$).
Let $E_{\max}$ and $E_{\min}$ denote the sets of maximal and minimal
 edges, respectively.
\end{defn}

\begin{defn}
Let $(V,E,\ge)$ be an ordered Bratteli diagram.
For each $0 < n < m$ and $v \in V_m$,
the set $\seb p \in E_{n,m} \mid r(p) = v \sen$
 is linearly ordered by the lexicographic order, 
 i.e. for $p \ne q \in E_{n,m}$ with $r(p) = r(q)$,
 $p < q$ if and only if
 $p(k) < q(k)$ with the maximal $k \in [n+1,m]$ such that $p(k) \ne q(k)$.
For each $n \bni$, suppose that
 $p, p' \in E_{n,\infty}$ are distinct cofinal paths,
 i.e. there exists a $k > n$
 such that $p(k) \ne p'(k)$, and for all $l > k$, $p(l) = p'(l)$.
We define the lexicographic order $p < p'$ if and only if $p(k) < p'(k)$.
In particular, we have defined the lexicographic order on $E_{0,\infty}$.
This is a partial order, and $p,q \in E_{0,\infty}$ is comparable if 
 and only if $p$ and $q$ are cofinal.
\end{defn}

Let $(V,E,\ge)$ be an ordered Bratteli diagram.
We define

\begin{fleqn}[20mm]
\begin{align*}
\ & E_{0,\infty,\min}  :=  \seb p \in E_{0,\infty} \mid 
 p(k) \in E_{\min} \myforall k \sen, \myand\\
\ & E_{0,\infty,\max}  :=   \seb p \in E_{0,\infty} \mid 
 p(k) \in E_{\max} \myforall k \sen.
\end{align*} 
\end{fleqn}
\begin{defn}
Let $(V,E,\ge)$ be an ordered Bratteli diagram.
For each $p \in E_{0,\infty} \setminus E_{0,\infty,\max}$,
 there exists the least $p' > p$ with respect to the lexicographic order.
We say that $(V,E,\ge)$ admits a continuous \textit{Vershik map}
 $\psi : E_{0,\infty} \to E_{0,\infty}$ if $\psi$ is continuous everywhere,
 $\psi(E_{0,\infty,\max}) = E_{0,\infty,\min}$, and 
 for each $p \in E_{0,\infty} \setminus E_{0, \infty,\max}$,
 it follows that $\psi(p)$ is the least path such that $p < \psi(p)$.
We note that $\psi$ is surjective and
 if $\abs{\psi^{-1}(x)} \ne 1$, then $x \in E_{0,\infty,\min}$.
\end{defn}

\begin{defn}
Let $(X,f)$ be a zero-dimensional system.
Let $(V,E,\ge)$ be an ordered Bratteli diagram
 with a continuous Vershik map $\psi : E_{0,\infty} \to E_{0,\infty}$.
Then, $(V,E,\ge,\psi)$
 is a \textit{Bratteli--Vershik model} of $(X,f)$ if
 $(X,f)$ is topologically conjugate to $(E_{0,\infty},\psi)$.
We also say that $(V,E,\ge,\psi)$ is a \textit{Bratteli--Vershik model}
 if $\psi$ is continuous and surjective.
\end{defn}

\begin{nota}\label{nota:paths-from-vertices}
Let $(V,E,\ge)$ be an ordered Bratteli diagram, $n > 0$, and $v \in V_n$.
We abbreviate $P(v) := \seb p \in E_{0,n} \mid r(p) = v \sen$.
We define $l(v) := \abs{P(v)}$ and
 write $P(v) = \seb p(v,0) < p(v,1) < \dotsb < p(v,l(v)-1) \sen$.
Let $U(v,i) := \seb x = (e_{x,1}, e_{x,2},\dotsc) \in E_{0,\infty} \mid 
 (e_{x,1}, e_{x,2}, \dotsc, e_{x,n}) = p(v,i) \sen$ for all $0 \le i < l(v)$.
We denote $\barU(v) := \bigcup_{0 \le i < l(v)}U(v,i)$.
Then, for any Vershik map $\psi$, $\barU(v)$
 is a tower, i.e. $\psi(U(v,i)) = U(v,i+1)$ for all $0 \le i < l(v)-1$.
\end{nota}

By referring to
 \cref{defn:infinite-constant-covering,defn:closing-weighted-covering},
 we will define analogies
 in the Bratteli--Vershik models.

\begin{defn}\label{defn:keeping-constant}
Let $(V,E)$ be a Bratteli diagram and $n \bni$.
We say that an infinite path $(e_{n+1},e_{n+2},\dotsc) \in E_{n,\infty}$
 is \textit{constant}
 if $\abs{r^{-1}(r(e_i))} = 1$ for all $i > n$.
A vertex $v \in V_n$ is \textit{infinitely constantly covered by} the path $p$
 if there exists a constant path $p \in E_{n,\infty}$
 with $v = s(p)$.
\end{defn}

\begin{defn}\label{defn:closing-BV}
A Bratteli--Vershik model $(V,E,\ge,\psi)$
 has the \textit{closing property} 
 if, for every constant path $(e_{n+1},e_{n+2}, \dotsc) \in E_{n,\infty}$
 with $n \bni$,
 the set $\bigcap_{m > n}\barU(s(e_m))$ is a periodic orbit
 (of least period $l(s(e_{n+1}))$).
\end{defn}

The theorem that corresponds to \cref{thm:closing-implies-basic-set-wg}
 also holds (see \cref{thm:closing-implies-basic-set-BV}).
The next lemma clarifies the meaning of the above definitions.

\begin{lem}
Let $n \bni$ and $(V,E,\ge,\psi)$ be a Bratteli--Vershik model
 that has the closing property.
Let $v \in V_n$ with some $n \bni$.
If $v$ is constantly covered by some infinite path,
 then $\barU(v)$ contains a periodic orbit of least period $l(v)$.
\end{lem}
\begin{proof}
Let $(e_{n+1},e_{n+2},\dotsc) \in E_{n,\infty}$ be an infinite path
 that covers $v$.
Then, by the closing property,
 the set $\bigcap_{m \ge n} \barU(s(e_m))$ is a periodic orbit.
\end{proof}

The periodicity-regulation is also obtained for the Bratteli--Vershik
 models.

\begin{defn}\label{defn:periodicity-regulated-BV}
Let $\bl : l_1 < l_2 < \dotsb$ be a sequence of positive integers.
We say that a Bratteli--Vershik model $(V,E,\ge,\psi)$
 is $\bl$-\textit{periodicity-regulated}
 if, for every $n > 0$ and every $v \in V_n$ with $l(v) \le l_n$,
 $\barU(v)$ has a periodic orbit of least period $l(v)$.
\end{defn}

\begin{lem}
Let $\bl : l_1 < l_2 < \dotsb$ be a sequence of positive integers.
If a Bratteli--Vershik model $(V,E,\ge,\psi)$
 is $\bl$-periodicity-regulated,
 then it has the closing property.
\end{lem}
\begin{proof}
Let $v \in V_n$ $(n \bni)$ and a path
 $p_v = (e_{n+1},e_{n+2}, \dotsc) \in E_{n,\infty}$ with
 $s(p_v) = v$ be a constant infinite path.
Take some $m > n$ such that $l_m \ge l(v)$.
Let $v_m := s(e_{m+1})$.
It is sufficient to show that $\bigcap_{i > n}\barU(v_i)$ is a periodic orbit.
It follows that $l(v_m) = l(v) \le l_m$.
By the $\bl$-periodicity-regulation property, for each $j \ge m$, $\barU(v_j)$ 
 has a periodic orbit of least period $l(v_j) = l(v)$.
This implies that the set $\bigcap_{i > n}\barU(v_i)$
 is a periodic orbit of least period $l(v)$.
\end{proof}

\subsection{Relation to weighted covering models}\label{subsec:weighted-covering-models}\label{subsec:relation}

Let $(X,f)$ be a general invertible zero-dimensional system
 and $\bl : l_1 < l_2 < \dotsb$ be a sequence of positive integers.
In \S \ref{sec:extended-graph-covering}, we showed that
 there exists some $\bl$-periodicity-regulated
 weighted graph covering model
 $\wgGcal : \covrepa{\wgG}{\fai}$ with
 $\invlim \wgGcal$ being topologically conjugate to $(X,f)$.
The base maps $B_n : \wgG_n \to X$ $(n \bni)$
 have been defined simultaneously.
We are now ready to show that $\wgGcal$ can be transformed to
 a Bratteli--Vershik model of $(X,f)$.
Subsequently, to show the converse,
 the following definitions are necessary.

\begin{nota}\label{nota:base-of-BV}
Let $(V,E,\ge,\psi)$ be a Bratteli--Vershik model with
 the Vershik map $\psi :E_{0,\infty} \to E_{0,\infty}$.
For each $n \bpi$ and $v \in V_n$, let $p(v,\min) \in P(v)$ be the
 minimal path from $v_0$ to $v$ and let $p(v,\max) \in P(v)$ be the
 maximal path from $v_0$ to $v$.
For each $v \in V_n$,
 denote the closed and open set
 $B(v) := \seb p \in E_{0,\infty} \mid p[0,n] = p(v,\min) \sen$ and
 $l(v) := \abs{P(v)}$.
Each $B(v)$ is said to be the \textit{base} of $v \in V_n$ and
 $l(v)$ is the \textit{height} of $v \in V_n$.
We denote $\barB(v) := \bigcup_{i=0}^{l(v)-1}\psi^i(B(v))$
 and obtain a decomposition 
 $E_{0,\infty} = \bigcup_{v \in V_n}\barB(v)$.
\end{nota}

\begin{defn}\label{defn:cluster}
Let $(V,E,\ge,\psi)$ be a Bratteli--Vershik model with
 the Vershik map $\psi :E_{0,\infty} \to E_{0,\infty}$.
For the set $\seb B(v) \mid v \in V_n\sen$ of bases in level $n$,
 define an equivalence relation
 $\simeq$
 generated by $B(v) \simeq B(v')$ if there exists some $v'' \in V_n$ such that
 $B(v) \cap \psi^{l(v'')}(B(v'')) \nekuu$ and,
 at the same time, $B(v') \cap \psi^{l(v'')}(B(v'')) \nekuu$.
Define $\Acal_n := \seb B(v) \mid v \in V_n\sen/\simeq$.
We write $\Acal_n = \seb v_{n,1},\dotsc,v_{n,a(n)} \sen$.
Now, for each $1 \le i \le a(n)$,
 we have a closed and open set $U_{n,i} := \bigcup_{B \in v_{n,i}}B$.
We say that each $U_{n,i}$ ($1 \le i \le a(n)$) is a \textit{cluster of} bases
 of $V_n$.
\end{defn}

\begin{defn}\label{defn:nesting-property}
Let $(V,E,\ge,\psi)$ be a Bratteli--Vershik model with
 the Vershik map $\psi :E_{0,\infty} \to E_{0,\infty}$.
The Bratteli--Vershik model has the
 \textit{nesting property at level} $n$
 if each cluster of $V_{n+1}$ is a subset of a base $B(v)$
 for some $v \in V_n$.
We say that a Bratteli--Vershik model has the \textit{nesting property}
 if it has the nesting property at all levels $n \bni$.
\end{defn}

A Bratteli--Vershik model $(V,E,\ge,\psi)$
 has the nesting property at level $n$
 if and only if,
 for each $v \in V_{n+1}$, there exists some $v' \in V_n$ such that
 $\psi^{l(v)}(B(v)) \subseteq B(v')$.

\begin{nota}
In the proof of the next theorem,
 using the term `edge' for both $\wgG$
 and a Bratteli diagram may result in ambiguity, even if we use symbols such as $\wge \in \wgE$.
Thus, we use the symbol $\wgp$ for the elements of $\wgE$
 and refer to them as `paths' of $\wgG$,
 i.e. a path $\wgp \in \wgE$ of $\wgG$.
However, the term `path' is also used for the paths of a Bratteli diagram.
This convention is applied to the cases in which both the Bratteli--Vershik
 model and weighted (flexible) graph covering model are required in the
 proof.
\end{nota}

\begin{thm}[From weighted covering model to BV model]
\label{thm:main-link-wg-to-BV}
Let $\wgGcal:  \covrepa{\wgG}{\fai}$ be a weighted graph covering model
 and $\invlim \wgGcal = (X,f)$.
Then, we can construct a Bratteli--Vershik model $(V,E,\ge,\psi)$
 with the nesting property
 and a homeomorphism $\phi : E_{0,\infty} \to X$
 such that $f \circ \phi = \phi \circ \psi$ and
  $\phi(E_{0,\infty,\min}) = \wgV_{\infty}$.
If $\wgGcal$ has the closing property,
 then the Bratteli--Vershik model has the closing property.
If $\wgGcal$ is $\bl$-periodicity-regulated,
 then the Bratteli--Vershik model is 
 $\bl$-periodicity-regulated.
The same is true for a flexible graph covering model $\fgGcal$
 \viatheargument.
\end{thm}

\begin{proof}
Let $\wgGcal:  \covrepa{\wgG}{\fai}$ be a weighted graph covering model
 and $\invlim \wgGcal = (X,f)$.
For each $n \bni$, $V_n \subset V$ will be identified with $\wgE_n$.
Suppose that $n \bni$ and every $\wgp \in \wgE_{n+1} = V_{n+1}$
 is written as $\fai_{n+1}(\wgp) = \wgp_1 \wgp_2 \dotsb \wgp_{k(p)}$ with
 $\wgp_i \in \wgE_n = V_n$ $(1 \le i \le k(p))$.
Then, $E_{n+1} \subset E$ will be identified
 with the set
 $\seb (\wgp,\wgp_i,i) \mid \fai(\wgp) = \wgp_1 \wgp_2 \dotsb \wgp_{k(p)},\ 
 \wgp \in V_{n+1},\ \wgp_i \in V_n\ \ (1 \le i \le k(p)) \sen$.
The order of $(\wgp,\wgp_i,i)$ will be $i$.

As described above, to construct an ordered Bratteli diagram $(V,E,\ge)$,
 for each $n \bni$, let $V_n := \wgE_n$.
Further, let $n \bni$ and $v \in V_{n+1}$.
Then, $v$ is a path $v = \wgp \in \wgE_{n+1}$ of $\wgG_{n+1}$.
Therefore, $\fai(\wgp)$ is a walk in $\wgG_n$.
We write $\fai(\wgp) = \wgp_1 \wgp_2 \dotsb \wgp_{k(p)}$
 with $\wgp_i \in \wgE_{n}$ $(1 \le i \le k)$.
Let $E_{n+1} :=
 \seb (v,\wgp_i,i) \mid \fai(\wgp) = \wgp_1 \wgp_2 \dotsb \wgp_{k(p)},\ 
 v = \wgp \in V_{n+1},\ \wgp_i \in V_n\ \ (1 \le i \le k(p)) \sen$.
We define the source map $s((v,u,i)) := u$ and the range map
 $r((v,u,i)) := v$.
The order of each edge $e = (v,u,i)$ is $i$, i.e. 
 if $e_1 = (v,u_1,i_1)$ and $e_2 = (v,u_2,i_2)$, then $e_1 < e_2$ if and only
 if $i_1 < i_2$.
Thus, we have constructed an ordered Bratteli diagram $(V,E,\ge)$
 from a weighted graph covering model.

Next, we will show that there exists a homeomorphism 
 $\phi : E_{0,\infty} \to X$ that satisfies
 $f \circ \phi = \phi \circ \psi$.
It is easy to check that, for each $n \bni$ and each $\wgp \in \wgE_n = V_n$,
 the value of $l(\wgp)$ coincides
 with both $\wgp \in \wgE_n$ and $\wgp \in V_n$.
In \cref{nota:basic-graph}, for each $n \bpi$ and each $\wgp \in \wgE_n$,
 we have formed the set $\baiV(\wgp)$ of vertices of the basic graph $\baiG_n$:
 $\baiV(\wgp) :=
 \seb \baiv_{\wgp,0} = s(\wgp),
 \baiv_{\wgp,1}, \baiv_{\wgp,2}, \dotsc, \baiv_{\wgp,l(\wgp)-1},
 \baiv_{\wgp,l(\wgp)} = r(\wgp) \sen$.
If we consider $\wgp \in V_n$ ($n \bpi$)
 in the ordered Bratteli diagram $(V,E,\ge)$,
 then in \cref{nota:paths-from-vertices}, we have defined
 $P(\wgp) = \seb (e_1,e_2,\dotsc,e_n) \in E_{0,n} \mid r(e_n) = \wgp \sen$
 with the lexicographic order.
Because $\abs{P(\wgp)} = l(\wgp)$,
 we can write 
 $P(\wgp) =
 \seb \bx_{\wgp,0} < \bx_{\wgp,1} < \bx_{\wgp,2} < \dotsb
 < \bx_{\wgp,{l(p)-1}} \sen$.
We define a map $\phi_{\wgp} : P(\wgp) \to \baiV(\wgp) \subseteq \baiV_n$
 by $\phi_{\wgp}(\bx_{\wgp,i}) := \baiv_{\wgp,i}$ $(0 \le i < l(\wgp))$.
Further, we define a map $\phi_n : \bigcup_{\wgp \in V_n}P(\wgp) \to \baiV_n$
 by $\phi_n|_{P(\wgp)} := \phi_{\wgp}$ ($\wgp \in V_n$).
For an arbitrarily fixed $\bx \in E_{0,n}$,
 let $C(\bx) := \seb x \in E_{0,\infty} \mid 
 x[0,n] = \bx \sen$.
For an arbitrarily fixed $x = (e_1,e_2,\dotsb) \in E_{0,\infty}$,
 we define $\bx_{n} := (e_1,e_2, \dotsc,e_n)$ for all $n \bpi$.
Then, for $0 < n < m$, $C(\bx_{n}) \supseteq C(\bx_{m})$.
We obtain a commutativity condition
 $\baifai_{m,n}(\phi_m(\bx_m)) = \phi_n(\bx_n)$.
Thus, we have a continuous map $\phi : E_{0,\infty} \to X$.
Apparently, this map is surjective.
We have to show that it is injective.
Suppose that the map is not injective.
Then, there exist $x_1 \ne x_2 \in E_{0,\infty}$
 such that $\phi(x_1) = \phi(x_2)$.
We write $x_i := (e_{i,1},e_{i,2},\dotsc)$ ($i = 1,2$).
Further, we write $\bx_{i,n} := (e_{i,1},e_{i,2},\dotsc,e_{i,n})$ ($i = 1,2$).
Suppose that one of $x_i$ ($i = 1,2$) is not minimal.
Then, there exists an $n > 0$ such that one of $\bx_{i,n}$ ($i = 1,2$)
 is not minimal and $\bx_{1,n} \ne \bx_{2,n}$.
In this case, it is easy to see that $\phi(x_1) \ne \phi(x_2)$,
 which is a contradiction.
Suppose that both $x_i$ ($i = 1,2$) are minimal.
In this case, if $r(e_{1,n}) = r(e_{2,n})$ (in the Bratteli diagram)
 for infinitely many $n$,
 then we have $\bx_{1,n} = \bx_{2,n}$ for such $n$.
Thus, $x_1 = x_2$, which is a contradiction.
Therefore, in this case, there exists an $N > 0$ such that
 $r(e_{1,n}) \ne r(e_{2,n})$ (in the Bratteli diagram) for all $n \ge N$.
Let $\wgp_{i,n} := r(e_{i,n}) \in \wgE_n = V_n$
 for $i = 1,2$ and for all $n \ge N$.
Recall that $\wgp_{i,n}$ ($i = 1,2$) are paths of $\wgG_n$.
Therefore, we have $s(\wgp_{i,n}) \in \wgV_n$ for $i = 1,2$ and $n \ge N$.
Because $\phi(x_1) = \phi(x_2)$, it follows that
 $s(\wgp_{1,n}) = s(\wgp_{2,n})$ for all $n \ge N$.
By the \pdirectionalitys of $\wgGcal$, we have $\wgp_{1,n} = \wgp_{2,n}$ for
 all $n \ge N$, which is a contradiction.
Thus, $\phi$ is bijective and it is a homeomorphism.

Next, we check that the Vershik map $\psi$ can be well defined and 
 that $f \circ \phi = \phi \circ \psi$ is satisfied.
We can always uniquely define the Vershik map
 $\psi$ on the set $E_{0,\infty} \setminus E_{0,\infty,\max}$ by
 the lexicographic order.
From the construction, $f \circ \phi = \phi \circ \psi$ is satisfied
 on the set $E_{0,\infty} \setminus E_{0,\infty,\max}$.
Because $\phi : E_{0,\infty} \to X$ is a homeomorphism,
 we can extend the map $\psi$ uniquely
 on all of $E_{0,\infty}$ such that
 $f \circ \phi = \phi \circ \psi$ is satisfied.
It is easy to check that $\phi(E_{0,\infty,\min}) = \wgV_{\infty}$;
 consequently, it is easy to check that
 $\psi (E_{0,\infty,\max}) = E_{0,\infty,\min}$.
We have to check the nesting property.
Let $(e_1,e_2,\dotsc,e_{n+1}) \in E_{n+1}$
 be a maximal path in the ordered Bratteli diagram.
Let $v = r(e_{n+1}) \in V_{n+1}$ in the Bratteli diagram.
Then, $v$ is a path $\wgp \in \wgE_{n+1}$.
We can find a vertex $\wgv \in \wgV_{n+1}$ such that
 $r(\wgp) = \wgv$ in the weighted graph.
Any paths $\wgp_1,\wgp_2 \in \wgE_{n+1}$ that start from $\wgv$ in the
 weighted graph satisfy
 $(\fai_{n+1}(\wgp_1))(\min) = (\fai_{n+1}(\wgp_2))(\min)$.
This implies the nesting property of the Bratteli--Vershik model.

Suppose that the weighted covering model has the closing property.
Then, because the conjugating homeomorphism
 preserves the periodic orbits,
 the closing property of the Bratteli--Vershik model follows.

Finally, suppose that $\wgGcal$ is $\bl$-periodicity-regulated.
Let $n \bni$ and $v \in V_n$ be such that $l(v) \le l_n$.
Then, $v \in V_n = \wgE_n$ is a circuit that is infinitely constantly 
 covered by circuits of $\wgE_m$ ($m \ge n$), i.e.
 a periodic orbit with least period $l(v)$ is included in $\barU(v)$,
 as desired.
Thus, the Bratteli--Vershik model is $\bl$-periodicity-regulated.
The last statement follows via \cref{rem:principle}.
\end{proof}

We now present the proof of \cref{thm:Bratteli-Vershik-for-all}:

\noindent \textit{Proof of \cref{thm:Bratteli-Vershik-for-all}.}
Let $(X,f)$ be a homeomorphic (compact) zero-dimensional system.
Let $\bl : l_1 < l_2 < \dotsb $ be an arbitrary sequence of positive integers.
In \cref{thm:main}, we constructed a weighted graph covering model
 $\wgGcal : \covrepa{\wgG}{\fai}$ that is $\bl$-periodicity-regulated.
Further, the inverse limit $\invlim \wgGcal := \invlim \baiGcal$
 is topologically conjugate to $(X,f)$.
By \cref{thm:main-link-wg-to-BV} above, from $\wgGcal$,
 we could have constructed a Bratteli--Vershik model $(V,E,\ge,\psi)$
 with the nesting property.
Furthermore, because $\wgGcal$ could be constructed to be
 $\bl$-periodicity-regulated,
 the Bratteli--Vershik model is $\bl$-periodicity-regulated.
This completes the proof.
\qed

The nesting property need not be used
 to show \cref{thm:Bratteli-Vershik-for-all}.
However, it is important in the following \cref{thm:main-link-BV-to-wg}.

\begin{lem}\label{lem:telescoping-implies-nesting-property}
Let $(V,E,\ge,\psi)$ be a Bratteli--Vershik model with
 the Vershik map $\psi :E_{0,\infty} \to E_{0,\infty}$.
After sufficient telescoping, we obtain a Bratteli--Vershik
 model with the nesting property.
\end{lem}
\begin{proof}
This proof requires the continuity of $\psi$.
Fix a metric $d$ on $E_{0,\infty}$.
Note that, after telescoping, we have a canonical isomorphism
 on $(E_{0,\infty},\psi)$ and, thus, the identification of the metric $d$.
Fix $n \bpi$.
Note that the bases $B(v)$ ($v \in V_n$) are mutually disjoint closed sets.
Thus, we can find some $\delta$ such that
 $0 < \delta < \min
 \seb d(x,y) \mid x \in B(v), y \in B(v'), v,v' \in V_n, v \ne v',\sen$.
Let $\ep > 0$ be such that, if $d(x,y) \le \ep$, then $d(f(x),f(y)) \le \delta$.
It is evident that, for sufficiently large $m > n$,
 each $f^i(B(v))$ $(1 \le i < l(v),\ v \in V_m)$ has diameter less than $\ep$.
Thus, by the continuity of $\psi$, for each base $B(v)$ of $v \in V_n$,
 $f^{l(v)}(B(v))$ has diameter less than $\delta$.
It follows that $f^{l(v)}(B(v))$ is contained in some
 base $B(v)$ ($v \in V_n$).
By telescoping from $n$ to $m$, the new Bratteli--Vershik model
 has the nesting property at level $n$.
Because $n \bni$ is arbitrary, we have the desired result.
\end{proof}

\begin{thm}[From BV model to weighted covering model]
\label{thm:main-link-BV-to-wg}
Let $(V,E,\ge,\psi)$ be a Bratteli--Vershik model
 with the nesting property.
There exists a corresponding weighted graph covering model  
 $\wgGcal:  \covrepa{\wgG}{\fai}$ such that, if
 $\invlim \wgGcal = (X,f)$, there exists
 a homeomorphism $\phi : X \to E_{0,\infty}$
 with  $\phi \circ f = \psi \circ \phi$ and
  $\phi(\wgV_{\infty}) = E_{0,\infty,\min}$.
If the Bratteli--Vershik model has the closing property,
 then $\wgGcal$ has the closing property.
If the Bratteli--Vershik model is 
 $\bl$-periodicity-regulated,
 then $\wgGcal$ is $\bl$-periodicity-regulated.
The same is true for a flexible graph covering model $\fgGcal$ \viatheargument.
\end{thm}
\begin{proof}
Let $(V,E,\ge,\psi)$ be a Bratteli--Vershik model with the nesting property.
For each $n \bni$, $V_n \subset V$ will be identified with $\wgE_n$.
Precisely, for each $v \in V_n$,
 define the closed and open set
 $B(v) := \seb p \in E_{0,\infty} \mid p[0,n] = p(v,\min) \sen$.
Then, we define $\barB(v) := \bigcup_{i=0}^{l(v)-1}\psi^i(B(v))$, and
 obtain the decomposition
 $E_{0,\infty} = \bigcup_{v \in V_n}\barB(v)$.
To construct the set of vertices in the weighted graphs,
 for each $n \bpi$, we define the set $\Acal_n = \seb U_{n,i} \mid
 1 \le i \le a(n) \sen$ of clusters (see \cref{defn:cluster}).
By the nesting property, for each $n \bpi$ and each cluster $U_{n+1,i}$
 ($1 \le i \le a(n+1)$),
 there exists a base $B(v)$ ($v \in V_n)$ that contains the cluster.
We denote $\wgV_n := \Acal_n$ and $\wgE_n = V_n$ for each $n \bpi$.
For each $\wgp = v \in \wgE_n = V_n$,
 the path $\wgp$ starts from the cluster that contains $B(v)$ and ends
 in the cluster that contains $\psi(B(v))$.
$\wgG_0$ is defined as the singleton graph
 such that $\wgE_0$ is identified with $V_0$.
Thus, we obtain the weighted graphs $(\wgV_n,\wgE_n)$ for each $n \bpi$.
The weight map $l : \wgE_n \to \Posint$ is defined as the height 
 of the towers $\barB(v)$ for each $v \in V_n$.
Fix $n \bpi$ and $v \in V_n$
 from the ordered Bratteli diagram.
We write $r^{-1}(v) = \seb e_1 < e_2 < \dotsc < e_k \sen$.
Thus, we obtain $v_1,v_2, \dotsc, v_k \in V_{n-1} = \wgE_{n-1}$.
We define $\fai_n(v) = v_1 v_2 \dotsc v_k$.
This can be considered as the map from $v \in V_n = \wgE_n$
 to the walk $v_1 v_2 \dotsc v_k \in \wgW_{n-1}$.
In this way, we obtain a covering map $\fai_n : \wgG_n \to \wgG_{n-1}$.
By the nesting property in the Bratteli--Vershik model,
 it follows that the covering maps $\fai_n$ are \pdirectional.

Thus, we have constructed a weighted graph covering model
 $\wgGcal : \covrepa{\wgG}{\fai}$.
We write $(X,f) = \invlim \wgGcal$.
We need to define a homeomorphism $\phi : X \to E_{0,\infty}$.
To obtain $(X,f)$, we construct a basic graph covering model
 $\baiGcal : \covrepa{\baiG}{\baifai}$.
Let $n \bni$ and $\wgp \in \wgE_n = V_n$.
Then, we write
 $\baiV(\wgp) := \seb \baiv_{\wgp,0} = s(\wgp),
 \baiv_{\wgp,1}, \baiv_{\wgp,2},
 \dotsc, \baiv_{\wgp,l(\wgp)-1}, \baiv_{\wgp,l(\wgp)} = r(\wgp)
 \sen$.
Because $\wgp$ is identified with some $v \in V_n$ with $l(\wgp) = l(v)$,
 for the case $l(\wgp) \ge 2$,
 we can assign each $\baiv_{\wgp,i}$ ($1 \le i \le l(\wgp) - 1$)
 to $\psi^i(B(v)) \subseteq E_{0,\infty}$.
Including the case in which $l(\wgp) = 1$,
 $\baiv_{\wgp,0}$ is assigned to the cluster $U_{n,j}$
 that satisfies $B(v) \subseteq U_{n,j}$, and $\baiv_{\wgp,l(\wgp)}$
 is assigned to 
 the cluster $U_{n,j'}$ that satisfies $\psi^{l(v)}(B(v)) \subseteq U_{n,j'}$.
Thus, for all $n \bni$, each $\baiv \in \baiV_n$ in the basic graph $\baiG_n$
 is assigned to a closed and open
 set of $E_{0,\infty}$, which we denote by $E_{0,\infty}(\baiv)$.
It is clear that if $\baifai_n(\baiv) =\baiu$ ($n \bpi$),
 then $E_{0,\infty}(\baiv) \subseteq E_{0,\infty}(\baiu)$.
We recall that each $x \in X$ is written
 as $x = (\baiv_0,\baiv_1,\baiv_2,\dotsc)$, 
 with $\baifai_n(\baiv_n) = \baiv_{n-1}$ for all $n \bpi$.
Thus, each $x \in X$ defines
 a closed set
 $E_{0,\infty}(x) := \bigcap_{n \bni}E(\baiv_n) \subseteq E_{0,\infty}$.
By the nesting property, it follows that $E_{0,\infty}(x)$
 consists of a single point.
We define $\phi(x) \in E_{0,\infty}$
 to satisfy $\seb \phi(x) \sen = E_{0,\infty}(x)$.
We have defined a map $\phi : X \to E_{0,\infty}$.
Because $E_{0,\infty} = \bigcup_{\baiv \in \baiV_n}E_{0,\infty}(\baiv)$
 is a disjoint union
 for each $n \bni$, $\phi$ is bijective.
Because each $E_{0,\infty}(\baiv)$ $(\baiv \in \baiV_n,\ n \bni)$ is compact,
 the continuity of $\phi$ follows.
Thus, $\phi$ is a homeomorphism.
It is easy to check the commutativity.
Finally, we need to check $\phi(\wgV_{\infty}) = E_{0,\infty,\min}$.
From the definition, we have $\phi(\wgV_{\infty}) \subseteq E_{0,\infty,\min}$.
Let $x = (\baiv_0,\baiv_1,\baiv_2,\dotsc) \notin \wgV_{\infty}$.
Then, there exists some $n \bpi$ with $\baiv_n \notin \wgV_n$.
Thus, it is easy to see that
 $E_{0,\infty}(\baiv_n) \in E_{0,\infty} \setminus E_{0,\infty,\min}$.

It is straightforward to check that
 the closing property of the Bratteli--Vershik models
 brings about the closing property of weighted covering models.
It is also straightforward to check that
 the $\bl$-periodicity-regulation property of the Bratteli--Vershik models
 brings about
 the $\bl$-periodicity-regulation property in weighted covering models.
This completes the proof.
\end{proof}

We note the construction of the clusters in \cref{thm:main-link-BV-to-wg}.
From the construction in \cref{thm:main-link-wg-to-BV,thm:main-link-BV-to-wg},
 a nested Bratteli--Vershik model
 can be translated to a weighted covering model
 and reversely translated to the original Bratteli--Vershik model.
However, a weighted graph covering model may not be reversely translated
 to the original, i.e. some portion of the elements in $\wgV_n$ $(n \bpi)$
 might be separated into multiple vertices.

\begin{thm}\label{thm:closing-implies-basic-set-BV}
A Bratteli--Vershik model $(V,E,\ge,\psi)$
 has the closing property if and only if
 the set $E_{0,\infty,\min}$ is a basic set.
\end{thm}
\begin{proof}
Let $(V,E,\ge,\psi)$ be a Bratteli--Vershik model.
Suppose that it has the closing property.
By \cref{lem:telescoping-implies-nesting-property},
 we obtain a Bratteli--Vershik model with
 the nesting property through telescoping.
It is easy to check that the new Bratteli--Vershik model also 
 has the closing property.
Thus, \cref{thm:main-link-BV-to-wg} implies that 
 the corresponding weighted covering model has the closing property.
By \cref{thm:closing-implies-basic-set-wg},
 it follows that $\wgV_{\infty}$ is a basic set.
Again, \cref{thm:main-link-BV-to-wg} implies that 
 $E_{0,\infty,\min}$ is a basic set.

To show the converse, suppose that $E_{0,\infty,\min}$ is a basic set.
By \cref{lem:telescoping-implies-nesting-property},
 we obtain a Bratteli--Vershik model
 $(V',E',\ge,\psi')$ with the nesting property through telescoping.
It is obvious that both $E_{0,\infty}$ and $E_{0,\infty,\min}$ are identically preserved.
Thus, \cref{thm:main-link-BV-to-wg} implies that
 in the corresponding weighted covering model,
 $\wgV_{\infty}$ is a basic set.
By \cref{thm:closing-implies-basic-set-wg},
 it follows that the weighted covering model has the closing property.
By \cref{thm:main-link-wg-to-BV}, we can reconstruct
 a Bratteli--Vershik model with the closing property.
By the proof of \cref{thm:main-link-wg-to-BV,thm:main-link-BV-to-wg},
 we have recovered $(V',E',\ge,\psi')$.
Thus, this Bratteli--Vershik model has the closing property.
Let $(e_{n+1},e_{n+2},\dotsc) \in E_{n,\infty}$ be a constant path
 of $(V,E,\ge,\psi)$.
By telescoping, we obtain a restricted constant path in $(V',E')$.
By the closing property in $(V',E',\ge,\psi')$, we have a periodic point,
 as desired.
Thus, $(V,E,\ge,\psi)$ has the closing property.
\end{proof}

As we have seen in the two theorems above,
 in establishing a formal link between the Bratteli--Vershik models
 and the three types of graph covering models discussed herein,
 flexible graph covering models are natural.
In \cref{sec:substitution}, we present an example
 in which links are established between flexible graph covering models and
 the Bratteli--Vershik models for the stationary case,
 where we define stationary
 flexible covering models and compare them with the well-known stationary
 ordered Bratteli diagrams.
In \cref{subsec:examples}, we examine the possibility of constructing
 examples of stationary flexible covering models
 in relation with the substitution subshifts.
In general, canonically corresponding substitution subshifts
 might not be topologically conjugate to either
 the stationary flexible covering models
 or the related Bratteli--Vershik models.
Nevertheless, there exist a number of cases
 in which there exist canonical
 isomorphisms.
In these discussions,
 we require the two-sided array system introduced in
 \cite{DOWNAROWICZ_2008FiniteRankBratteliVershikDiagAreExpansive}.
Next, in \cref{subsec:array-systems},
 we recall the array system
 approach for the natural extension of flexible covering models, which is not restricted 
 to stationary cases.
Nevertheless, the arguments are presented
 in terms of weighted covering models.

\subsection{Array systems}\label{subsec:array-systems}
\label{sec:array-systems}
In this subsection,
 following \cite{DOWNAROWICZ_2008FiniteRankBratteliVershikDiagAreExpansive},
 we present the notion of the \textit{array systems} for a brief study
 of the stationary flexible covering models
 described in \cref{sec:substitution}.
First, we translate flexible covering models into the corresponding
 weighted covering models.
Let $\wgGcal : \covrepa{\wgG}{\fai}$ be a weighted covering model.
To construct the inverse limit,
 recall that we translated the weighted covering model into the basic covering model:
 $\baiGcal : \covrepa{\baiG}{\fai}$.
We denoted $\invlim \wgGcal := \invlim \baiGcal$.
Here, we write $\invlim \baiGcal = (X,f)$.
Recall that $f : X \to X$ is continuous and surjective.
The \textit{natural extension} of $(X,f)$
 denotes a zero-dimensional system $(\hX,\hf)$ with
 $\hX := \seb (x_i)_{i \bi} \in X^{\Z} \mid
 f(x_i) = x_{i+1} \myforall i \bi \sen$ and a homeomorphism
 $\hf : \hX \to \hX$ defined by $\hf((x_i)_{i \bi}) = (x_{i+1})_{i \bi}$.

For each $\hx = (x_i)_{i \bi} \in \hX$ and $n \ge 0$,
 there exists a unique $\baiv \in \baiV_n)$ such that $x_i \in U(\baiv)$.
We express this $\baiv$ as $\baiv_{n,i}$.
We define a sequence $\ddx := (\baiv_{n,i})_{n \bni, i \bi}$.
Further, we define
 $\ddX := \seb \ddx \mid \hx \in \hX\sen
 \subset \prod_{n \bni}\left({\baiV_n}^{\Z}\right)$
 and $\ddx(n,i) := \baiv_{n,i}$ for $n \bni$ and $i \bi$.
Then, $\ddX$ is the set of all
 $y \in \prod_{n \bni}\left({\baiV_n}^{\Z}\right)$
 such that, if we write $y = y(n,i)$ with $n \bni,\ i \bi$,
 the following are satisfied:
 $y(n,i) \in \baiV_n$, 
 $\baifai_{n+1}(y(n+1,i)) = y(n,i)$ for all $n \bni, i \bi$,
 and $(y(n,i),y(n,i+1)) \in \baiE_n$ for all $n \bni$ and $i \bi$.
This is a closed condition, and $\ddX$ is thus a compact metrizable
 zero-dimensional space.
Clearly, $\ddf$ is bijective.
Thus, $(\ddX,\ddf)$ is an invertible zero-dimensional system.
We denote $\ddx[n] := (v_{n,i})_{i \bi}$, i.e. the $n$th line.
For an infinite walk $\ddx[n]$ and some integers $a < b$,
 we denote a finite walk
 $(\ddx[n])[a,b] = (\ddx(n,a),\ddx(n,a+1),\dotsc,\ddx(n,b))$ of $\baiG_n$.
In this paper, the invertible zero-dimensional system $(\ddX,\ddf)$
 is called an \textit{array system generated by} $\baiGcal$.
Let $\hx \in \hX$, $n \ge 0$, and $i \bi$.
Now, we consider three cases.

\begin{indentation}{5mm}{0mm}
\noindent Case 1: Suppose that
 there exists some $e \in \wgE_n$ such that $l(e) > 1$ and\\
 $(\ddx(n,i),\ddx(n,i+1)) \in \baiE(e)$.
In this case, such an $e$ is unique and 
 we can denote $\barx(n,i) := e$.
Once an $e \in \wgE_n$ appears with $l(e) > 1$,
 then this $e$ continues at least $l(e)$ times.

\noindent Case 2: 
Suppose that $(\ddx(n,i),\ddx(n,i+1)) \in \baiE(e)$
 for some $e \in \wgE_n$ with
 $l(e) = 1$ and that
 $(\ddx(n+1,i),\ddx(n+1,i+1)) \in \baiE(e_{n+1})$ and $l(e_{n+1}) > 1$.
Then, by taking a factor map $\fai_{n+1,n}(e_{n+1})$ in $\wgGcal$,
 we can determine a unique $e \in \wgE_n$ in a suitable position
 of the walk $\fai_{n+1,n}(e_{n+1})$.
We can denote $\barx(n,i) := e$.

\noindent Case 3: Suppose that
 $(\ddx(n+1,i),\ddx(n+1,i+1)) \in \baiE(e_{n+1})$ and $l(e_{n+1}) = 1$.
We note that $e_{n+1}$ are not identified uniquely.
Nevertheless, $\ddx(n+1,i) = v_{n+1} \in \wgV_{n+1}$ is identified uniquely.
Thus, by the \pdirectionalitys condition, $e = \fai_{n+1}(e_{n+1})$
 is identified uniquely.
We denote $\barx(n,i) := e$.
\end{indentation}

\noindent
Thus, we have defined $\barx(n,i)$ for all $n \bni$ and all $i \bi$.
We define a sequence
 $\barx[n] := (\dotsc, \barx(n,-1), \barx(n,0), \barx(n,1), \dotsc)$.
For integers $n \ge 0$ and $s < t$,
 we denote $(\barx[n])[s,t] :=
 (\barx(n,s),\barx(n,s+1),\barx(n,s+2),\dotsc, \barx(n,t))$.
For each $n \ge 0$,
 the lines $\barx[n]$ are compiled
 (see \cref{fig:array-system,fig:array-system-2}).
Following \cite{DOWNAROWICZ_2008FiniteRankBratteliVershikDiagAreExpansive},
 we make an $n$-\textit{cut} in each $\barx[n]$
 just before each $i$ with
 $\ddx(n,i) = s(\barx(n,i))$
 (see \cref{fig:array-system,fig:array-system-2}).
We recall the notation of the singleton weighted graph
 $\wgG_0 = (\seb v_0 \sen, \seb e_0 \sen)$.
Thus, for each $\hx \in \hX$, $\barx[0] = (\dotsc,e_0,e_0,e_0,\dotsc)$
 that is cut everywhere.
We define $\barX := \seb \barx \mid x \in X \sen$.
Because there exists a bijective factoring map
 from $\ddX$ onto $\barX$,
 $\barX$ is also a compact metrizable zero-dimensional 
 space.
We define the shift map $\barf : \barX \to \barX$ that shifts left.
Thus, we have an invertible zero-dimensional system $(\barX,\barf)$,
 which we call
 an \textit{array system generated by $\wgGcal$}.
If it is necessary to distinguish the beginning of the edge,
 this can be done by changing
 $\barx(n,i) = e$ to $\barx(n,i) = \dot{e}$
 for all $i$ with $\ddx(n,i) = s(\barx(n,i))$.
The next theorem is obvious.
\begin{thm}
Let $\wgGcal : \covrepa{\wgG}{\fai}$ be a weighted covering model.
We write $\invlim \wgGcal = (X,f)$.
Then, the natural extension $(\hX,\hf)$ defined above is topologically 
 conjugate to the array systems $(\ddX,\ddf)$ and $(\barX,\barf)$.
\end{thm}

\begin{figure}
\begin{center}\leavevmode
\xy
(-5,18)*{}; (105,18)*{} **@{-},
 (49,15)*{e_{n,1}
 \hspace{9mm} e_{n,3} \hspace{4mm}
 \hspace{12mm} e_{n,1} \hspace{6mm}
 \hspace{8mm} e_{n,3} \hspace{3mm}
 \hspace{10mm} e_{n,2} \hspace{4mm}
 \hspace{8mm} e_{n,1}
 \hspace{1mm} },
(7,18)*{}; (7,12)*{} **@{-},
(28,18)*{}; (28,12)*{} **@{-},
(49,18)*{}; (49,12)*{} **@{-},
(70,18)*{}; (70,12)*{} **@{-},
(84,18)*{}; (84,12)*{} **@{-},
(-5,12)*{}; (105,12)*{} **@{-},
 (54,9)*{e_{n+1,5}
 \hspace{35mm} e_{n+1,1} \hspace{3mm}
 \hspace{28mm} e_{n+1,3} \hspace{2mm} },
(28,12)*{}; (28,6)*{} **@{-},
(84,12)*{}; (84,6)*{} **@{-},
(-5,6)*{}; (105,6)*{} **@{-},
\endxy
\end{center}
\caption{$n$th and $(n+1)$th rows of an array system with cuts.}
\label{fig:array-system}
\end{figure}
For an interval $[n,m]$ with $m > n \ge 0$, the combination of rows $\barx[n']$
 with $n \le n' \le m$ is denoted as $\barx[n,m]$.
For each $\hx \in \hX$, 
 we obtain an infinite combination $\barx[0,\infty) \in \hX$
 of rows $\barx[n]$ for all $0 \le n < \infty$
 (see \cref{fig:array-system,fig:array-system-2}).
Note that for $m > n \ge 0$, if there exists an $m$-cut at position $i$
 (i.e. just before position $i$), then
 there exists an $n$-cut at position $i$.

\begin{figure}
\begin{center}\leavevmode
\xy
(-5,30)*{}; (105,30)*{} **@{-},
 (49,27)*{e_0\phantom{{}_{{},f}}\ e_0\phantom{{}_{{},f}}
 \ e_0\phantom{{}_{{},l}}
 \ e_0\phantom{{}_{{},l}}\ e_0\phantom{{}_{{},l}}
 \ e_0\phantom{{}_{{},f}}\ e_0\phantom{{}_{{},l}}
 \ e_0\phantom{{}_{{},l}}\ e_0\phantom{{}_{{},l}}
 \ e_0\phantom{{}_{{},f}}\ e_0\phantom{{}_{{},l}}
 \ e_0\phantom{{}_{{},l}}\ e_0\phantom{{}_{{},l}}
 \ e_0\phantom{{}_{{},f}}\ e_0\phantom{{}_{{},l}}\ e_0},
(0,30)*{}; (0,24)*{} **@{-},
(7,30)*{}; (7,24)*{} **@{-},
(14,30)*{}; (14,24)*{} **@{-},
(21,30)*{}; (21,24)*{} **@{-},
(28,30)*{}; (28,24)*{} **@{-},
(35,30)*{}; (35,24)*{} **@{-},
(42,30)*{}; (42,24)*{} **@{-},
(49,30)*{}; (49,24)*{} **@{-},
(56,30)*{}; (56,24)*{} **@{-},
(63,30)*{}; (63,24)*{} **@{-},
(70,30)*{}; (70,24)*{} **@{-},
(77,30)*{}; (77,24)*{} **@{-},
(84,30)*{}; (84,24)*{} **@{-},
(91,30)*{}; (91,24)*{} **@{-},
(98,30)*{}; (98,24)*{} **@{-},
(105,30)*{}; (105,24)*{} **@{-},
(-5,24)*{}; (105,24)*{} **@{-},
 (48,21)*{e_{1,3}
 \hspace{11mm} e_{1,3} \hspace{8mm}
 \hspace{8mm} e_{1,1} \hspace{10mm}
 \hspace{5mm} e_{1,3} \hspace{8mm}
 \hspace{4mm} e_{1,2} \hspace{8mm}
 \hspace{3mm} e_{1,1}
 \hspace{0mm} },
(7,24)*{}; (7,18)*{} **@{-},
(28,24)*{}; (28,18)*{} **@{-},
(49,24)*{}; (49,18)*{} **@{-},
(70,24)*{}; (70,18)*{} **@{-},
(84,24)*{}; (84,18)*{} **@{-},
(-5,18)*{}; (105,18)*{} **@{-},
(56,15)*{e_{2,3} \hspace{0.7mm}\
 \hspace{35mm} e_{2,1} \hspace{27mm}
 \hspace{11mm} e_{2,3}
 \hspace{2mm} },
(28,18)*{}; (28,12)*{} **@{-},
(84,18)*{}; (84,12)*{} **@{-},
(-5,12)*{}; (105,12)*{} **@{-},
(36,9)*{e_{3,3} \hspace{26mm}
 \hspace{15mm} e_{3,1}},
(28,12)*{}; (28,6)*{} **@{-},
(-5,6)*{}; (105,6)*{} **@{-},
(50,4)*{\vdots},
\endxy
\end{center}
\caption{First four rows of an array system.}\label{fig:array-system-2}
\end{figure}

\begin{nota}\label{nota:n-symbol}
For each edge $e \in \wgE_n$, if we write
$\fai_n(e) = a_1 a_2 \dotsb a_{w(e)}$ with $a_j \in \wgE_{n-1}$
 ($1 \le j \le w(e)$) as 
 a walk in $\wgG_{n-1}$,
 then each $a_j$ similarly determines a walk of $\wgG_{n-2}$.
Thus, we can determine a set of circuits arranged in a square form, as in
 \cref{fig:2-symbol}.
This form is said to be the \textit{$n$-symbol} and is denoted by $e$.
For $0 \le m < n$,
 the projection $e[m]$ that is a finite sequence of circuits of $\wgG_m$
 is also defined.
\end{nota}

\begin{figure}
\begin{center}\leavevmode
\xy
(28,38)*{}; (84,38)*{} **@{-},
(57,35)*{e_0\phantom{{}_{{},l}}\ e_0\phantom{{}_{{},l}}
 \ e_0\phantom{{}_{{},l}}\ e_0\phantom{{}_{{},l}}
 \ e_0\phantom{{}_{{},f}}\ e_0\phantom{{}_{{},l}}
 \ e_0\phantom{{}_{{},l}}\ e_0\phantom{{}_{{},l}}},
(28,38)*{}; (28,32)*{} **@{-},
(35,38)*{}; (35,32)*{} **@{-},
(42,38)*{}; (42,32)*{} **@{-},
(49,38)*{}; (49,32)*{} **@{-},
(56,38)*{}; (56,32)*{} **@{-},
(63,38)*{}; (63,32)*{} **@{-},
(70,38)*{}; (70,32)*{} **@{-},
(77,38)*{}; (77,32)*{} **@{-},
(84,38)*{}; (84,32)*{} **@{-},
(28,32)*{}; (84,32)*{} **@{-},
 (60,29)*{
 \hspace{7mm} e_{1,1} \hspace{9mm}
 \hspace{5mm} e_{1,3} \hspace{7mm}
 \hspace{3mm} e_{1,2} \hspace{8mm}
 \hspace{2mm} },
(28,32)*{}; (28,26)*{} **@{-},
(49,32)*{}; (49,26)*{} **@{-},
(70,32)*{}; (70,26)*{} **@{-},
(84,32)*{}; (84,26)*{} **@{-},
(28,26)*{}; (84,26)*{} **@{-},
(56,23)*{
 \hspace{23mm} e_{2,1} \hspace{23mm}},
(28,26)*{}; (28,20)*{} **@{-},
(84,26)*{}; (84,20)*{} **@{-},
(28,20)*{}; (84,20)*{} **@{-},
\endxy
\end{center}
\caption{$2$-symbol corresponding to the edge $e_{2,1}$
 of \cref{fig:array-system-2}.}\label{fig:2-symbol}
\end{figure}
It is clear that $\barx[n] = \barx'[n]$
 implies $\barx[0,n] = \barx'[0,n]$.
If $x \ne x'$ $(x, x' \in X)$,
 then there exists some $n > 0$ with $x[n] \ne x'[n]$.
For $x, x' \in X$,
 we say that the pair $(x,x')$ is \textit{$n$-compatible}
 if $x[n] = x'[n]$.
If $x[n] \ne x'[n]$,
 then we say that $x$ and $x'$ are \textit{$n$-separated}.
We recall that,
 if there exists an $n$-cut at position $k$,
 then there exists an $m$-cut at position
 $k$ for all $0 \le m \le n$.
The set $\barX_n := \seb \barx[n] \mid x \in X \sen$ is a two-sided subshift
 of a finite set
 $\wgE_n \cup \seb \dot{e} \mid e \in \wgE_n \sen$.
The factoring map is denoted by $\pi_n : \barX \to \barX_n$,
 and the shift map is always denoted by $\barT : \barX_n \to \barX_n$.
For $m > n \ge 0$, the factoring map $\pi_{m,n} : \barX_m \to \barX_n$
 is defined by $\pi_{m,n}(\barx[m]) = \barx[n]$ for all $x \in X$.

\section{Stationary models and substitution subshifts}
\label{sec:substitution}
In this section, we present some examples of stationary graph covering models.
In \cite{DURAND_1999SubstDynSysBratteliDiagDimGroup},
 Durand, Host, and Skau showed
 the relation between the stationary Bratteli diagrams and
 primitive substitution systems.
Further, in \cite{BEZUGLYI_2009AperioSubstSysBraDiag},
 Bezuglyi, Kwiatkowski, and Medynets
 extended these results to aperiodic substitution systems.
In this paper, we will not conduct such a detailed investigation.
Without thinking of topological conjugacies, we treat only the formal relation
 between substitution subshifts and
 stationary graph covering models.
The resulting array systems have to be factored to the first lines
 $(\barX_1,\barT)$
 to produce the intended substitution subshifts.

\subsection{Stationary graph covering models}
\label{subsec:stationary-systems}
In this subsection, as an analogue of the stationary Bratteli--Vershik models,
 we define stationary graph covering models.
We say that a flexible cover $\fai : \fgG \to \fgG$
 is a flexible \textit{self-cover}.
\begin{defn}[Stationary flexible graph covering]
\label{defn:stationary-flexible-graph-covering}
Suppose that there exists a flexible graph $\fgG = (\fgV,\fgE)$
 and a flexible self-cover $\fai : \fgG \to \fgG$.
Suppose that there exists a sequence of positive integers
 $n(e)\  (e \in \fgE)$.
Define $\fgG_0$ to be the singleton graph $(\seb v_0 \sen, \seb e_0 \sen)$,
 $\fgG_n := \fgG$ for
 all $n \bpi$, $\fai_n := \fai$ for all $n \ge 2$,
 and $\fai_1$
 to be the unique natural homomorphism such that, for each $e \in \fgG_1$,
 $\fai_1(e) = e_0^{n(e)} := \underbrace{e_0\ e_0\ \dotsb\ e_0}_{n(e)}$.
We say that $\fgGcal : \covrepa{\fgG}{\fai}$ is a
 \textit{stationary flexible graph covering model generated by}
 a flexible self-cover $\fai: \fgG \to \fgG$ and
 a sequence $n(e)\ (e \in \fgE)$.
\end{defn}

\begin{rem}
We remark that the stationary ordered Bratteli diagrams
 might not have continuous Vershik maps
 (see the example after
 \cite[Proposition 2.5.]{Medynets_2006CantorAperSysBratDiag}).
However, by the \pdirectionalitys condition of a flexible cover
 $\fai : \fgG \to \fgG$, stationary flexible graph covering models
 always give continuous zero-dimensional systems.
\end{rem}

Let $\fai : \fgG \to \fgG$ be a flexible cover and $\fgG = (\fgV,\fgE)$.
Because $\fai : \fgV \to \fgV$ is a map, for each $v \in \fgV$, there
 exists a positive integer $K(v)$ such that the sequence
 $v, \fai(v), \fai^2(v), \dotsc $ is eventually periodic with least
 period $K(v)$.
We write $K(\fgG,\fai) := \lcm \seb K(v) \mid v \in \fgV \sen$.
Then, we have a flexible cover $\fai' : \fgG \to \fgG$ defined by
 $\fai' := \fai^{K(\fgG,\fai)}$.
For each $v \in \fgV$,
 the sequence $v, \fai'(v), {\fai'}^2(v), \dotsc $ is eventually constant.
Let $\lim v \in \fgV$ be the constant vertex, i.e.
 there exists some $n$ such that
 ${\fai'}^n(v) = \lim v$ and $\fai'(\lim v) = \lim v$.
If we take sufficiently large $L >0$,
 then for all $v \in \fgV$, it follows that ${\fai'}^L(v) = \lim v$.
Thus, by taking $K := L \cdot K(\fgG,\fai)$ and defining $\fai'' := \fai^{K}$,
 it follows that, for all $v \in \fgV$, $\fai''(v) = \lim v$
 and $\fai''(\lim v) = \lim v$.
To simplify the formulation further,
 we assume that $\fai(v) = \lim v$ for all $v \in \fgV$.
Let $e \in \fgE$
and set $v = s(e)$.
By the \pdirectionalitys of $\fai$,
 we obtain a unique
 $\fai(e)(\min) \in s^{-1}(\fai(v)) = s^{-1}(\lim v)$
 that is independent of the choice of $e \in s^{-1}(v)$.
With the condition $\fai(\lim v) = \lim v$, we apply $\fai$
 onto all $s^{-1}(\lim v)$, and we obtain a unique $\limf e \in s^{-1}(\lim v)$,
 i.e. we denote $\limf e := \fai( \fai(e)(\min) )(\min)$.
Thus, $\limf e$ is uniquely determined by $\lim v \in \fgV$.
Taking $\fai^2$ instead of $\fai$,
 we assume that for each $e \in \fgE$,
 $\fai(e)(\min) = \limf e$.
In particular,
 we have that $\fai(\limf e)(\min) = \limf e$.
Let $v \in \fgV$ and $e = r^{-1}(v)$.
If we consider only \bidirectionals flexible self-covers,
 then the same argument
 is possible and we obtain 
 $\liml e$ as well.
Suppose that the \bidirectionalitys condition does not hold.
We define a map $\rho : \fgE \to \fgE$ by
 $\rho(e) = \fai(e)(\max)$ for all $e \in \fgE$.
Then, the sequence $e, \rho(e), \rho^2(e), \dotsc$ is eventually periodic
 and the periodic edges are in $r^{-1}(\lim r(e))$.
Let $K'(e) \ge 1$ be the least period.
By defining $K' := \lcm \seb K'(e) \mid e \in \fgE \sen$
 and taking sufficiently
 large $L' \ge 1$, we define $\barfai := \fai^{L'K'}$.
We define $\barrho : \fgE \to \fgE$
 by $\barrho(e) := \barfai(e)(\max)$ for all $e \in \fgE$.
Then, for each $e \in \fgE$, the sequence
 $\barrho(e), \barrho^2(e), \dotsc$ is constant from the beginning.
This constant edge is denoted as $\liml e$.
For each $e \in \fgE$, we have that $\fai(e)(\max) = \liml e$
 and $\fai(\liml e)(\max) = \liml e$.
Note that, for each $v \in \Vcal$,
 the set $\seb \liml e \mid e \in \fgE \myand \lim r(e) = \lim v \sen$
 might have more than one element.
\begin{defn}\label{defn:straight-flexible-cover}
A flexible self-cover $\fai : \fgG \to \fgG$ is \textit{straight} 
 if, for each $v \in \fgV$, there exists $\lim v \in \fgV$ such that
 $\fai(v) = \lim v$ and $\fai(\lim v) = \lim v$, and for each $e \in \fgE$,
 there exist
 $\limf e \in s^{-1}(\lim s(e))$ and
 $\liml e \in r^{-1}(\lim r(e))$
 such that
 $\fai(e)(\min) = \limf e$,
 $\fai(e)(\max) = \liml e$,
 $\fai(\limf e)(\min) = \limf e$, and
 $\fai(\liml e)(\max) = \liml e$.
We define $\lim \fgV := \seb \lim v \mid v \in \fgV \sen$,
 $\limf \fgE := \seb \limf e \mid e \in \fgE \sen$, and
 $\liml \fgE := \seb \liml e \mid e \in \fgE \sen$.

A stationary flexible graph covering model is \textit{straight}
 if it is generated by a straight flexible self-cover.
\end{defn}

By the \pdirectionalitys condition,
 for each $v \in \lim \fgV$, there exists a unique $e \in \limf \fgE$
 with $s(e) = v$.

\begin{rem}
Let $\fgGcal : \covrepa{\fgG}{\fai}$ be a stationary flexible covering model
 generated by a straight flexible self-cover $\fai : \fgG \to \fgG$.
Consider the corresponding
 weighted covering model $\wgGcal : \covrepa{\wgG}{\fai}$.
Suppose that $e \ne e' \in \liml \fgE$ satisfy $r(e) = r(e')$
 and $e_n, e'_n \in \wgE_n$ are copies of them.
Because $e, e' \in \liml \fgE$,
 we have $\fai_{n+1}(e_{n+1})(\max) = e_n$ and
 $\fai_{n+1}(e'_{n+1})(\max) = e'_n$ for all $n \bpi$.
If $l(e_n) = l(e'_n) = 1$ for all $n \bpi$,
 then it is easy to see that $s(e_n) \ne s(e'_n)$ for all $n \bpi$.
\end{rem}

\subsection{Substitutions}
\label{subsec:substitutions}

In this subsection, similar to 
 Durand, Host, and Skau \cite{DURAND_1999SubstDynSysBratteliDiagDimGroup}
 and Bezuglyi, Kwiatkowski, and Medynets
 \cite{BEZUGLYI_2009AperioSubstSysBraDiag},
 we follow a
 standard method of introduction for the theory of substitution systems.
Further, we make a link with the stationary
 flexible covering models.
Denote a finite alphabet as $A$ and let $A^+$ be the set of all non-empty words
 over $A$.
Let $\sigma : A \to A^+$ be a map.
Let $A_l$ be the set of all letters $a \in A$
 such that $\abs{\sigma^n(a)} \to \infty$ as $n \to \infty$.
Let $A_s = A \setminus A_l$.
We say that a map $\sigma$ is a substitution if $A_l \nekuu$.
We set $A^* := A^+ \cup \seb \ew \sen$, where $\ew$ is the empty word
 of length $0$.
Given a word $u = u_1 u_2 \dotsb u_m$ and an interval $[i,j] \subseteq [1,m]$,
 we write $u[i,j]$ to denote the sub-word $u_i u_{i+1} \dotsb u_j$.
We extend this notation in the obvious way to infinite intervals.
Suppose that $v = v_1 v_2 \dotsb v_m$ is a word.
Then, a word $u = u_1 u_2 \dotsb u_n$ is a \textit{factor} of $v$
 if there exists $[i,j] \subseteq [1,m]$ such that $u = v[i,j]$.
If $u$ is a factor of $v$, then we write $u \prec v$.

It is natural to make a link with flexible self-covers.

\begin{defn}\label{defn:substitution-read-flexible-cover}
Let $\fgG = (\fgV, \fgE)$ be a flexible graph,
 and $\fai : \fgG \to \fgG$ be a flexible self-cover.
Let $\iota : A \to \fgE$ be a bijection.
For each $e \in \fgE$, we have a walk $\fai(e) = e_1 e_2 \dotsb e_k$
 with $e_i \in \fgE$ ($1 \le i \le k$).
Let $\iota(a) = e$ and $\iota(u_i) = e_i$.
Then, we have $\sigma(a) = u_1 u_2 \dotsb u_k$
for each $a \in A$.
If $A_l \nekuu$, then
 we obtain a substitution $\sigma$ called a \textit{substitution read on} the
 flexible self-cover $\fai : \fgG \to \fgG$.
\end{defn}

Note that not all substitution maps are generated by flexible covers in this way.

\begin{defn}
Let $\sigma : A \to A^+$ be a substitution
 with a letter $a \in A$
 such that the length $\abs{\sigma^n(a)} \to \infty$ as $n \to \infty$.
Let $\sL(\sigma)$ denote the \textit{language of} $\sigma$, i.e.
 $\sL(\sigma)$
 is the set of all words on $A$ that are factors of $\sigma^n(a)$
 for some $a \in A$ and some $n \bpi$.
Let $X_{\sigma}$ denote the subshift of $A^{\Z}$ associated with this language,
 i.e. the set of all $x \in A^{\Z}$ in which every finite factor belongs to
 $\sL(\sigma)$.
It follows that $X_{\sigma}$ is a non-empty closed set in $A^{\Z}$ and
 is invariant under the shift. We denote the restriction of
 the shift to $X_{\sigma}$ as $T$, i.e. $(T(x)_i) = x_{i+1}$ ($i \bi$)
 for all $x \in X_{\sigma} \subseteq A^{\Z}$.
The invertible zero-dimensional system
 $(X_{\sigma}, T)$ is called the
 \textit{substitution subshift associated with} $\sigma$.
\end{defn}

\begin{defn}
Let $\fai : \fgG \to \fgG$ be a flexible self-cover and $\sigma : A \to A^+$
 be the substitution read on the flexible self-cover.
We assume that $A_l \nekuu$.
Define the sequence $n(e) = 1\  (e \in \fgE)$.
Suppose that the stationary flexible covering model
 $\fgGcal : \covrepa{\fgG}{\fai}$
 generated by $\fai: \fgG \to \fgG$ and
 the sequence $n(e) = 1\ (e \in \fgE)$
 has the inverse limit $\invlim \fgGcal = (X,f)$.
Each walk $w$ on $\fgG_1$ or $\fgG$
 is considered to be a word by the identification $\iota$, i.e.
 we define $\iota^{-1}(w) \in A^+$.
Further, for each $\barx[1] \in \barX_1$ and $s < t$,
 we define $\iota^{-1}(\barx[1][s,t]) \in A^+$.
We also define $\iota^{-1}(\barx[1]) \in A^{\Z}$.
\end{defn}

\begin{lem}\label{lem:natural-injection-from-substitution-system-to-X1}
Let $\fai : \fgG \to \fgG$ be a flexible self-cover and $\sigma : A \to A^+$
be the substitution read on the flexible self-cover.
We assume that $A_l \nekuu$.
Define the sequence $n(e) = 1\  (e \in \fgE)$.
Suppose that the stationary flexible covering model
 $\fgGcal : \covrepa{\fgG}{\fai}$
 generated by $\fai: \fgG \to \fgG$ and
 the sequence $n(e) = 1\ (e \in \fgE)$
 has the inverse limit $\invlim \fgGcal = (X,f)$.
Then, there exists a canonical injection
 $(X_{\sigma},T) \to (\barX_1,\barT)$.
\end{lem}
\begin{proof}
We recall that there exists a corresponding weighted graph covering model
 $\wgGcal : \covrepa{\wgG}{\fai}$.
Because $n(e) = 1\  (e \in \fgE)$, each edge in $\wgG_1$ has length $1$.
Thus, $\barX_1$ is a sequence of elements of $\fgE$.
We show that, for every word $w \in \sL_{\sigma}$,
 there exists some $\barx[1] \in \barX_1$ and $s < t$
 such that $w = \iota^{-1}(\barx[1][s,t])$.
Then, for every $x \in X_{\sigma}$ and $s < t$,
 there exists some $\barx_1 \in \barX_1$ and $s' < t'$ such that
 $x[s,t] = \barx_1[s',t']$, implying that $X_{\sigma} \subseteq \barX_1$.
Let $w \in \sL_{\sigma}$.
Then, there exists some $a \in A$ and $n \bpi$ such that 
 $w$ is a factor of $\sigma^n(a)$.
Take $e = \iota(a)$ from $\fgE_{n+1} = \fgE$.
We have the corresponding $e_{n+1} \in \wgE_{n+1}$.
We consider $e_{n+1}$ as an $(n+1)$-symbol.
Then, we have that $e_{n+1}[1] = \sigma^n(a)$ with the identification $\iota$.
By the edge surjectivity of $\fai_i$ ($i \bpi$),
 it is obvious that there exists some $\barx \in \barX$ such that
 $\barx[n+1]$ contains an edge corresponding to $\iota(a)$.
Then, $\barx[1]$ contains $w$, as desired.
\end{proof}

\begin{nota}\label{nota:iota-infty}
The injection that we have obtained by
 \cref{lem:natural-injection-from-substitution-system-to-X1}
 is also denoted as
 $\iota_{\infty} : (X_{\sigma},T) \to (\barX_1,\barT)$.
\end{nota}

Thus, to express the substitution subshifts by some stationary models,
 it is convenient if $\iota_{\infty}$ is bijective.
For this, we obtained \cref{prop:iota-infty-to-be-bijective}.
We begin by stating a lemma that is somewhat trivial.
Nevertheless, because we know the relations among three types of
 graph covering models only abstractly,
 we describe this lemma in detail, concretely.  

\begin{lem}\label{lem:last-limvertex-first-implies-barx}
Let $\fgGcal : \covrepa{\fgG}{\fai}$ be a
 straight stationary flexible graph covering model generated by
 a straight flexible self-cover $\fai: \fgG \to \fgG$ and
 a sequence $n(e)=1$ for all $e \in \fgE$.
For every $v \in \lim \fgV$
 and for every pair $(e_1,e_2) \in (\liml \fgE) \times (\limf \fgE)$
 with $r(e_1) = v = s(e_2)$,
 there exists some $\barx[1] \in \barX_1$ such that $e_1 e_2$ is a factor
 of $\barx[1]$.
\end{lem}
\begin{proof}
From the flexible graph covering model $\fgGcal : \covrepa{\fgG}{\fai}$,
 we obtain a weighted graph covering model $\wgGcal : \covrepa{\wgG}{\fai}$.
From the weighted graph covering model $\wgGcal : \covrepa{\wgG}{\fai}$,
 we obtain a basic graph covering model $\baiGcal : \covrepa{\baiG}{\baifai}$.
We have defined
 $\invlim \fgGcal = \invlim \wgGcal = \invlim \baiGcal$,
 which we denote as $(X,f)$.
From $(X,f)$, we consider the natural extension $(\hX,\hf)$.
Finally, from $(\hX,\hf)$, $\baiGcal$, and $\wgGcal$ (or $\fgGcal$),
 we construct $(\barX,\barf)$.
Let us write $\wgG_n = (\wgV_n,\wgE_n)$ for each $n \bni$.
For each $n \bpi$, there exist
 $v_n \in \wgV_n$ and $e_{n,1}, e_{n,2} \in \wgE_n$,
 which are the copies
 of $v \in \lim \fgV$, $e_1 \in \lim_l \fgE$, and $e_2 \in \lim_f \fgE$,
 respectively.
Let $l_{n,1} = l(e_{n,1})$ and $l_{n,2} = l(e_{n,2})$.
In \cref{nota:basic-graph,lem:weighted-cover-to-basic-cover,nota:from-weighted-to-basic-covering},
 we transformed $\wgGcal$ into $\baiGcal$.
We constructed vertices of the basic graph $\baiG_n = (\baiV_n,\baiE_n)$
 as follows:
 for each $e \in \wgE_n$, we construct
 $\baiV(e) :=
 \seb \baiv_{e,0} = s(e),
 \baiv_{e,1}, \baiv_{e,2}, \dotsc, \baiv_{e,l(e)-1},
 \baiv_{e,l(e)} = r(e) \sen$.
For each $e \in \wgE_n$,
 we produced the edges of the basic graph as
 $\baiE(e) := \seb (\baiv_{e,i}, \baiv_{e,i+1}) \mid 0 \le i < l(e)\sen$.
Thus, each $e_{n,1}, e_{n,2}$ can be described as a walk
 $(\baiv_{n,i,0} = s(e_{n,i}), \baiv_{n,i,1},
 \baiv_{n,i,2}, \dotsc, \baiv_{n,i,l_{n,i}-1},
 \baiv_{n,i,l_{n,i}} = r(e_{n,i}))$ in $\baiG_n$ for $i = 1,2$.
Because
 $v \in \lim \fgV$, $(e_1,e_2) \in (\liml \fgE) \times (\limf \fgE)$
 and $r(e_1) = v = s(e_2)$,
 we have
 $\fai_{n+1}(v_{n+1}) = v_n$,
 $\baiv_{n,1,l_{n,1}} = v_n = \baiv_{n,2,0}$,
 $(\fai_{n+1}(e_{n+1,1}))(\max) = e_{n,1}$, and
 $(\fai_{n+1}(e_{n+1,2}))(\min) = e_{n,2}$.
Thus, we can define three points $p_1,p_2,p_3 \in X$ as
 $p_1 = (\baiv_0, \baiv_{1,1,l_{1,1}-1}, \baiv_{2,1,l_{2,1}-1},
 \baiv_{3,1,l_{3,1}-1},\dotsc)$,
 $p_2 = (\baiv_0, \baiv_1, \baiv_2, \baiv_3, \dotsc)$, and
 $p_3 = (\baiv_0, \baiv_{1,2,1}, \baiv_{2,2,1}, \baiv_{3,2,1}, \dotsc)$.
Because $(\baiv_{n,1,l_{n,1}-1}, v_n)$ and $(v_n, \baiv_{n,2,1})$
 are edges of $\baiG_n$ for all $n \bpi$,
 we have that $f(p_1) = p_2$ and $f(p_2) = p_3$.
Because $v_n$ is also considered to be a vertex in $\wgV_n$,
 at any level $n$, a unique tower of level $n$ never passes all of
 $p_1,p_2$, and $p_3$ simultaneously.
In the form of the array system $(\ddX,\ddf)$ (see \cref{sec:array-systems}),
 intuitively, there seems to be an infinite cut between $p_1$ and $p_2$.
Concretely, in $\ddX$, we can define $\ddx \in \ddX$ such that
 $\ddx[n][-l_{n,1},0]
 = (\baiv_{n,1,0}, \baiv_{n,1,1}, \baiv_{n,1,2},
 \dotsc, \baiv_{n,1,l_{n,i}})$ and
 $\ddx[n][0,l_{n,2}]
 = (\baiv_{n,2,0}, \baiv_{n,2,1}, \baiv_{n,2,2},
 \dotsc, \baiv_{n,2,l_{n,2}})$.
Even if $l_{n,i} = 1$ for some $i = 1,2$,
 by the definition of $\barx$,
 we can conclude that $\barx(1,-1) = e_{1,1}$ and $\barx(1,0) = e_{1,2}$.
Thus, we obtain $\barx[1][-1,0] = e_1 e_2$ by the identification
 $\fgG = \wgG_1$.
\end{proof}

Suppose that  $\iota_{\infty}$ is surjective.
Then, for every such $v,\ e_1$, $e_2$ as in the above lemma,
 if we define $\iota(a_i) = e_i$ $(i = 1,2)$,
 it follows that $a_1 a_2 \in \sL(\sigma)$. 
In addition, if $\barx(1,s) = e_1,\ \barx(1,s+1) = e_2$,
 then for any $L_1 > 0$ and $L_2 > 0$,
 it follows that
 $\barx[1][s-L_1,s+1+L_2] \in \sL(\sigma)$ by the identification
 with $\iota$, i.e.
 there must exist $a \in A$ and $n \bpi$ such that
 $\iota^{-1}(\barx[1][s-L_1,s+1+L_2]) \prec \sigma^n(a)$.
From this observation, we can state the following:
\begin{lem}\label{lem:iota-infty-is-surjective-then-limfgV-is-in-Midfai}
Let $\fai : \fgG \to \fgG$ be a straight flexible self-cover,
 and $\sigma : A \to A^+$
be the substitution read on the flexible self-cover.
We assume that $A_l \nekuu$.
Suppose that $\iota_{\infty}$ is surjective.
It then follows that, for every $v \in \lim \fgV$,
 $e_1 \in \liml \fgE$, and $e_2 \in \limf \fgE$ with
 $r(e_1) = v = s(e_2)$,
 there exist some $e \in \fgE$ and $n > 0$
 such that $e_1 e_2$ is a sub-walk of $\fai^n(e)$.
\end{lem}
\begin{proof}
We omit the proof because we have already seen the proof just before the statement of the lemma.
\end{proof}

To obtain a sufficient condition for $\iota_{\infty}$ to be surjective,
 we need the following definition:
\begin{defn}
Let $\fai : \fgG \to \fgG$ be a straight flexible self-cover.
A finite sequence $(v_1,v_2,\dotsc,v_k)$
 with $v_i \in \lim \fgV$ $(1 \le i \le k)$
 is \textit{constant} if
 there exists a sequence $e_1,e_2,\dotsc,e_{k-1} \in \fgE$
 such that $v_i = s(e_i)$ and $r(e_i) = v_{i+1}$ for all $1 \le i < k$,
 and the walk $w := (e_1,e_2,\dotsc,e_{k-1})$ satisfies
 $\fai(w) = w$.
Note that these $e_i$ do not change length, even when
 they are realised in $\wgG_n$ for $n > 1$.
Further, we establish a convention that a sequence $(v)$ (of length $1$)
 with $v \in \lim \fgV$ is constant in the above definition.
\end{defn}

\begin{rem}
Let $\fai : \fgG \to \fgG$ be a straight flexible self-cover and
 $(v_1,v_2,\dotsc,v_k)$ be a constant sequence.
By definition, there exists a walk $w = (e_1,e_2,\dotsc,e_{k-1})$ on $\fgG$
 such that $v_i = s(e_i)$ and $r(e_i) = v_{i+1}$ for all $1 \le i < k$
 and $\fai(w) = w$.
Because of the \pdirectionalitys condition, the walk $w$ has to be unique.
\end{rem}

\begin{defn}
Let $(v_1,v_2,\dotsc,v_k)$ be a constant sequence.
Let $w = (e_1,e_2,\dotsc,e_{k-1})$ be the unique walk on $\fgG$
 with $v_i = s(e_i)$ and $r(e_i) = v_{i+1}$ for all $1 \le i < k$ and
 $\fai(w) = w$.
Then, we say that $(v_1,v_2,\dotsc,v_k)$ is \textit{overlapped}
 if, for every $e_0 \in \liml \fgE$ with $r(e_0) = v_1$ and unique
 $e_k \in \limf \fgE$ with $s(e_k) = v_k$,
 there exists some $e \in \fgE$ and $n \bpi$ such that
 $\fai^n(e)$ has a sub-walk $w' = (e_0,e_1,e_2, \dotsc,e_k)$.
\end{defn}

The next lemma provides a necessary condition for $\iota_{\infty}$ to be
 surjective.

\begin{lem}\label{lem:surjective-implies-overlap}
Let $\fai : \fgG \to \fgG$ be a straight flexible self-cover.
Suppose that $\iota_{\infty}$ is surjective.
Then, every constant sequence $(v_1,v_2,\dotsc,v_k)$ is overlapped.
\end{lem}

\begin{proof}
Let $(v_1,v_2,\dotsc,v_k)$ be a constant sequence.
In \cref{lem:iota-infty-is-surjective-then-limfgV-is-in-Midfai},
 we have already proved the case in which $k = 1$.
We assume that $k \ge 2$.
Let $w = (e_1,e_2,\dotsc,e_{k-1})$ be a unique walk on $\fgG$ such that
 $v_i = s(e_i)$ and $r(e_i) = v_{i+1}$ for all $1 \le i < k$ and $f(w) =w$.
Because $\fai(w) = w$, it follows that $\fai(e_i) = e_i$ for all $1 \le i <k$.
Take an arbitrary $e_0 \in \liml \fgE$ and a unique $e_k \in \limf \fgE$
  such that $r(e_0) = v_1$ and $s(e_k) = v_k$.
Then, for all $n > 0$, there exist copies $e_{n,i}$ of $e_i$ for all
 $0 \le i \le k$.
It follows that
 $\fai_{n+1}(e_{n+1,i}) = e_{n,i}$ for all $1 \le i \le k-1$,
 $\fai_{n+1}(e_{n+1,0})(\max) = e_{n,0}$, and
 $\fai_{n+1}(e_{n+1,k})(\min) = e_{n,k}$.
The sequence $(e_{n+1,0},e_{n+1,1},\dotsc,e_{n+1,k})$ is realized in
 some $\barx[n] \in \barX_n$.
As for $k=1$,
 the surjectivity of $\iota_{\infty}$ implies that all sub-blocks of 
 $\barx[1]$ are obtained in $\sL(\sigma)$, as desired.
\end{proof}

\begin{prop}\label{prop:iota-infty-to-be-bijective}
Let $\fai : \fgG \to \fgG$ be a straight flexible self-cover.
The map $\iota_{\infty}$ is bijective if and only if
 all the constant sequences are overlapped.
\end{prop}
\begin{proof}
By \cref{lem:natural-injection-from-substitution-system-to-X1,nota:iota-infty,lem:surjective-implies-overlap},
 we only need to show that $\iota_{\infty}$ is surjective if every constant
 sequence $(v_1,v_2,\dotsc,v_k)$ is overlapped.
Thus, suppose that every constant
 sequence $(v_1,v_2,\dotsc,v_k)$ is overlapped.
To show that $\iota_{\infty}$ is surjective,
 take an arbitrary $\barx[1] \in \barX_1$ and $s < t$.
We need to show that there exists some $e \in \fgE$ and $n > 0$ such that
 $\barx[1][s,t]$ is a sub-walk of $\fai^n(e)$.
We consider an array system $\barx$.
We say that there exists an infinite cut at position $i$
 (i.e. just before position $i$)
 if, for all $n \bni$,
 there exists an $n$-cut at position $i$.  
Suppose that there exists no infinite cut in $(s,t]$.
Then, there exists some $n > 0$ such that there is no $n$-cut
 in $(s,t]$.
We find that $(e_n,e_n,\dotsc,e_n) = \barx[n][s,t]$ for some $e_n \in \fgE_n$.
Let $e_n$ be a copy of $e \in \fgE$.
Then,
 it is evident that $\fai^{n-1}(e)$ contains $\barx[1][s,t]$ as a sub-block.
Suppose that there exists only one infinite cut in $(s,t]$.
Let $i_0 \in (s,t]$ be the position of the infinite cut.
There exists some $n > 0$ such that there is no $n$-cut
 in $(s,i_0-1] \cup [i_0 +1,t]$.
We find that 
 $(e_{n,1},e_{n,1},\dotsc,e_{n,1} \underset{\text{cut}}{,}
 \underset{i_0}{e_{n,2}},e_{n,2},\dotsc,e_{n,2})
 = \barx[n][s,t]$ for some $e_{n,1},e_{n,2} \in \fgE_n$.
Because $\fgE_n = \fgE$, we have that $e_{n,1}, e_{n,2} \in \fgE$.
For every $m$ with $m > n$,
 it follows that $\fai(e_{m,1})(\max) = e_{m-1,1}$ and
 $\fai(e_{m,2})(\min) = e_{m-1,2}$.
Thus, by the straightness of the flexible self-cover,
 there exist $e_1,e_2 \in \fgE$ such 
 that the $e_{m,1}$ are copies of $e_1 \in \liml \fgE$ and
 the $e_{m,2}$ are copies of $e_2 \in \limf \fgE$.
Further, we have that $v = r(e_1)= s(e_2) \in \lim \fgV$ and the sequence $(v)$
 of length $1$ is a constant sequence.
By the overlapping property, we find that there exist some $e \in \fgE$
 and $n' > 0$ such that $\fai^{n'}(e)$ contains $e_1 e_2$.
Now, $\fai^{n'+n-1}(e)$ contains $\barx[1][s,t]$.
Finally, suppose that there exist at least $2$ positions of infinite cuts
 in $(s,t]$.
Let $i_1$ be the first position
 and $i_2$ be the last position of the infinite cut in $(s,t]$.
Let $n > 1$ be such that, for all $m \ge n$, the positions of the $m$-cuts in $(s,t]$
 are equal to the positions of the infinite cuts in $(s,t]$.
Then, we find that there exist $e_{n,1}, e_{n,2} \in \fgE_n$ such that
 $(e_{n,1},e_{n,1},\dotsc,e_{n,1},
 \text{(first cut)}\ \barx[n][i_1,i_2-1],\text{(last cut)}\ 
 e_{n,2},e_{n,2},\dotsc,e_{n,2}) = \barx[n][s,t]$.
From $\barx[m][i_1,i_2-1]$ with $m \ge n$,
 we obtain a walk that is denoted as $w_m$.
Thus, we have $\fai_{m,m'}(w_m) = w_{m'}$ for all $m > m' \ge n$ and 
 $\fai_{m,1}(w_m) = \barx[1][i_1,i_2-1]$ for all $m \ge n$.
From this,
 and using the assumption that the $m$-cuts do not change for $m \ge n$,
 each edge that consists of $w_m$ is mapped to a single edge of $w_{m'}$
 for $m > m' \ge n$.
By the straightness of the flexible self-cover,
 all $w_m$ $(m \ge n)$ are identical 
 copies of $w$ for some fixed walk $w$ on $\fgG$;
 in particular, we have that $\fai(w) = w$.
Let $e_{n,i}\ (i = 1,2)$ be copies of $e_i\ (i = 1,2)$.
Then, by the straightness of the flexible self-cover, it follows that
 $e_1 \in \liml \fgE$ and $e_2 \in \limf \fgE$.
Thus, by the overlapping property, there exist some $e \in \fgE$ and $n' > 0$
 such that $\fai^{n'}(e)$ contains a sub-walk $e_1\ w\ e_2$.
It follows that $\fai^{n'+n-1}(e)$ has a sub-walk $\barx[1][s,t]$,
 as desired.
\end{proof}

\subsection{Examples}\label{subsec:examples}\label{subsec:examples}

In this subsection,
 we present two substitution subshifts that can be expressed
 using stationary flexible covering models.

The Fibonacci sequence is formed by the substitution
 $\tau(a) = ab$ and $\tau(b) = a$
 (cf. for example
 \cite[Definition 2.6.1.]
 {Fogg_2002SubstitutionsInDynamicsArithmeticsAndCombinatorics}).
To describe this by the flexible self-cover $\fai : \fgG \to \fgG$,
 let $\fgV := \seb v \sen$, $\fgE := \seb e_a,e_b \sen$,
 and $\phi : \fgG \to \fgG$ be such that
 $\phi(e_a) = e_a e_b$,  $\phi(e_b) = e_a$.
It is easy to check that $\phi$ is \pdirectional.
However, this is not straight.
We consider $\phi^2(e_a) = \phi(e_a e_b) = e_a e_b e_a$ and
 $\phi^2(e_b) = \phi(e_a) = e_a e_b$, and
 define $\fai := \phi^2$.
We then have $\fai(e_a) = e_a e_b e_a$ and $\fai(e_b) = e_a e_b$
 (see \cref{fig:Fibonacci-self-cover}).
It is easy to check that $\fai : \fgG \to \fgG$ is straight.
It is evident that $\lim \fgV = \fgV = \seb v \sen$,
 $\liml \fgE = \seb e_a, e_b \sen$, and $\limf \fgE = \seb e_a \sen$.
Because $\fai^n(e_x) \ne e_x$ for $n > 0$ and $x = a,b$,
 the only constant sequence is $(v)$.
To apply \cref{prop:iota-infty-to-be-bijective},
 we have to check that both $e_a e_a$ and $e_b e_a$ appear in $\fai^n(e)$
 for some $e \in \fgE$ and $n > 0$.
Note that $e_b e_a$ has already appeared in $\fai(e_a)$.
We compute 
 $\fai^2(e_a) = \fai(e_a e_b e_a) = e_a e_b\ e_a e_a\ e_b e_a e_b e_a$.
Thus, $e_a e_a$ has also appeared.
The related substitution read is $\sigma(a) = a b a$, $\sigma(b) = a b$.
Thus, we have shown that
 $\iota_{\infty} : (X_{\sigma},T) \to (\barX_1,\barT)$
 is an isomorphism.
Because the powers $\fai^i$ are not bidirectional, 
 for the stationary flexible graph covering model
 $\fgGcal : \covrepa{\fgG}{\fai}$
 generated by $\fai$ and the sequence $n(e_a) = n(e_b) = 1$,
 the inverse limit $\invlim \fgGcal$, which we denote by $(X,f)$,
 is not a homeomorphism.
Thus, $(X,f)$ is topologically conjugate to neither
 $(\barX,\barf)$ nor $(\barX_1,\barT)$, though the latter is topologically 
 conjugate to $(X_{\sigma},T)$
 by \cref{prop:iota-infty-to-be-bijective}.
We do not investigate further, e.g. the coincidence of $(\barX,\barf)$
 with $(\barX_1,\barT)$ or $(X_{\sigma},T)$ is not studied here.
For topological conjugacy between substitution subshifts
 and stationary Bratteli diagrams,
 we refer readers to
\cite{BEZUGLYI_2009AperioSubstSysBraDiag,DURAND_1999SubstDynSysBratteliDiagDimGroup,Yassawi10Bratteli-VershikRepresentationsOfSomeOne-SidedSubstitutionSubshifts}.

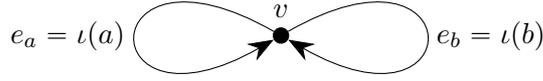
\begin{figure}
\centering
\begin{tikzpicture}[scale=0.4]
\GraphInit[vstyle=Classic]
\tikzset{VertexStyle/.append style={minimum size=3pt}}
\Vertex[x=0, y=0,Lpos=90,L=$v$]{Center};
\path[arrows = {-Stealth[scale=2]}] (Center) edge [out = 150, in = -150,
 distance=7cm] node[pos=0.5,left]{$e_a = \iota(a)$} (Center);
\path[arrows = {-Stealth[scale=2]}] (Center) edge [out = 30, in = -30,
 distance=7cm] node[pos=0.5,right]{$e_b = \iota(b)$} (Center);
\node[draw=none] at (12,3) {$\sigma(a) = aba,\ \sigma(b) = ab$};
\end{tikzpicture}
\vspace{0mm}
\normalsize
\caption{Self-cover for the Fibonacci sequence.}
\label{fig:Fibonacci-self-cover}
\end{figure}

The next example is a straight flexible self-cover $\fai : \fgG \to \fgG$
 that constructs
 a substitution system with two fixed points
(see \cref{fig:self-cover-with-fixed-points}).
The vertices of the graph $\fgG = (\fgV,\fgE)$ include the left vertex $v_l$,
 middle vertex $v_m$, and right vertex $v_r$, i.e. 
 $\fgV = \seb v_l, v_m, v_r \sen$.
There exist six edges: $\fgE = \seb e_a, e_b, e_c, e_d, e_e, e_f \sen$.
The substitution read on $\fai$
 is written as in \cref{fig:self-cover-with-fixed-points}.
Thus, $\fai$ is defined as $\fai(e_a) = e_a,\ \fai(e_b) = e_a e_b e_d e_e,\
 \fai(e_c) = e_d e_e e_c e_a,\ \fai(e_d) = e_d e_e e_d e_f,
 \ \fai(e_e) = e_f e_e,
 \myand \fai(e_f) = e_f$.
It is easy to check the \bidirectionalitys condition
 and the straightness condition, in which
 $\lim \fgV = \seb v_r, v_m, v_l \sen$,
 $\lim_l \fgE = \seb e_a, e_e, e_f \sen$, and
 $\lim_f \fgE = \seb e_a, e_d, e_f \sen$.
Because of the \bidirectionalitys condition,
 if we construct the stationary flexible covering model
 $\fgGcal : \covrepa{\fgG}{\fai}$ and its inverse limit
 $\invlim \fgGcal = (X,f)$,
 then $(X,f)$ is isomorphic to the natural extension $(\hX,\hf)$
 and to the array system $(\barX,\barf)$.
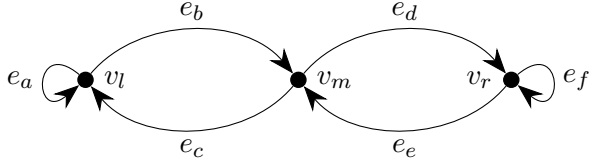
\begin{figure}
\centering
\begin{tikzpicture}[scale=0.4]
\GraphInit[vstyle=Classic]
\tikzset{VertexStyle/.append style={minimum size=3pt}}
\Vertex[x=-7, y=0,Lpos=0,L=$v_l$]{Left};
\Vertex[x=0,  y=0,Lpos=0,L=$v_m$]{Mid};
\Vertex[x=7,  y=0,Lpos=180,L=$v_r$]{Right};
\path[arrows = {-Stealth[scale=2]}] (Left) edge [out = 140, in = -140,
 distance=2cm] node[pos=0.5,left]{$e_a$} (Left);
\path[arrows = {-Stealth[scale=2]}] (Left) edge [out = 50, in = 130,
 distance=2.5cm] node[pos=0.5,above]{$e_b$} (Mid);
\path[arrows = {-Stealth[scale=2]}] (Mid) edge [out = -130, in = -50,
 distance=2.5cm] node[pos=0.5,below]{$e_c$} (Left);
\path[arrows = {-Stealth[scale=2]}] (Mid) edge [out = 50, in = 130,
 distance=2.5cm] node[pos=0.5,above]{$e_d$} (Right);
\path[arrows = {-Stealth[scale=2]}] (Right) edge [out = -130, in = -50,
 distance=2.5cm] node[pos=0.5,below]{$e_e$} (Mid);
\path[arrows = {-Stealth[scale=2]}] (Right) edge [out = 40, in = -40,
 distance=2cm] node[pos=0.5,right]{$e_f$} (Right);
\node   [draw=none] (A) at (12,6)
 {$\sigma(a) = a,\ \sigma(b) = abde\ \sigma(c) = deca$,};
\node [draw=none,below=of A.west,anchor=west]
 {$\sigma(d) = dedf,\ \sigma(e)=fe,\ \sigma(f) = f$};
\end{tikzpicture}
\vspace{0mm}
\normalsize
\caption{Self-cover with fixed points.}\label{fig:self-cover-with-fixed-points}
\end{figure}

The edges $e \in \fgE$ that satisfy $\fai(e) = e$ are $e_a,e_f$.
It follows that the constant sequences are only
 $(v_l,v_1,\dotsc,v_1)$ with length $\ge 1$,
 $(v_m)$,
 and $(v_r,v_r,\dotsc,v_r)$ with length $\ge 1$.
To check the overlapping property, we only need to consider
 the following walks: $\underbrace{e_a e_a \dotsb e_a}_{\ge 2}$
 for $(v_l,v_l,\dotsc,v_l)$ with length $\ge 1$,
 $e_e e_d$ for $(v_m)$, and $\underbrace{e_f e_f \dotsb e_f}_{\ge 2}$
 for $(v_r, v_r, \dotsc, v_r)$ with length $\ge 1$.
To check $e_a e_a \dotsb e_a$, we compute $\fai^2(e_b) = e_a e_a e_b \dotsb$,
 $\fai^4(e_b) = e_a e_a \fai^2(e_b) \dotsb = e_a e_a e_a e_a e_b \dotsb$,
 and so on.
To check $e_e e_d$, we can see that $\fai(e_d) = e_d\ \ e_e e_d\ \ e_f$.
To check $e_f e_f \dotsb e_f$,
 we compute $\fai^n(e_e) = e_f e_f \dotsb e_f e_e$.
Thus, this straight flexible self-cover satisfies the overlapping property.
Therefore, $(X_{\sigma},T)$ is
 isomorphic to $(\barX_1,\barT)$.
From the calculation of $\sigma^n(a)$ and $\sigma^n(f)$,
 we can easily see that $(X_{\sigma},T)$ has fixed points
 $(\dotsc,a,a,a,\dotsc)$ and $(\dotsc,f,f,f,\dotsc)$.
If we construct $\wgGcal : \covrepa{\wgG}{\fai}$ from $\fgGcal$,
 then the copies $e_{n,b}, e_{n,c}, e_{n,d}, e_{n,e}$ on $\wgG_n$
 of $e_b, e_c, e_d, e_e$ for $n > 0$ have lengths
 $\to \infty$ for $n \to \infty$.

For this example,
 it is possible to conclude that $(\barX,\barf)$ is isomorphic
 to $(\barX_1,\barT)$, as shown by the next lemma.
Thus, we can conclude that $\invlim \fgGcal = (X,f)$ is isomorphic to
 $(X_{\sigma},T)$.
\begin{lem}\label{lem:example2-isomorphic-to-barX}
According to the flexible self-cover
 in \cref{fig:self-cover-with-fixed-points},
 it follows that $(X_{\sigma},T)$ is isomorphic to $(\barX,\barf)$.
\end{lem}
\begin{proof}
We establish the proof by showing that, for all $n \bpi$,
 from any $\barx[n] \in \barX_n$, we can recover
 $\barx[n+1]$ in a unique way.
Because the method we use here is the same for all $n \bpi$,
 we only show the case in which $n = 1$.
Let $x_1 = \barx[1]$.
First, suppose that only $e_a$ appears in $x_1$.
Then, clearly, we can find a unique $\barx \in \barX$.
The same holds in the case that only $e_f$ appears in $x_1$.
We then exclude these two cases.
Second, suppose that $e_b$ appears in $x_1$ at position $i$.
Then, it follows that $x_1[i-1,i+2] = e_a e_b e_d e_e$,
and so $\barx[2][i-1,i+2] = e_{2,b} e_{2,b} e_{2,b} e_{2,b}$.
Again, we obtain an occurrence of a copy of $e_b$ in $\barx[2]$ in $[i-1,i+2]$.
Thus, we can clarify all the $\barx[2][i-1-1,i+2 + 4 + 2] = \barx[2][i-2,i+8]$.
Further, we have that $\barx[3][i-2,i+8] = e_{3,b} \dotsb e_{3,b}$.
This clarifies all the $x_1[i-2,i+8]$.
In this way, we can identify all of $x_1$,
 and all of $\barx[2], \barx[3], \dotsc$, thus all of $\barx$.
Third, suppose that $e_c$ appears in $x_1$.
Because it must be in the form $e_d e_e e_c e_a$, with an identical argument,
 we can identify the unique $\barx \in \barX$.
Thus, we assume that $x_1$ does not contain $e_b$ or $e_c$.
In this case, $e_a$ does not appear either.
Thus, we need to consider the case in which
 only $e_d$, $e_e$, and $e_f$ appear.
In this case, in $\barx[2]$, only $e_{2,d}$, $e_{2,e}$, and $e_{2,f}$ appear.
In $x_1$, the edge $e_e$ appears in the form $e_f e_e$ or $e_d e_e e_d e_f$.
This form can easily be identified by looking at the left-hand side of $e_e$.
To decode this, we first decode all the $e_d e_e e_d e_f$,
 to get $e_d e_d e_d e_d$ at the same position in $\barx[2]$.
All the other occurrences of $e_e$ remain in the form $e_f e_e$.
Thus, these occurrences are decoded in $\barx[2]$ as $e_e e_e$,
which means that we could have decoded into $\barx[2]$ uniquely at
 the positions where $e_e$ appears.
Now, it is easy to see that none of the $e_d$ remains.
The remaining $e_f$ in $\barx[1]$ is decoded as $e_f$ in $\barx[2]$ at the 
 same position.
Thus, from $x_1$, we could have recovered $\barx[2]$ uniquely, as desired.
\end{proof}

\subsection{Construction of transitive substitution subshifts}
\label{subsec:transitive-substitution}
In this subsection, we discuss a way of constructing some class of transitive 
 substitution subshifts.
Firstly, we show that
 if there exist a vertex $v \in \lim \fgV$ and an edge
 $e \in s^{-1}(v) \cap \lim_f \fgE$
 such that $\fai(e)$ passes every element of $\fgE$ at least once,
 then $(X,f) = \lim \fgGcal$ has a positively transitive point.
To see this, let $v_n \in \wgV_n$ ($n \bpi$) be the copies of $v$.
Then, the point $p := (v_0, v_1, v_2, \dotsc)$ is the required
 positively transitive point.
It is easy to see that if $(X,f)$ is positively topologically transitive,
 then the natural extension is also positively topologically transitive.
Let $\fai' : \fgG' \to \fgG'$ and $\fai'' : \fgG'' \to \fgG''$
 be straight flexible self-covers
 with the substitution reads $\sigma : A' \to A'^+$
 and $\sigma'' : A'' \to {A}''^+$ on them.
We consider the case in which the injections
 $\iota'_{\infty} : (X'_{\sigma},T) \to (\barX'_1,\barT)$
 and ${\iota''}_{\infty} : ({X''}_{\sigma},T) \to ({\barX}''_1,\barT)$
 are both bijective.
Take and fix vertices $u' \in \lim \fgV'$ and $u'' \in \lim \fgV''$
 and edges $e'_1 \in s^{-1}(u') \cap \lim_f \fgE'$ and
 $e''_1 \in s^{-1}(u'') \cap \lim_f \fgE''$.
We assume that there exist finite walks $w' = e'_1 \dotsc$
 and $w'' = e''_1 \dotsc$ that satisfy $r(w') = u'$ and $r(w'') = u''$;
 $w'$ passes all of $\fgE'$ and $w''$ passes all of $\fgE''$.
We make a new flexible self-cover $\fai : \fgG \to \fgG$ that contains
 the above two.
We take $\fgV = \fgV' \cup \fgV''$ as a disjoint union
and make new edges $e'$ from $u'$ to $u''$
 and $e''$ from $u''$ to $u'$.
We take $\fgE = \seb e', e''\sen \cup \fgE' \cup \fgE''$ as a disjoint union.
The covering maps restricted on $\fgE'$ and $\fgE''$ coincide with
 $\fai'$ and $\fai''$, respectively.
We define $\fai(e') := w' e' e'' e' w''$
 and $\fai(e'') := w'' e'' e' e'' w'$.
It is easy to check that $\fai$ is a self-cover.
Evidently, both $\fai(e'), \fai(e'')$ pass all $\fgE$.
Finally, we consider the self-cover $\fai^2$ instead of $\fai$ itself.
Then, it is straightforward to check that
 $\fai^2$ is straight
 and $\lim \fgV = \lim \fgV' \cup \lim \fgV''$.
It is easy to check that no new constant sequences arise.
Therefore,
 we have a bijection $\iota_{\infty} : (X_{\sigma},T) \to (\barX_1,\barT)$.
Because $(\barX,\barf)$ is transitive
 and the factor of a transitive dynamical system is transitive,
 $(\barX_1,\barT)$ and $(X_{\sigma},T)$ are transitive, as desired.

\vspace{3mm}

Finally, at the end of this paper, we must note that
 the satisfactory regularity condition on the Bratteli--Vershik models
 may not be the $\bl$-periodicity-regulation property.
It may be the case that the Bratteli--Vershikizability plays an important role in this.
Many substitutions
 that do not arise from flexible self-covers
 may exist (see \cref{defn:stationary-flexible-graph-covering}). 
We could not clarify the substitution subshifts
 that are canonically identical to $(\barX_1,\barT)$ using
 a flexible self-cover.
The topological rank that was determined for invertible
 Cantor minimal systems by Downarowicz and Maass
 in \cite{DOWNAROWICZ_2008FiniteRankBratteliVershikDiagAreExpansive}
 can be extended to all invertible zero-dimensional systems.
In this regard, we have yet to answer the following question:

\noindent \textbf{Question.} Does there exist an invertible minimal Cantor
 system whose topological ranks differ?

\vspace{5mm}

\noindent
\textsc{Acknowledgments:}
The author would like to thank the anonymous referee(s) of our previous
 submissions for their insightful advice,
on the basis of which this paper has improved considerably in comparison with the first version.
We also would like to extend our gratitude to
 Fumiaki~Sugisaki and Masamichi~Yoshida for valuable discussions
 during the revision process.
Finally, we would like to thank Editage (www.editage.jp)
 for providing English-language editing services
 during the total revision processes.
This work was partially supported by JSPS KAKENHI
(Grant Number 16K05185).

\providecommand{\bysame}{\leavevmode\hbox to3em{\hrulefill}\thinspace}
\providecommand{\MR}{\relax\ifhmode\unskip\space\fi MR }
\providecommand{\MRhref}[2]{%
  \href{http://www.ams.org/mathscinet-getitem?mr=#1}{#2}
}
\providecommand{\href}[2]{#2}


\end{document}